\documentclass[12pt]{article}
\usepackage{graphicx}
\graphicspath{ {pictures/} } 

\usepackage[usenames]{color}

\newcommand{\black}{\black}

\usepackage{amsmath,amsfonts,amssymb,mathrsfs}
\usepackage[hidelinks]{hyperref}
\hypersetup{
    colorlinks=true,
    linkcolor=black,
    citecolor=black,
    filecolor=black,
    urlcolor=black,
}
\usepackage{amsthm}
\usepackage{bbm} 
\usepackage{tikz} 
\usepackage{enumitem}
\usepackage{mathtools}
\usepackage{tikzsymbols}
\usepackage{cleveref}
\usepackage[font=small,labelfont=bf]{caption}
\usepackage[left=3cm,right=3cm,top=3cm,bottom=3cm]{geometry}

\theoremstyle{plain}
\newtheorem{theorem}{Theorem}[section]
\newtheorem{assumptions}[theorem] {Assumptions}
\newtheorem{corollary}[theorem]{Corollary}
\newtheorem{definition}[theorem]{Definition}
\newtheorem{example}[theorem]{Example}
\newtheorem{lemma}[theorem]{Lemma}
\newtheorem{proposition}[theorem]{Proposition}

\newtheorem{remark}[theorem]{Remark}

\numberwithin{equation}{section}

\tikzset{ dot/.style={ circle, inner sep=0pt, minimum size=1.5mm,
    fill=black } }
\usetikzlibrary{calc,decorations.markings}
\usetikzlibrary{decorations.pathmorphing}

\makeatletter
\newcommand{\mylabel}[2]{#2\def\@currentlabel{#2}\label{#1}}
\makeatother

\newcommand{\Ccal}{\mathcal{C}}
\newcommand{\Dcal}{\mathcal{D}}

\newcommand{\Fcal}{\mathcal{F}}

\newcommand{\Lcal}{\mathcal{L}}
\newcommand{\NN}{\mathbb{N}}

\newcommand{\RR}{\mathbb{R}}
\newcommand{\ZZ}{\mathbb{Z}}
\newcommand{\1}{\mathbbm{1}}
\newcommand{\indic}[1]{\1_{\{#1\}}}

\newcommand{\mc}[1]{\mathcal{#1}}
\newcommand{\V}[1]{{\boldsymbol #1 }}
\newcommand{\sss}[1]{\scriptscriptstyle{#1}}

\newcommand{\co}[1]{\left[#1\right )}
\newcommand{\oc}[1]{\left(#1\right ]}
\newcommand{\ob}[1]{\left(#1\right )}
\newcommand{\cb}[1]{\left[#1\right ]} 
\newcommand{\abs}[1]{\left\vert#1\right\vert} 
\newcommand{\norm}[1]{\|#1\|} 
\newcommand{\mnorm}[1]{{\left\vert\kern-0.25ex\left\vert\kern-0.25ex\left\vert #1 
    \right\vert\kern-0.25ex\right\vert\kern-0.25ex\right\vert}}
\newcommand{\pair}[1]{\left\langle#1\right\rangle}

\newcommand{\Pa}{P_a} 
\newcommand{\Ea}{E_a} 
 
\newcommand{\EEa}{E}

\newcommand{\Pprod}{\mathbb{P}} 
\newcommand{\Eprod}{\mathbb{E}}

\newcommand{\pf}{\cpl |\varphi|^4}

\newcommand{\ra}{\rightarrow}
\newcommand{\tb}{\V{t}} 
 
\newcommand{\varphib}{\V{\varphi}} 
 
\newcommand{\nn}{\nonumber}

\newcommand{\sssL}{\sss(\Lambda)}
\newcommand{\sssI}{\sss(\infty)}
\newcommand{\ctwsaw}{lattice Edwards model\ } 
\newcommand{\ctwsawn}{lattice Edwards model}  
\newcommand{\edwards}{Edwards model}
\newcommand{\Ledwards}{(lattice) Edwards model}

\newcommand{\bydef}{\coloneqq}

\newcommand{\Z}[1]{Z_{\V{\tau}_{#1}}}
 
\newcommand{\rrptn}{Y}     
\newcommand{\rptn}{\mathcal{Y}}
\newcommand{\rptnL}{\rptn^{\sssL}}

\newcommand{\lgrptn}{\log \rptn}
\newcommand{\brptn}{\bar{\rptn}}
\newcommand{\vertex}[1]{r_{#1}}
\newcommand{\barvert}[1]{\bar{r}_{#1}}
\newcommand{\vertexL}[1]{r^{\sssL}_{#1}}
\newcommand{\barvertL}[1]{\bar{r}^{\sssL}_{#1}}
\newcommand{\ZI}{Z^{\sssL}}

\newcommand{\cpl}{g} 

\newcommand{\cv}{\cpl} 
\newcommand{\cirb}{K} 

\newcommand{\laces}{\Lcal} 
\newcommand{\lace}{L} 
\newcommand{\weight}{w}
\newcommand{\filter}{\mathcal{I}} 
\newcommand{\ngn}{\Gamma} 
\newcommand{\xb}{\V{x}}
\newcommand{\tp}{P}
\newcommand{\Pweight}{\tp}
\newcommand{\ti}{J}
\newcommand{\selfloop}[1]{L_{#1}}
\newcommand{\decon}{D}
\newcommand{\PiL}{\Pi^{\sssL}}
\newcommand{\PiI}{\Pi^{\sssI}}
\newcommand{\aPi}{\Psi}
\newcommand{\aPiL}{\aPi^{\sssL}}
\newcommand{\aPiI}{\aPi^{\sssI}}
\newcommand{\weightL}{\weight^{\sssL}}
\newcommand{\selfloopL}[1]{L^{\sssL}_{#1}}
\newcommand{\selfloopI}[1]{L^{\sssI}_{#1}}

\newcommand{\www}{w}
\newcommand{\parto}{\partial_1}
\newcommand{\partt}{\partial_2}

\newcommand{\G}[1]{G_{\V{\tau}_{#1}}}
\newcommand{\greens}{S}
\newcommand{\Jump}{J}
\newcommand{\hJ}{\hat{\Jump}}
\newcommand{\LX}{X^{\sssL}}
\newcommand{\ZX}{X^{\sssI}} 
\newcommand{\Ltau}{\tau^{\sssL}}

\newcommand{\TL}{T^{\sssL}} 
\newcommand{\dgreens}{\tilde{\greens}}
\newcommand{\DL}{\Delta^{\sssL}}
\newcommand{\DI}{\Delta^{\sssI}} 
\newcommand{\GL}{\greens^{\sssL}}
\newcommand{\GI}{\greens}

\newcommand{\GrL}{G^{\sssL}}
\newcommand{\GrI}{G^{\sssI}}  
\newcommand{\jumprate}{\Jump_+}
\newcommand{\HaraJ}{Q}

\usepackage[numbers]{natbib}

\usepackage{enumitem}

\title{The continuous-time  lace expansion}
\author{David C. Brydges\thanks{University of British Columbia,
    Department of Mathematics.} 
  \and
  Tyler Helmuth\thanks{Durham University, School of
    Mathematical Sciences} 
  \and Mark
    Holmes\thanks{University of Melbourne, School of Mathematics and
      Statistics.} }
\date{September 1, 2021}

\begin{document}                   
\maketitle   

\begin{abstract}
  We derive a continuous-time lace expansion for a broad class of
  self-interacting continuous-time random walks. Our expansion applies
  when the self-interaction is a sufficiently nice function of the
  local time of a continuous-time random walk.  As a special case we
  obtain a continuous-time lace expansion for a class of spin systems
  that admit continuous-time random walk representations.

  We apply our lace expansion to the $n$-component $\pf$ model on
  $\ZZ^{d}$ when $n=1,2$, and prove that the critical Green's function
  $G_{\nu_{c}}(x)$ is asymptotically a multiple of $\abs{x}^{2-d}$
  when $d\geq 5$ at weak coupling. As another application of our method we establish
  the analogous result for the \ctwsaw at weak coupling.
\end{abstract}

\section{Introduction}
\label{sec:intro}

Many lattice spin systems are expected to exhibit mean-field behaviour
on $\ZZ^{d}$ when $d>d_{c}=4$. Results of this type have been proven
by making use of \emph{random walk
  representations}~\cite{Aiz81,Aiz82,Frohlich,FFS92} and recently
\cite{AD2020} these methods have been extended to $d=4$ by taking into
account logarithmic corrections to mean field theory.  In this paper
we extend this program by analyzing the critical Green's function of
$n$-component lattice spin models for $n=1,2$ at weak coupling for
$d>4$. We use the random walk representation originating in the work
of Symanzik~\cite{Symanzik69} and developed in~\cite{BFS82,Dyn83}. For
recent developments regarding this representation
see~\cite{1903.04045,Sz11}, and for alternative random walk
representations see~\cite{Aiz82,HS94b}.

To be more precise, we determine the asymptotics of the infinite
volume critical two-point function
$\pair{\varphi_{a}\cdot\varphi_{b}}$. Here $\pair{\cdot}$ denotes the
expectation of an $O(n)$-invariant $\pf$ spin model; the spins
$\varphi$ take values in $\RR^{n}$. The definitions of these models
and what it means to be critical are given in
\Cref{sec:greens-fn}. Let $\abs{x}$ and $\abs{\varphi}$ denote the
Euclidean norms of $x\in \ZZ^{d}$ and $\varphi \in \RR^{n}$.

\begin{theorem}
  \label{thm:main-intro}
  Let $d>d_{c}=4$ and $n\in\{1,2\}$. Let $\pair{\cdot}$ denote
  expectation with respect to the critical $n$-component $\pf$
  model. For $\cpl>0$ sufficiently small there is a constant $C
  >0$ such that
  \begin{equation}
    \label{eq:Intro-thm}
    \pair{\varphi_{a}\cdot \varphi_{b}} \sim \frac{C}{|b-a|^{d-2}}, 
  \qquad \textrm{ as $|b-a| \ra \infty$}.
  \end{equation}
\end{theorem}
The relation $\sim$ in~\eqref{eq:Intro-thm} means the ratio of the
left-hand and right-hand sides tends to one in the designated
limit. Our theorem exhibits mean-field behaviour in the sense that the
exponent $d-2$ in~\eqref{eq:Intro-thm} is the exponent predicted by
Landau's extension of mean-field theory
\cite[Chapter~2]{Kopietz-et-al2010}.  The right hand side of
\eqref{eq:Intro-thm} is Euclidean invariant, so for weak coupling the
conclusion strengthens the triviality results \cite{Aiz81, Frohlich}
by showing that the scaling limit of the two-point function of this
model is Euclidean invariant and equals the massless free field
two-point function. When $n=1$ \Cref{thm:main-intro} was already
proven by Sakai~\cite{Sakai2015}. The case $n=2$ is new.  For
$d=d_{c}=4$ the asymptotics in \eqref{eq:Intro-thm} have been
established by a rigorous renormalization group technique for the
$n$-component $\pf$ model for all $n\in\NN$~\cite{SladeTomberg}.  For
$n=1$ and $d=4$ the analysis of the Green's function by rigorous
renormalization group techniques began with \cite[(8.32)]{GK86} and
\cite[Theorem I.2]{FMRS87}.

Sakai's proof of the $n=1$ case of \Cref{thm:main-intro} made use of
the {\em lace expansion}, a technique originally introduced to prove
mean-field behaviour for discrete-time weakly self-avoiding
walk~\cite{BrSp85a}.  The lace expansion has since been reformulated
in many different settings: unoriented and oriented percolation
\cite{HS90perc,NY91,HS03op}, the contact process \cite{HofSak04},
lattice trees and animals \cite{HS90trees}, Ising and $\pf$ models
\cite{Sakai07,Sakai2015}, the random connection model~\cite{RCMlace},
and various self-interacting random walk models
\cite{HH12interacting,HammondHelmuthASAW,Helmuth,Ueltschi}.  Within
these settings the lace expansion has been applied to a variety of
problems, ranging from proofs of weak convergence on path space for
branching particle systems \cite{DS98,HHP, Hol16} to proofs of
monotonicity properties of self-interacting random
walks~\cite{HH10,H12,HS12}. In each case the expansion is based on a
discrete parameter that plays the role of time.

To prove \Cref{thm:main-intro} we introduce a lace expansion in
continuous time. Our methods naturally apply to a broader class of
problems than $\pf$ models, and to illustrate this we also analyze the
\ctwsawn. A precise formulation of our main results is given in
\Cref{sec:greens-fn}, after the introduction of the basic objects of
our paper.

\section{Random walk and local times} 
\label{sec:random-walk}

To fix notation and assumptions, we define continuous-time random walk
started at a point $a$ in $\ZZ^{d}$ and killed outside of a finite
subset $\Lambda$ of $\ZZ^{d}$.  These stochastic processes are central
to the rest of the paper.

\subsection{Infinite volume}
\label{sec:ctrw-inf}

We begin by defining the class of jump distributions that we will
allow.  Recall that a one-to-one map $T$ from the vertex set of
$\ZZ^{d}$ onto itself such that edges $\{x,y\}$ of $\ZZ^{d}$ are
mapped to edges $\{Tx,Ty\}$ of $\ZZ^{d}$ is called an
\emph{automorphism}. Let $\text{Aut}_{0}(\ZZ^{d})$ denote the subgroup
of automorphisms that fix the origin $0$. For example, reflections in
lattice planes containing the origin are in
$\text{Aut}_{0}(\ZZ^{d})$. An automorphism in
$\text{Aut}_{0}(\ZZ^{d})$ permutes the nearest neighbors in $\ZZ^{d}$
of the origin and this permutation determines the automorphism.  A
function $f$ on $\ZZ^{d}$ is \emph{$\ZZ^{d}$-symmetric} if
$f(Tx)=f(x)$ for $x\in\ZZ^{d}$ and $T\in \text{Aut}_{0}(\ZZ^{d})$.  A
function $f(x,y)$ of two variables is \emph{$\ZZ^{d}$-symmetric} if
$f (Tx,Ty)=f (x,y)$ for all automorphisms $T$ of $\ZZ^{d}$. The
condition of being $\ZZ^{d}$-symmetric includes translation invariance
($f(x,y)=f(0,y-x)$) and hence is equivalent to the function
$g(x):=f(0,x)$ of one variable being $\ZZ^{d}$-symmetric.
\begin{assumptions}
  \label{hyp:Jnew}
  Assume $\Jump\colon\ZZ^d\ra \RR$ satisfies
  \begin{enumerate}
  \item[\mylabel{as:J1}{(J1)}] $\Jump(x)\geq 0$ for $x\neq 0$, and
    $\Jump(0)\bydef -\sum_{x\neq 0}\Jump(x)$ is finite,
  \item[\mylabel{as:J2}{(J2)}] the set
    $\{x\in\ZZ^{d}\mid \Jump(x)>0\}$ is a generating set for
    $\ZZ^{d}$,
  \item[\mylabel{as:J3}{(J3)}] $\Jump$ is $\ZZ^{d}$-symmetric, 
  \item[\mylabel{as:J4n}{(J4)}] $\Jump$ has finite range $R>0$, i.e.,
    $\Jump(x)=0$ if $|x|\geq R$.
  \end{enumerate}
\end{assumptions}
A condition that $\Jump$ decays like $|x|^{-3(d-2)}$ might serve
instead of \ref{as:J4n}; to avoid having to extend standard cited
results such as ~\cite{SimonIneq,Lieb,Rivasseau} we work with a fixed
choice of $\Jump$ satisfying \ref{as:J1}--\ref{as:J4n}.  Let
$\DI\colon\ZZ^{d}\times\ZZ^{d}\ra \RR$ be the infinite matrix with
entries
\begin{equation}
    \label{e:Delta-infty-matrix-def}
    \DI_{x,y}\bydef\Jump(y-x) .
\end{equation}

By~\ref{as:J3}, $\DI$ is symmetric, and \ref{as:J1} implies that $\DI$
has non-negative off-diagonal elements and that its row sums are all
equal to zero.  This implies $\DI$ is the generator of a
continuous-time random walk $\ZX$ on $\ZZ^{d}$. The assumption
\ref{as:J2} ensures this walk is irreducible. Let
\begin{equation}
  \label{e:hJ-def}
  \hJ \bydef -\DI_{x,x} = -J(0) ,
  \qquad
  \jumprate(y) \bydef \Jump(y) \indic{y\neq 0}.
\end{equation}
By \ref{as:J1}, $\hJ$ is finite. The walk $\ZX$ has a mean $\hJ^{-1}$
exponential holding time at each $x$, and jumps from $x$ to $y\ne x$
with probability $\jumprate(y-x)/\hJ$. We write $\Pa$ for a
probability measure under which $\ZX$ is a continuous-time random walk
on $\ZZ^d$ started at $a\in\ZZ^{d}$, and $\Ea$ for the corresponding
expectation.

\begin{example}
  \label{ex:latlap}
  The most important example is when
  \begin{equation}
    \Jump(x) = \indic{\abs{x}=1} - 2d\indic{x=0}.
  \end{equation}
In this case $\DI$ is called the \emph{lattice Laplacian}, and $\ZX$
is a continuous-time nearest-neighbour random walk on $\ZZ^{d}$.
\end{example}

\subsection{Finite volume}
\label{sec:ctrw-fin}

Let $\Lambda$ be a finite subset of $\ZZ^{d}$ and let $\TL$ be the
first time $\ZX$ exits  $\Lambda$:
\begin{equation}
  \label{eq:TL}
  \TL\bydef \inf\{t\geq 0:\ZX_t\notin \Lambda\}.
\end{equation}
Let $\ast\notin \ZZ^d$ be an additional ``cemetery'' state, and define
$\LX_{t}$ by
\begin{equation}
   \label{e:X-def}
   \LX_{t} 
   = 
   \begin{cases}
   \ZX_{t},& t < \TL,\\
   \ast,& t \ge \TL.
   \end{cases}
\end{equation}
For each finite $\Lambda$ the process $\LX=(\LX_t)_{t\ge 0}$ is a
continuous-time Markov chain on $\Lambda\cup \{\ast\}$ with absorbing
state~$\ast$. Note that this construction defines the processes $\LX$
on the same probability space for all finite $\Lambda\subset\ZZ^{d}$.

The generator $\DL_{\ast}$ of $\LX$ is a
$(|\Lambda|+1)\times (|\Lambda|+1)$ matrix with a row of zeros
since~$\ast$ is an absorbing state.  We will let
$\DL\colon\Lambda \times \Lambda \ra \RR$ denote the matrix obtained
by removing the row and column corresponding to transition rates to
and from~$\ast$, i.e., $(\DL)_{x,y}\bydef(\DI)_{x,y}$ for
$x,y \in \Lambda$.

\subsection{Local time and free Green's functions}
\label{sec:ctrw-loc}
For $x \in \Lambda$ and a Borel set $I \subset [0,\infty)$ the
\emph{local time of $\LX$ at $x$ during $I$} is
\begin{align}
    \label{e:tauI-def}  
    \Ltau_{I,x}
    \bydef    
    \int_{I} d\ell\,\indic{\LX_{\ell}=x} ,
\end{align}
where $d\ell$ is Lebesgue measure. Let
$\V{\tau}^{\sss(\Lambda)}_{I} \bydef (\Ltau_{I,x})_{x\in\Lambda}$
denote the vector of all local times, and let $\Ltau_{x}$ denote
$\Ltau_{[0,\infty),x}$. We will often omit the superscript $\Lambda$
when there is no risk of confusion.  For $a,b\in \Lambda$ we define
the \emph{free Green's function in $\Lambda$} by
\begin{equation}
  \label{e:gabdef}
  \GL (a,b)\bydef \Ea[\Ltau_{b}]
  =
  \int_0^\infty d\ell\, \Pa(\LX_{\ell}=b). 
\end{equation}
Note that $\GL (a,b)<\infty$ since the expected time for $\LX$ to exit
$\Lambda$ is finite. The next lemma, proved in
Appendix~\ref{appendix}, explains why $\GL (a,b)$ is
called the free Green's function.

\begin{lemma}\label{lem:free-green}
$\DL$ is invertible, and for $a,b\in\Lambda$
\begin{equation}
  \label{e:free-green-1}
    \GL (a,b)
    = (- \DL)^{-1}_{a,b}.
\end{equation}
\end{lemma}

There is also an infinite volume version of \Cref{lem:free-green}, and
it is nicer because the infinite volume limit restores translation
invariance. For $d\geq 3$, define $\GI(a,b)$ to be the expected time
spent at $b$ by the random walk $\ZX$ started from $a$, i.e.,
\begin{equation}
\label{e:Sinf} 
    \GI(a,b)\bydef E_a[\tau_b^{\sss(\infty)}],
\end{equation}
where $\tau_b^{\sss(\infty)}$ is defined as in \eqref{e:tauI-def} but
with $\LX_{\ell}$ replaced by $\ZX_{\ell}$. Since
$\TL\uparrow \infty$ as $\Lambda \uparrow \ZZ^d$ a.s., monotone
convergence implies that
\begin{equation}
  \label{e:GsrwIV}
  \GL (a,b)
  =
  E_a[\tau_{[0,\TL],b}^{\sss(\infty)}]
  \uparrow
  E_a[\tau_{[0,\infty),b}^{\sss(\infty)}]
  =
  \GI(a,b)<\infty.
\end{equation}
where the finiteness holds by transience. The translation invariance
of the infinite volume random walk implies
\begin{equation}
  \label{S1def}
  \GI(x)\bydef\GI(a,a+x),
  \qquad
  x,a\in \ZZ^{d}
\end{equation}
is well-defined, i.e., independent of $a\in\ZZ^{d}$. Recall, see
\cite[Theorem 4.3.5]{lawler_limic_2010}, that there is a $C_{\Jump}$
depending on $\Jump$ such that
  \begin{equation}
  \label{SRWG}
  \GI(x)\sim \frac{C_{\Jump}}{|x|^{d-2}}.
\end{equation}
Since $\GI(x)$ is positive for all $x$ this implies that there is
a constant $c_{\Jump}$ depending on $\Jump$ such that $\GI(x)\ge
\frac{c_{\Jump}}{|x|^{d-2}}$ for $x\ne 0$.

Recall that the discrete convolution $f\ast h$ of functions $f$ and
$h$ on $\ZZ^{d}$ is defined by
$f\ast h(x)\bydef\sum_{y\in\ZZ^d} f(y)h(x-y)$.  For $n\in\NN$ let
$f^{\ast n}$ denote the $n$-fold convolution of $f$ with itself, and
let $f^{\ast 0}(x)=\indic{x=0}$. The next lemma is proved in
Appendix~\ref{appendix}.
\begin{lemma}
  \label{lem:free-green-inf}
  Suppose $d\geq 3$. For $\GI$ defined by \eqref{S1def} and
  $x \in \ZZ^{d}$,
  \begin{equation}
    \label{e:free-green-inf-1}
    \Jump \ast \GI (x)
    =
    \GI \ast \Jump(x)
    =
   - \indic{x=0} . 
  \end{equation}
  Moreover, recalling $\hJ$ and $\jumprate$ from \eqref{e:hJ-def},
  \begin{equation}
    \label{e:free-green-inf-2}
    \GI (x)
    =
    \sum_{n \ge 0} \hJ^{- (n+1)} \, \jumprate^{\ast n} (x),
  \end{equation}
  where the right-hand side is a convergent sum of positive terms.
\end{lemma}

\subsection{A convenient technical choice}
\label{sec:conv-techn-choice}

In this section we make a specific choice for the measurable space
$(\Omega_{1},\Fcal)$ on which $\ZX$ is defined so that the paths of
$\ZX$ have desirable regularity properties. This reduces
the number of statements that have to be qualified as holding almost
surely (a.s.). Let
\begin{equation}
    \label{e:Omega-def}
    \Omega_1
    =
    \big\{\ZX\colon [0,\infty) \rightarrow \ZZ^{d}
   \, \big|\,\text{$\ZX$ is c\`{a}dl\`{a}g} \big\} ,
\end{equation}
where we recall a function is c\`{a}dl\`{a}g if it is right continuous
with left limits. Let $(\mc{F}_t)_{t\ge 0}$ denote the natural
filtration of $\ZX$, i.e.,~$\mc{F}_t$ is the smallest $\sigma$-algebra
on $\Omega_1$ such that $\{\ZX\,|\,\ZX_{s}=y\}\in \mc{F}_t$ for each
$s \in [0,t]$ and $y \in \ZZ^d$, and let $\mc{F}$ denote the smallest
$\sigma$-algebra on $\Omega_1$ containing $\cup_{t\ge 0}\mc{F}_t$.
Henceforth we let $\Pa$ denote the probability measure on
$(\Omega_1,\mc{F})$ under which $\ZX$ is a continuous-time random walk
on $\ZZ^d$ started at $a\in\ZZ^{d}$.

\section{The Green's function}
\label{sec:greens-fn}

In this section we define the object $\GrL_{\V{t}}(a,b)$ at the center
of our results and we refer to it as the \emph{Green's function}. It
involves a self-interacting walk starting at $a$ and ending at $b$. We
have included a parameter $\V{t}$ that specifies the additional
interaction that arises when conditioning on an initial segment of the
walk.  It would be more standard to reserve the name ``Green's
function'' for the case $\V{t}=\V{0}$.

We have two motivations for studying this Green's function. The first
is that two-point correlations of lattice spin models such as the
$n$-component $\pf$ model have representations in terms of
$\GrL_{\V{t}}(a,b)$; see Definition~\ref{def:Zpf} and
Theorem~\ref{thm:just}.  The second motivation is that it is a point
of departure for the study of random walks with self-interactions that
are functions of local time. Such models are of interest in chemistry,
physics, and probability; they include the lattice Edwards model which
we define in Definition~\ref{def:SAW}.  This is a canonical model of
self-avoiding walk.

Fix a finite set $\Lambda\subset\ZZ^{d}$.  For a set $A$,
$A^{\Lambda}$ denotes the set of sequences $(x_v)_{v\in \Lambda}$ with
each component $x_{v}$ in $A$.  Let
$Z\colon \co{0,\infty}^{\Lambda}\to\ob{0,\infty}$,
$\V{t} \mapsto Z_{\V{t}}$ be a bounded continuous positive function.
For a random variable $\V{\sigma}$ taking values in
$[0,\infty)^{\Lambda}$, $Z_{\V{\sigma}}$ denotes $Z$ evaluated at the
random point $\V{\sigma}$. For $\V{t},\V{s}\in [0,\infty)^{\Lambda}$
let
\begin{equation}
  \label{e:rptn}
  \rrptn_{\V{t},\V{s}}\bydef\frac{Z_{\V{t}+\V{s}}}{Z_{\V{t}}}. 
\end{equation}
For $a,b\in \Lambda$ and $\V{t} \in[0,\infty)^\Lambda$ define the
\emph{Green's function}
\begin{equation}
    \GrL _{\V{t}}(a,b)
    \bydef
    \int_{[0,\infty)} d\ell \,
    \;
    \Ea\,\Big[
    \rrptn_{\V{t},\V{\tau}^{\sss(\Lambda)}_{[0,\ell]}}
    \indic{\LX_{\ell} = b}\Big] .
    \label{e:GF}
\end{equation}
Note that $\GrL_{\V{t}}(a,b)>0$ since $Z_{\V{t}}$ is continuous and
positive. We extend the definition \eqref{e:GF} by setting $\GrL
_{\V{t}}(a,b)=0$ if $a$ or $b$ is the cemetery state $\ast$.

The free Green's function $\GL (a,b)$ is the special case of
$\GrL _{\V{t}}$ when $Z\bydef 1$, see~\eqref{e:gabdef}. For each
$\V{t}$ the function $\rrptn_{\V{t},\V{\tau}^{\sss(\Lambda)}_{[0,\ell]}}$ is bounded as a
function of $\omega \in \Omega_1$ and $\ell\in\co{0,\infty}$ because
$Z_{\V{t}}$ is bounded and positive. By \eqref{e:GsrwIV} this implies
\begin{equation}
  \label{eq:GF-finite}
  \GrL_{\V{t}}(a,b)<\infty, \qquad \V{t}\in[0,\infty)^{\Lambda}.
\end{equation}

Our primary interest in this paper is $\GrL_{\V{0}}(a,b)$ given by
\eqref{e:GF} when $\V{t}\mapsto Z_{\V{t}}$ is one of the choices
described in the next two sections. Both choices involve parameters
$g>0$ and $\nu\in \RR$ called {\em coupling constants}.

\subsection{The \edwards}
\label{sec:latt-edwards-model}

\begin{definition}
  \label{def:SAW} 
  Fix $\cpl>0$ and $\nu\in \RR$.
  The \emph{Green's function $\GrL_{\V{t}}(a,b)$ of the \Ledwards}
  is given by \eqref{e:GF} and \eqref{e:rptn} with the choice
\begin{equation}
  \label{eq:defSAW}
  Z_{\V{t}} \bydef \exp\left\{-\cpl\sum_{x \in \Lambda}
  t_{x}^{2}-\nu\sum_{x\in\Lambda}t_{x}\right\}.
\end{equation}
\end{definition}
To explain Definition~\ref{def:SAW} note that
\begin{equation}
    \nonumber
    \sum_{x} \tau_{[0,\ell],x}^{2} 
     =  
    \iint_{[0,\ell]^{2}}ds\,dr\, \indic{\LX_{s}=\LX_{r}}
\end{equation}
is the time $\LX$ spends intersecting itself up to time $\ell$.  Since
$g>0$, the choice of $Z_{\V{t}}$ in Definition~\ref{def:SAW} weights a
walk in \eqref{e:GF} by the exponential of minus its self-intersection
time: self-intersection is discouraged.  The parameter $\nu\in\RR$ is
called the \emph{chemical potential}, and it controls the expected
length of a walk. Thus the \edwards\ is a continuous time
self-avoiding walk.  See~\cite{Bolthausen2002} for further details and
background on this model.

\subsection{The \texorpdfstring{$\pf$}{phi4} models}
\label{sec:pf-models}

Our second choice of $Z_\V{t}$ requires some preparation.  Let
$\RR^{n\Lambda}\bydef (\RR^{n})^{\Lambda}$,
$\V{\varphi} \bydef (\varphi_{x})_{x\in \Lambda}$ be a point in
$\RR^{n\Lambda}$, and let $\varphi_{x}^{\sss[i]}$ denote the $i$th
component of $\varphi_{x}\in\RR^{n}$. Define a centered Gaussian
measure $P$ on the Borel sets of $\RR^{n\Lambda}$ in terms of a
density with respect to Lebesgue measure $d\varphib$ on
$\RR^{n\Lambda}$ by
\begin{equation}\label{e:dP-def}
    dP(\varphib)
    \bydef 
    C e^{\frac{1}{2} (\varphib,\DL \varphib)}\,d\varphib ,
\end{equation}
where $C$ normalises the measure to have total mass one and the
quadratic form $(\varphib,\DL \varphib)$ is defined by:
\begin{equation}
  \label{e:Lambda-dirichlet-form-n-comp}
  (\DL\varphib)_x^{\sss[i]}
  \bydef
  \sum_{y \in \Lambda}\DL_{x,y}\varphi_{y}^{\sss[i]}, 
  \qquad 
  (f,h)
  \bydef
  \sum_{x \in \Lambda}\sum_{i=1}^n f_{x}^{\sss[i]}h_{x}^{\sss[i]}.
\end{equation}
The covariance of $\varphib$ under $P$ is the
$n|\Lambda|\times n|\Lambda|$ positive definite matrix
$(-\DL)^{-1}_{x,y}\delta_{i,j}$; positive definiteness follows from
\eqref{posdef-Delta} in \Cref{appendix}. By Lemma~\ref{lem:free-green},
\begin{equation}
    \label{e:dP-cov}
    \int_{\RR^{n\Lambda}} dP(\varphib)\, \varphi_{x}^{\sss[i]}\varphi_{y}^{\sss[j]}
    =
    \GL (x,y)\delta_{i,j}.
\end{equation}

\begin{definition}
  \label{def:Zpf}
  Fix $\cpl>0$, $\nu\in \RR$, and $n\in\NN_{\geq 1}$.  The
  \emph{Green's function $\GrL_{\V{t}}( a,b)$ of the $n$-component
    $\pf$ model} is given by \eqref{e:GF} and \eqref{e:rptn} with the
  choice
  \begin{equation}
    \label{ZdefPF1}
    Z_{\V{t}}\bydef
    \int_{\RR^{n\Lambda}} dP(\varphib)\;
    \exp\left\{ -\sum_{x \in \Lambda} 
      \Big(
      \cpl(\frac{1}{2}|\varphi_{x}|^{2} + t_{x})^{2} + \nu (\frac{1}{2}|\varphi_{x}|^{2}+t_{x}) \Big)
    \right\} .
  \end{equation}
\end{definition}

The justification for Definition~\ref{def:Zpf} is given by the next
theorem. To state the theorem, define an expectation
$\pair{\cdot}^{\sssL}_{\cpl,\nu, \V{t}}$ by
  \begin{equation}
    \label{eq:pfdef}
    \pair{F}_{\cpl,\nu, \V{t}}^{\sssL}
    \bydef
    \frac{1}{Z_{\V{t}}}\int_{\RR^{n\Lambda}}F(\varphib) 
    e^{-\frac{1}{2} (\varphib, -\DL \varphib)}\;
    \prod_{x\in \Lambda}
    \left(e^{-V_{t_{x}} (\frac{1}{2}|\varphi_x|^{2})}d\varphi_{x} \right), 
  \end{equation}
  where for $\psi,s\in \RR$,
  $V_{s}(\psi)=\cpl(\psi+s)^{2}+\nu(\psi+s)$.  We abbreviate
  $\pair{F}^{\sssL}_{\cpl,\nu,\V{0}}$ to
  $\pair{F}^{\sssL}_{\cpl,\nu}$.
  
\begin{theorem}
  \label{thm:just}
  Let $\GrL_{\V{t}}(a,b)$ be given by \Cref{def:Zpf}.  Then
  \begin{equation}
    \label{e:just}
    \GrL_{\V{t}}(a,b) =
    \frac{1}{n}\pair{\varphi_{a}\cdot\varphi_{b}}_{\cpl,\nu,\V{t}}^{\sssL}. 
  \end{equation}
\end{theorem}
When $\V{t} = \V{0}$ the right hand side
$\frac{1}{n}\pair{ \varphi_{a}\cdot \varphi_{b}}^{\sssL}_{\cpl,\nu}$
in \eqref{e:just} is the standard
definition of the $n$-component $\pf$ two-point function, see, e.g.,
\cite[Section~1.6]{BBSbook}. Note this reference writes $\pair{
\varphi^{\sss[1]}_{a} \varphi^{\sss[1]}_{b}}_{\cpl,\nu}^{\sssL}$ in
place of $\frac{1}{n}\pair{ \varphi_{a}\cdot
\varphi_{b}}_{\cpl,\nu} ^{\sssL}$ which is the same by $O(n)$
invariance.

\begin{proof}[Proof of \Cref{thm:just}]
  We will use the BFS-Dynkin isomorphism as formulated in
  \cite[Theorem~11.2.3]{BBSbook}. To translate between the notation of
  the present article and~\cite{BBSbook} note that in the latter
  $\tau_x= \frac12 |\varphi_x|^2$ and $L_{T}$ is the vector of local
  times of the walk up to time $T$.  For $x\not=y\in\Lambda$ let
  $\beta_{x,y}=J(y-x)$ and for $x\in\Lambda$ let
  $\gamma_{x} = \sum_{y \not \in \Lambda}J(y-x)$. By
  \cite[(11.1.9)]{BBSbook} the Laplacian $\Delta_{\beta}$ in
  \cite[Theorem~11.2.3]{BBSbook} is a $\Lambda\times\Lambda$ matrix
  whose rows sum to zero with matrix elements
  $(\Delta_{\beta})_{x,y} = J(y-x) + \1_{x=y}\gamma_{x}$. By
  comparison with the matrix $\DL$ defined by the last line of
  \Cref{sec:ctrw-fin} and \eqref{e:Delta-infty-matrix-def} we obtain
  $(-\Delta_{\beta})_{x,y} = (-\DL)_{x,y} -
  \1_{x=y}\gamma_{x}$. Therefore, by \cite[Theorem~11.2.3]{BBSbook}
  with
  \begin{equation}
    F (\V{s})
    =
    \exp \big[\sum_{x\in \Lambda} \big(
    \gamma_{x} s_{x} - V (s_{x})
    \big)\big],
  \end{equation}
  we obtain
  $\pair{\varphi_{a}^{[1]}\varphi_{b}^{[1]}}_{\cpl,\nu,\V{0}}^{\sssL}
  = \GrL_{\V{0}}(a,b)$ which is the same as \eqref{e:just} with
  $\V{t}=0$. The desired \eqref{e:just} with $\V{t}$ not necessarily
  $\V{0}$ is obtained by replacing $F(\V{s})$ by $F(\V{s}+\V{t})$.

\end{proof}

\subsection{Main result} 
\label{sec:main-result}

Our main result Theorem~ \ref{thm:main} concerns the infinite volume
limit of the Green's function for the examples in the previous
sections.

\begin{lemma}[Proof in~\Cref{sec:verification}]
  \label{lem:Z-1-earlier}
  For $\Lambda \ni a,b$, $\GrL_{\V{0}}(a,b)$ is
  non-decreasing in $\Lambda$ for the $n=1,2$-component $\pf$ and
  Edwards models.
\end{lemma}

By inspecting Definitions~\ref{def:SAW} and \ref{def:Zpf} we observe
that our examples satisfy $Z_{\V{t}+\V{s}}\le e^{-\nu s}Z_{\V{t}}$ for
all $\Lambda$. Therefore, by \eqref{e:rptn} and \eqref{e:GF},
$\sum_b\GrL _{\V{0}}(a,b)\le\int d\ell \,
\Ea\,\big[e^{-\nu\ell}1\big]$ $\le \frac{1}{\nu}$ for $\nu>
0$. Accordingly, \Cref{lem:Z-1-earlier} implies that
$\lim_{ \Lambda \uparrow \ZZ^d}\GrL _{\V{0}}(a,b)$ exists for our
examples, and is finite if $\nu>0$.  By a standard monotonicity
argument (\Cref{lem:Ginf}) the limit is $\ZZ^{d}$ invariant.  We
define
\begin{equation}
    \label{e:Ggnu-def}
    G_{\cpl,\nu} (x) = \GrI_{\cpl,\nu}(x)
    \bydef
    \lim_{ \Lambda \uparrow \ZZ^d}\GrL _{\V{0}}(a,a+x) .
\end{equation}
A related monotonicity property of our models is

\begin{lemma}[Proof in~\Cref{sec:verification}]
  \label{lem:Z-1-earlier-2}
  For each $x\in\ZZ^{d}$, $G_{\cpl,\nu}(x)$ is non-increasing in
  $\nu$ for the $n=1,2$-component $\pf$ and Edwards models. 
\end{lemma}
This lemma motivates defining the \emph{critical value} of $\nu$ by
\begin{equation}
  \label{e:critical}
  \nu_{c} \bydef \inf\big\{\nu\in\RR \mid
  \sum_{x\in\ZZ^{d}} G_{\cpl,\nu}(x)<\infty\big\}.
\end{equation}
Up to this point all we know is that $\nu_{c}\in [-\infty,0]$. We will
prove that $\nu_{c}\not = -\infty$; this is known
\cite{FrohlichSimonSpencer}, \cite[Corollary 3.2.6]{MadrasSlade} for
essentially the same models.  Since our models depend on $\cpl$,
$\nu_{c}$ is a function of $\cpl$ and, when necessary, we write
$\nu_{c}=\nu_{c} (\cpl)$.  When $\nu = \nu_{c}$ we say that the
Green's function is \emph{critical}.

Our main result is the following precise version of
Theorem~\ref{thm:main-intro} which now also includes the Edwards
model.

\begin{theorem}
  \label{thm:main} 
  Suppose $d\geq 5$ and $\Jump$ satisfies \ref{as:J1}--\ref{as:J4n}.   Consider
  the Edwards and the $\pf$ models given by
  Definition~\ref{def:SAW} and Definition~\ref{def:Zpf} with
  $n=1,2$. For both models there exists $\cpl_0 =\cpl_0(\Jump)>0$ such
  that, for $\cpl\in (0,\cpl_0)$, $\nu_{c}(\cpl)$ is finite,
  $G_{\cpl,\nu_{c}} (x)$ is finite for all $x\in\ZZ^{d}$, and there
  exists $C=C(\cpl,\Jump)>0$ such that
\begin{equation}
  \label{e:main}
  G_{\cpl,\nu_{c}} (x)
  \sim
  \frac{C}{|x|^{d-2}},
  \qquad \text{ as } |x|\ra \infty.
\end{equation}
\end{theorem}                                          
\Cref{thm:main} describes mean field asymptotics of the infinite
volume Green's function at the critical point, c.f.\ \eqref{SRWG}. The
restriction to $n=1,2$ for the $\pf$ models is necessary because our
proof uses the Griffiths II inequality, which is not known to hold for
$n>2$.  

The proof of \Cref{thm:main} occupies Sections~\ref{sec:Overview}
through~\ref{sec:verification}.  \Cref{sec:Overview} serves as an
overview of lace expansion methods and reduces a key step of our
argument to some auxiliary lemmas. The remainder of the argument is
comprised of three parts: the derivation of a lace expansion in finite
volume (\Cref{sec:laces,sec:LF}), establishing an infinite volume
expansion (\Cref{sec:Pi-DB-outer} through \ref{sec:nIVL}), the
analysis of this expansion (\Cref{sec:completion-proof}), and the
application of this analysis to our examples
(\Cref{sec:verification}). The contents of individual sections will be
discussed locally.

\subsection{Related lace expansion results}
\label{sec:related-work}

In~\cite{Sakai07} Sakai proved a similar result for the Green's
function of the Ising model. He applied the lace expansion for
percolation to the random current representation of the Ising model.
For his expansion to converge he required the dimension $d$ of the
lattice to be large or alternatively the range of the Ising coupling
to be large. In a second paper~\cite{Sakai2015} he extended his
results to the scalar $\pf$ model; he approximated the scalar $\pf$
model by Ising models using the Griffiths-Simon
trick~\cite{GriffithsSimon} and thereby derived a lace expansion for
$\pf$ that converges for weak coupling. His breakthrough inspired us
to find the expansion used in this paper.

\section{Infrared bound and overview}
\label{sec:Overview}

A key step in the proof of \Cref{thm:main} is to obtain the upper
bound on $G_{\cpl,\nu_{c}}(x)$ provided by
\Cref{thm:Main-Pre-models} below. This section begins the proof of
\Cref{thm:Main-Pre-models} by reducing it to lemmas which will be
proved in later sections.  Our reduction reviews the guiding ideas of
proofs by lace expansion, which are explained in more detail and
attribution in~\cite{Slade}. See also~\cite{HHS2003,BHK}.

Recall the definitions of the Edwards model and the $\pf$ model from
Sections~\ref{sec:latt-edwards-model} and \ref{sec:pf-models}. The
infinite volume Green's functions $G_{\cpl,\nu}(x)$ for these models
are given by \eqref{e:Ggnu-def}.  We are mainly interested in the case
where $\nu=\nu_{c}$, the critical value given by \eqref{e:critical}.
\emph{The hypotheses for results in this section include
  \Cref{hyp:Jnew}.}

\begin{theorem}
  \label{thm:Main-Pre-models}
  Suppose $d\geq 5$. For the $n=1,2$-component $\pf$ and Edwards
  models there are $\cpl_{0}=\cpl_{0}(d,\Jump)>0$ such that if
  $0<\cpl<\cpl_{0}$ then $\nu_{c}(\cpl)$ is finite and
    \begin{equation}\label{eq:Main-Pre00}
       G_{\cpl,\nu_{c}}(x) \le 2\greens (x)
       , \qquad x\in\ZZ^{d}.
  \end{equation}
\end{theorem}

Equation~\eqref{eq:Main-Pre00} is called an \emph{infrared
  bound}. Infrared bounds in Fourier space for nearest neighbour
models (i.e., $\Jump$ as in Example~\ref{ex:latlap}) were first proved
for $n\ge 1$ and $d>2$ in~\cite{FrohlichSimonSpencer} with the $2$
replaced by a $1$.  The relation between Fourier infrared bounds
and~\eqref{eq:Main-Pre00} is not trivial, see \cite[Appendix~A]{Sokal}
and \cite[Example~1.6.2]{MadrasSlade}.

While \Cref{thm:Main-Pre-models} only stated $\nu_{c}$ is finite, a
more precise estimate holds:

\begin{proposition}[Proof in \Cref{sec:verification}]
  \label{prop:og}
  For the $n=1,2$-component $\pf$ and Edwards models,
  $\nu_{c}=-O(\cpl)$ as $\cpl\downarrow 0$.
\end{proposition}

In \Cref{prop:og} and in what follows, for functions $f,r$, the
notation $f(x)=O(r(x))$ as $x\to a$ has its standard meaning,     
i.e., that there exists a $C>0$ such that $\abs{f(x)}\leq C r(x)$
if $x$ is sufficiently close to $a$.

\subsection{The infrared bound}
\label{sec:infrared}

Recall that $S$ is the free Green's function from~\eqref{e:Sinf}. The
heart of the proof of Theorem~\ref{thm:Main-Pre-models} is
establishing the next proposition.

\begin{proposition}
  \label{prop:infrared}
  Suppose $d\geq 5$.  For the $n=1,2$-component $\pf$ and Edwards
  models, there are
  $\cpl_{0}=\cpl_{0}(d,\Jump)>0$ such that if $0<\cpl<\cpl_{0}$ then
  \begin{equation}
    \label{eq:Main-Pre1}
    G_{\cpl,\nu} \leq 2 \greens, \qquad \text{for }\nu > \nu_{c}.  
  \end{equation}
  The possibility $\nu_{c}(\cpl)=-\infty$ is included.
\end{proposition}

Before giving the proof we review the strategy and state some
preparatory results. Lace expansion arguments have been reduced to
three schematic steps, all for $\nu>\nu_{c}$. As we discuss these steps
it will be helpful to recall \eqref{e:free-green-inf-1}, i.e.,
$\Jump \ast \greens= - \indic{x=0}$.  This is equivalent to
\begin{equation}
  \label{e:Int-Conv2}
    \hJ \greens (x)
    =
    \indic{x=0} + \jumprate\ast \greens (x), 
\end{equation}
where $\hJ$ and $\jumprate$ were defined in \eqref{e:hJ-def}. Define $\selfloop{\cpl,\nu} \in \RR$ by
\begin{equation}
    \selfloopL{\cpl,\nu,x}
  \bydef
  \lim_{\V{t}\downarrow\V{0}}\partial_{t_{x}}\log Z^{\sssL}_{\V{t}},
  \qquad
  \selfloop{\cpl,\nu}
  =
  \selfloopI{g,\nu,x}
  \bydef
  \lim_{\Lambda\uparrow\ZZ}
  \selfloopL{\cpl,\nu,x},
  \label{eq:selfloop-def}
\end{equation}
where $x$ is any point in $\ZZ^{d}$ and $\ZI_{\V{t}}$ depends on the
model: for the Edwards model $\ZI_{\V{t}} = Z_{\V{t}}$ in
\eqref{eq:defSAW}; for the $\pf$ model $\ZI_{\V{t}} = Z_{\V{t}}$ in
\eqref{ZdefPF1}. We say that $\selfloop{\cpl,\nu}$ is well-defined if
the limits exist and $\selfloop{\cpl,\nu}$ does not depend on $x$. For
the Edwards model \eqref{eq:defSAW} implies that
$\selfloopL{\cpl,\nu,x} = - \nu \indic{x\in\Lambda}$ and therefore
$\selfloop{\cpl,\nu}=-\nu$.

\paragraph{Step one}
\label{sec:step-one}

We will call a bound of the form
  \begin{equation}
    \label{eq:KIRB}
    G_{\cpl,\nu}\leq \cirb\greens   
\end{equation}
a \emph{$\cirb$-infrared bound} or \emph{$\cirb$-IRB}.  Step one
assumes a $3$-IRB and uses the assumption that $g$ is sufficiently
small to prove that there exists an $O (\cpl)$ integrable function
$\aPi_{\cpl,\nu}\colon \ZZ^{d}\to \RR$ such that for all $x$
\begin{equation}
  \label{eq:Int-Conv} 
  (\hJ-\selfloop{\cpl,\nu})G_{\cpl,\nu}(x)
  =
  \indic{x=0} + \jumprate\ast
  G_{\cpl,\nu}(x) + \aPi_{\cpl,\nu}\ast G_{\cpl,\nu}(x) .
\end{equation}
This is a generalization of \eqref{e:Int-Conv2}. The proof of this
step is accomplished by a formula for $\aPi_{\cpl,\nu}$ called the
\emph{lace expansion}.

\paragraph{Step two}
\label{sec:step-two} 
Step two assumes $\aPi_{\cpl,\nu}$ is small relative to $\Jump$ and
shows that \eqref{eq:Int-Conv} implies that $G_{\cpl,\nu}(x)$
satisfies a $2$-IRB.  Thus steps one and two combined show that a
$3$-IRB implies a $2$-IRB.

\paragraph{Step three}
\label{sec:step-three} 
Step three removes the $3$-IRB assumption of step one so that
\eqref{eq:Main-Pre1} holds unconditionally.  The removal of the
$3$-IRB assumption uses continuity of $G_{\cpl,\nu} (x)$ in $\nu$
together with an auxiliary result that $G_{\cpl,\nu} (x)$ satisfies a
$2$-IRB if $\nu$ is large enough. The $2$-IRB cannot be lost as $\nu$
is decreased towards $\nu_{c}$ because step two implies that
$G_{\cpl,\nu} (x)$ cannot continuously become greater than
$3\greens(x)$.

\subsection{Proof of \Cref{prop:infrared}}
\label{sec:proof-prop:infrared}

The three steps outlined in the previous section become the proof
  of Proposition~\ref{prop:infrared} given at the end of this section,
  but first we state the lemmas which are the precise versions of these
  steps. This requires two auxiliary functions. For $|z|\le \hJ^{-1}$, define $\dgreens_{z}$ and
  $\decon^{\dgreens}_{z}$ on $\ZZ^d$ by
\begin{equation}
  \label{e:dgreensz-def}
    \dgreens_{z} (x)
    \bydef
    \sum_{n\ge 0} (z \jumprate)^{\ast n}(x), 
    \quad  
    \decon^{\dgreens}_{z}(x)
    \bydef
    - \indic{x=0} + z \jumprate(x).
\end{equation}
$\decon^{\dgreens}_{z}$ is a variant of $\Jump$ such that when
$z=\hJ^{-1}$ the jump rates are normalised by $z$ to probabilities and
when $z \in [0,\hJ^{-1})$ the walk has a positive killing rate.  By
\eqref{e:free-green-inf-2} the series defining $\dgreens_{z} (x)$ in
\eqref{e:dgreensz-def} is absolutely convergent, and it is
straightforward to check that
$\decon^{\dgreens}_{z}\ast\dgreens_{z}(x) = -\indic{x=0}$. Let
\begin{equation}
    \label{e:dgreens-choice}
    \dgreens(x)
    \bydef
    \dgreens_{\hJ^{-1}}(x).
\end{equation}
By comparing the definition of $\dgreens_{z}$ with \Cref{e:free-green-inf-2}
and using \eqref{SRWG} 
\begin{equation}
    \label{SRWG2}
    \dgreens_{z}(x) 
    \le 
    \dgreens(x) 
    = 
    \hJ \greens(x)
    \le
    \frac{\tilde C_{\Jump}}{\mnorm{x}^{d-2}},
\end{equation}
for some $\tilde C_{\Jump}>0$, where $\mnorm{x}\bydef\max\{|x|, 1\}$.

Step one is 
\Cref{lem:BHK-1-1*} \ref{BHK1-1*a} and \ref{BHK1-1*b}, and
\Cref{lem:BHK-1-2*}~\ref{BHK1-1} stated below.
\begin{lemma}[Proof in \Cref{sec:verification}] 
  \label{lem:BHK-1-1*}
  Let $d\geq 5$, and consider the \ctwsaw or the $\pf$ model with
  $n=1,2$.  There exist $\alpha>0$, $\cpl_0>0$, $\aPi_{\cpl,\nu}$ such
  that for all $\cpl\in (0,\cpl_0)$ if $G_{\cpl,\nu} \leq 3 \greens$
  then $\selfloop{\cpl,\nu} = O(\cpl)$ is well-defined and
  \begin{enumerate}[label=(\roman*)]
  \item \label{BHK1-1*a} $\aPi$ is a $\ZZ^{d}$-symmetric
      function,
  \item \label{BHK1-1*b}
    $\abs{\aPi_{\cpl,\nu}(x)}\leq \cv \alpha\mnorm{x}^{-3 (d-2)}$,
  \item \label{BHK1-1*c} \eqref{eq:Int-Conv} holds,
  \item \label{BHK1-1*d}
    $\hJ - \selfloop{\cpl,\nu} \geq \frac{1 + O (\cpl)}{3\GI(0)}$.
\end{enumerate}
\end{lemma}

By the lower bound of \cref{BHK1-1*d} we can rewrite
equation~\eqref{eq:Int-Conv} as
\begin{equation}
  \label{eq:Int-Conv-m}
  \decon_{\cpl,\nu} \ast \tilde G_{\cpl,\nu} (x)
  =
  - \indic{x=0},
\end{equation}
with the definitions
\begin{equation}      
  \label{eq:Int-Conv-m-decon}
  \begin{aligned}
    \www(\cpl,\nu) \bydef (\hJ-\selfloop{\cpl,\nu})^{-1}, \qquad \tilde
    G_{\cpl,\nu}(x) \bydef \www(\cpl,\nu)^{-1} G_{\cpl,\nu}(x) ,
    \\
    \decon_{\cpl,\nu} \bydef \decon^{\dgreens}_{\www(\cpl,\nu)} +
    \tilde\aPi_{\cpl,\nu} , 
    \qquad \tilde\aPi_{\cpl,\nu}(x) \bydef
    \www(\cpl,\nu)\aPi_{\cpl,\nu}(x).
  \end{aligned}
\end{equation}
For $C>0$ let $\Dcal_{C}$ be the class of all functions
$\decon\colon \ZZ^{d}\to \RR$ with the properties
\begin{enumerate}[label=(\roman*)]
  \label{as:BHK2}
\item \label{as:BHK2-1} $\decon$ is $\ZZ^d$-symmetric,
\item \label{as:BHK2-2} $\sum_{x\in\ZZ^{d}} \decon(x) \leq 0$,
\item \label{as:BHK2-3} there exists
  $z=z(\cpl,\decon)\in\cb{0,\hJ^{-1}}$ such that
  $\abs{\decon(x)-\decon^{\dgreens}_{z}(x)} \leq
  C\cv\mnorm{x}^{-(d+4)}$.
\end{enumerate}

\begin{lemma}[Proof in \Cref{sec:verification}]
\label{lem:BHK-1-2*}
With the same hypotheses as \Cref{lem:BHK-1-1*}, for
$\nu \in (\nu_{c},\cpl]$, there exists $C_0>0$ such that
\begin{enumerate}[label=(\roman*)]
\item \label{BHK1-2} $\decon_{\cpl,\nu} \in \Dcal_{C_{0}}$,
\item \label{BHK1-1} $\abs{\selfloop{\cpl,\nu}}\leq C_{0}\cpl$.
\end{enumerate}
\end{lemma}

Step two rests on \Cref{lem:BHK-2} below, which is a generalization
of~\cite[Lemma~2]{BHK} to finite range step distributions.

\begin{lemma}[Proof in \Cref{sec:model-indep-lemm}] \label{lem:BHK-2}
  Let $d\geq 5$, and $C>0$.  There exist
  $\cpl_{0}=\cpl_{0}(d,\Jump,C)>0$ and $C'>0$ such that for
  $\cpl\in (0,\cpl_0)$ and $D \in\Dcal_{C}$ there exists
  $H\colon \ZZ^d\to \RR$ such that
  \begin{align}
    &
      \decon\ast H = -\indic{x=0}
      \label{eq:BHK-2-inverse}
    \\
    &
      |H(x) -\dgreens_{\mu}(x)|\le C'\cpl 
      \mnorm{x}^{-(d-2)},
      &
      x \in \ZZ^d,
      \label{eq:BHK-2}
  \end{align}
  where
  $\mu \bydef\hJ^{-1} \big(1+\sum_{x\in\ZZ^{d}}\decon(x)\big) \in
  \cb{-(2\hJ)^{-1},\hJ^{-1}}$.
\end{lemma}

By \ref{as:J3} and \Cref{lem:BHK-1-1*}, $\decon_{\cpl,\nu}$ given by
\eqref{eq:Int-Conv-m-decon} is in $\Dcal_{C}$. For details refer to
the proof of \Cref{lem:BHK-1*-gen}.  We will apply
Lemma~\ref{lem:BHK-2} with $\decon = \decon_{\cpl,\nu}$, and the $\mu$
created by the lemma will be denoted by $\mu(\cpl,\nu)$.

Step three uses the fact that the continuous image of a connected
interval is connected. We state this in the form of the next lemma and
apply it to the function $F$ defined in \eqref{e:Fnu-def*} below. The
use of this lemma to extend lace expansion estimates up to the
critical point originated in~\cite{Slade87}; a related application is
in \cite{BFS83}.
\begin{lemma}[{\cite[Lemma~2.1]{HHS2003}}]
  \label{lem:BS}
  For $\nu_{1}>\nu_{c}$ let $F\colon \oc{\nu_{c},\nu_{1}}\to\RR$.  If
  \begin{enumerate}[label=(\roman*)]
  \item \label{bs1} $F(\nu_{1})\leq 2$,
  \item \label{bs2} $F$ is continuous on $\oc{\nu_{c},\nu_{1}}$,
  \item \label{bs3} for $\nu \in  \oc{\nu_{c},\nu_{1}}$ the inequality
    $F(\nu)\leq 3$ implies the inequality $F(\nu)\leq 2$,
  \end{enumerate}
  then $F(\nu)\leq 2$ for $\nu\in \oc{\nu_{c},\nu_{1}} $.
\end{lemma}

The next two lemmas provide hypotheses \ref{bs1} and \ref{bs2} of
\Cref{lem:BS} with $\nu_1=g$ for the function
$F\colon(\nu_{c},\infty)\to \RR$ defined by
\begin{equation}
  \label{e:Fnu-def*}
  F(\nu)
  \bydef
  \sup_{x\in\ZZ^{d}} \frac{G_{\cpl,\nu}(x)}{\greens(x)}.
\end{equation}

\begin{lemma}[Proof in \Cref{sec:verification}] 
\label{lem:large-nu} For the \ctwsawn\ and the $n=1,2$-component $\pf$
model, with $F$ as in \eqref{e:Fnu-def*} 
\begin{equation}
\label{eq:large-nu}
     \text{$F (\nu) \le 2$ when $\nu =g$.}
  \end{equation}
\end{lemma}

\begin{lemma}[Proof in \Cref{sec:verification}] 
\label{lem:Fcontinuous}
For the \ctwsawn\ and the $n=1,2$-component $\pf$ model the function
$F$ in \eqref{e:Fnu-def*} is continuous on $(\nu_{c},\cpl]$.
\end{lemma}

For $U,V\colon\Lambda\times\Lambda\to\RR$ we write
$UV(x,y) = \sum_{u\in\Lambda}U(x,u)V(u,y)$.  The following lemma is
the well-known algebraic fact that left and right inverses coincide
for algebraic structures with an associative product.  We will use it
in the proof of Proposition~\ref{prop:infrared} and in
\Cref{sec:infin-volume-limit-2}.

\begin{lemma}
  \label{lem:inv}
  Let $U,V,W\colon \ZZ^{d}\times\ZZ^{d}\to \RR$ satisfy
  $UV(x,y)=WU(x,y)=\indic{x=y}$ and
  $\sum_{u,v\in\ZZ^{d}}|W(x,u)| \, |U(u,v)| \, |V(v,y)|<\infty$ for
  all $x,y \in \ZZ^{d}$.  Then $V=W$.
\end{lemma}
\begin{proof}
  The absolute convergence of the sum over $u$ and $v$ implies
  associativity, $W(UV)=(WU)V$.  Therefore $W=W(UV) = (WU)V =V$.
\end{proof}

Let $\ell^{p}=\ell^{p}(\ZZ^{d})$ denote the set of
$f\colon \ZZ^{d}\to\RR$ with $\sum_{x\in\ZZ^{d}}\abs{f(x)}^{p}$
finite.

\begin{proof}[Proof of Proposition~\ref{prop:infrared}]
  By hypothesis, $d\geq 5$ and $\cpl$ is small enough such that the
  results we have stated above in this section are applicable.  By the
  remark below \eqref{e:critical} $\nu_{c}\le 0$ so $(\nu_{c},\cpl]$
  is not empty.  In terms of definition \eqref{e:Fnu-def*} we will
  prove that
  \begin{equation}
    \label{e:nts}
    F (\nu) \le 2 \qquad \text{ for $\nu\in\oc{\nu_{c},\cpl}$},
  \end{equation}
  and the desired conclusion \eqref{eq:Main-Pre1} then follows by
  \Cref{lem:Z-1-earlier-2}.  In \Cref{lem:BS} set $\nu_{1}=g$ so that
  hypotheses \ref{bs1} and \ref{bs2} of \Cref{lem:BS} are supplied by
  \Cref{lem:Fcontinuous,lem:large-nu}.  Since \eqref{e:nts} is the
  conclusion of \Cref{lem:BS}, it is enough to prove
  hypothesis~\ref{bs3} which is: for $\nu\in\oc{\nu_{c},g}$,
  $F(\nu)\leq 3$ implies $F(\nu)\leq 2$. To this end, 
  assume $F(\nu)\leq 3$. Then the conclusions of \Cref{lem:BHK-1-1*}
  and \Cref{lem:BHK-1-2*} hold and provide the hypotheses of
  \Cref{lem:BHK-2} for $\decon = \decon_{\cpl,\nu}$ given by
  \eqref{eq:Int-Conv-m-decon}. By part \eqref{eq:BHK-2-inverse} of
  \Cref{lem:BHK-2}, $H$ is the right-convolution inverse of
  $-\decon_{\cpl,\nu}$. By \eqref{eq:Int-Conv-m} $\tilde G_{\cpl,\nu}$
  is also a right-convolution inverse of $-\decon_{\cpl,\nu}$. We will
  show that this implies that $\tilde G_{\cpl,\nu}=H$ after
  demonstrating that the result follows from this claim.

  By the lower bound for $\greens(x)$ below \eqref{SRWG}, the equality
  $\dgreens(x) = \hJ \greens(x)$ in \eqref{SRWG2} and \eqref{eq:BHK-2}
  we have
  $|\tilde G_{\cpl,\nu}(x)-\dgreens_{\mu}(x)| \le O(g)\dgreens(x)$.
  Therefore $\tilde G_{\cpl,\nu}(x) \le (1+O (\cpl))\dgreens(x)$ since
  $\abs{\mu}=\abs{\mu(\cpl,\nu)}\leq \hJ^{-1}$.  By
  \eqref{eq:Int-Conv-m-decon}, \Cref{lem:BHK-1-2*} and
  $\dgreens(x) = \hJ \greens(x)$ this is the same as
  \begin{equation}
    G_{\cpl,\nu}(x)
    \le
    \frac{\hJ}{\hJ - O(\cpl)}
    \big(1+O (\cpl)\big) \greens(x) = \big(1+O(\cpl)\big)\greens(x) ,
  \end{equation}
  for $v\in (v_c,g]$ and $x \in \ZZ^d$.  By taking $\cpl$ smaller if
  necessary we have $G_{\cpl,\nu}(x) \le 2\greens(x)$ as desired.

  It only remains to prove our claim that $\tilde G_{\cpl,\nu}=H$
  given that $F(\nu)\leq 3$ and $\nu\in\oc{\nu_{c},g}$.  By
  \eqref{eq:BHK-2}, $|H(x)|\le C|x|^{-(d-2)}$. The function
  $\tilde G_{\cpl,\nu}(x)$ also decays like $|x|^{-(d-2)}$ by
  $F(\nu)\leq 3$. By \Cref{lem:BHK-1-1*} part (ii),
  $\decon_{\cpl,\nu}(x)$ decays like $|x|^{-3(d-2)}$. By the decay of
  $\tilde G_{\cpl,\nu}$ and $\decon_{\cpl,\nu}$ and $d>4$ the sum that
  defines the convolution in \eqref{eq:Int-Conv-m} is absolutely
  convergent, see \Cref{lem:HHS-2}.  Therefore this convolution is
  commutative and $\tilde G_{\cpl,\nu}$ is a two-sided convolution
  inverse to $-\decon_{\cpl,\nu}$.  Furthermore, by \Cref{lem:HHS-2},
  $\sum_{u,v\in\ZZ^{d}}|\tilde G_{\cpl,\nu}(x-u)|\,
  |\decon_{\cpl,\nu}(u-v)| \, |H(v-y)|<\infty$. By \Cref{lem:inv} with
  $W=\tilde G_{\cpl,\nu}$, $U = \decon_{\cpl,\nu}$ and $V=H$, we have
  $\tilde G_{\cpl,\nu}=H$ as claimed, and hence the proof is complete.
\end{proof}

\subsection{Proof of \Cref{thm:Main-Pre-models}}
\label{sec:proof-thm:Main-Pre-models}

\begin{lemma}[Proof in \Cref{sec:verification}] 
  \label{lem:new-nucfin}
  For the \ctwsawn\ and the $n=1,2$-component $\pf$ model, the
  hypothesis $\nu_c = -\infty$ implies $\selfloop{\cpl,\nu} \to\infty$ as
  $\nu\to-\infty$.
\end{lemma}

We use this result to prove that $\nu_{c}$ is finite as claimed in
\Cref{thm:Main-Pre-models}: since $\nu_{c}\le 0$ it is enough to rule
out $\nu_{c}=-\infty$. Towards a contradiction, suppose
$\nu_{c}=-\infty$. Then \Cref{prop:infrared} implies a $2$-IRB holds
for all $\nu\leq\cpl$, and hence \Cref{lem:BHK-1-2*}~\ref{BHK1-1} 
implies $\abs{\selfloop{\cpl,\nu}}\leq C_{0}\cpl$ for all
$\nu\leq\cpl$. This contradicts \Cref{lem:new-nucfin}.

\begin{lemma}
  \label{lem:Z5}
  For the \ctwsawn\ and the $n=1,2$-component $\pf$ models, the finite
  volume $\GrL _{\V{0}}(0,x)$ is continuous in $\nu$ at $\nu_{c}$.
\end{lemma}

\begin{proof}
  For the Edwards model, observe from \eqref{eq:defSAW} that
  $Z_{\V{t}}$ is continuous in $\nu$ pointwise in $\V{t}$ and
  uniformly bounded in $\V{t}$ for each $\nu$.  By dominated
  convergence it follows from \eqref{e:GF} that the \emph{finite}
  volume $\GrL _{\V{0}}(0,x)$ is continuous in $\nu$ at $\nu_{c}$. A
  similar argument applies to the $\pf$ model.
\end{proof}

\begin{proof}[Proof of \Cref{thm:Main-Pre-models}]

  For future reference, we note that the remainder of this proof
  deduces the desired \eqref{eq:Main-Pre00} from \eqref{eq:Main-Pre1},
  and (i) $\GrL$ is monotone in $\Lambda $, and (ii) the \emph{finite
    volume} $\GrL _{\V{0}}(0,x)$ is continuous in $\nu$ at $\nu_{c}$.
  Claim (i) is \Cref{lem:Z-1-earlier}. Claim (ii) is \Cref{lem:Z5}.

  By (i) and \eqref{eq:Main-Pre1}, for $\nu >\nu_{c}$,
  \begin{equation}
    \GrL _{\V{0}}(0,x)
    \le
    G_{\cpl,\nu}(x)
    \le
    2 \greens (x)
  \end{equation}
  By (ii) $\GrL _{\V{0}}(0,x)$ is bounded above by $2 \greens (x)$
  when $\nu =\nu_{c}$. By taking the infinite volume limit with
  $\nu =\nu_{c}$ we obtain $G_{\cpl,\nu_{c}}(x) \le 2\greens (x)$.
\end{proof}

\begin{remark}
  \label{rem:dep}
  By reviewing \Cref{sec:proof-prop:infrared} we find that the proof
  of Proposition \ref{prop:infrared} has been reduced to
  \Cref{lem:Z-1-earlier,lem:Z-1-earlier-2,lem:BHK-1-1*,lem:BHK-1-2*,lem:BHK-2,lem:BS,lem:large-nu,lem:Fcontinuous}. By
  reviewing \Cref{sec:proof-thm:Main-Pre-models} we find that the
  proof of Theorem \ref{thm:Main-Pre-models} requires
  \Cref{lem:new-nucfin,lem:Z5} in addition. The proof of \Cref{lem:Z5}
  is valid for any $Z_{\V{t}}$ that is continuous in $\nu$ pointwise
  in $\V{t}$ and uniformly bounded in $\V{t}$ for each~$\nu$.  We
  classify \Cref{lem:BHK-1-1*,lem:BHK-1-2*,lem:BHK-2,lem:Fcontinuous}
  as \emph{model-independent}: although the hypotheses of
  \Cref{lem:BHK-1-1*,lem:BHK-1-2*} mention our specific models, they
  apply, with understood variations in the $\mnorm{x}^{-3 (d-2)}$
  decay, to all models with convergent lace expansions. We classify
  \Cref{lem:Z-1-earlier,lem:Z-1-earlier-2,lem:large-nu,lem:new-nucfin,lem:Z5}
  as \emph{model dependent}.
\end{remark}

\subsection{Outline of the remainder of the paper}
\label{sec:outl-rema-paper}

Our analysis is done in a general context that includes the Edwards
and the $\pf$ models with $n=1,2$ as special cases. The general
context is a set of hypotheses on the function
$\V{t}\mapsto Z_{\V{t}}$ that enters into the definition \eqref{e:GF}
of the Green's function; see \Cref{sec:final-hypoth-proof} for a full
list of hypotheses. In the course of the paper we introduce these
hypotheses on $Z_{\V{t}}$ as they are needed. In some initial sections
we use hypotheses that will eventually be superseded; these are
indicated by ending in a $0$, e.g., \ref{as:G0} below. We verify that
the Edwards and the $\pf$ models with $n=1,2$ satisfy the hypotheses
in \Cref{sec:verification}.  We have based our proof on hypotheses on
$Z$ to facilitate extending the continuous-time lace expansion to
other models: isolating properties that currently play a role should
help the search for more appealing hypotheses.

In \Cref{sec:laces} and \Cref{sec:LF} we develop a lace expansion for
Green's functions as in \eqref{e:GF} in \emph{finite volumes}
$\Lambda\subset\ZZ^{d}$.  Working in a finite volume is essential, as
we have only defined the $\pf$ model as the infinite volume limit of
finite volume models.

The next part of the paper, Sections~\ref{sec:Pi-DB-outer}
through~\ref{sec:nIVL}, develops estimates on our finite volume lace
expansion, under the hypothesis that the Green's function satisfies a
$3$-IRB. These estimates establish the infinite-volume lace expansion
equation \eqref{eq:Int-Conv} under the general hypotheses on
$Z_{\V{t}}$ and provide the key inputs for the proofs of
\Cref{lem:BHK-1-1*} and \Cref{lem:BHK-1-2*}.

In \Cref{sec:model-indep-lemm} and \ref{sec:model-depend-lemm} we
complete the proofs of the lemmas we have used in the last two
sections, and thus establish the conclusions of
\Cref{thm:Main-Pre-models} for any $Z_{\V{t}}$ satisfying our
hypotheses.  We then make use of this result, in conjunction with a
theorem of Hara~\cite{Hara2008}, to obtain the Gaussian asymptotics of
\Cref{thm:main}.

\section{Functions on a set of intervals}
\label{sec:laces}

In this section we begin to derive the lace expansion needed for step
one of \Cref{sec:step-one}.  The main result of this section is
\Cref{thm:logZ}, which is an expansion for a function
$\rptn\colon\mc{D}\ra \RR$, $(s,t) \mapsto \rptn_{s,t}$, where
\begin{equation}
    \Dcal
    \bydef
    \{(s,t):0 \le s \le  t <\infty\}\subset [0,\infty)^2.
\end{equation}
Here $(s,t)$ denotes an ordered pair, but the same notation will
  be used for an open interval. \Cref{thm:logZ} will be used in the
next section to derive our lace expansion for Green's functions of the
form \eqref{e:GF}.  We begin with notation and minimal assumptions on
$\rptn$ needed for the main result of the section.

For a function $\rptn$ on $\mc{D}$ and $t\in (0,\infty)$, we denote by
$\rptn_{\cdot, t}\colon [0,t]\to \RR$ the function $s\mapsto
\rptn_{s,t}$, and for each $s\in [0,\infty)$, we denote by $\rptn_{s,
\cdot}\colon [s,\infty)\to \RR$ the function $t\mapsto\rptn_{s,t}$.
We will write $\parto$ and $\partt$ to denote partial differentiation
with respect to the first and second coordinates, respectively.

We will assume $\rptn$ satisfies the following assumptions. The
\emph{almost every} (a.e.) statements in the assumptions are with
respect to Lebesgue measure.

\begin{assumptions}
  \label{hyp:rptn}
  \leavevmode
  \begin{enumerate}
  \item $\rptn$ is continuous and strictly positive on $\Dcal$,
    and $\rptn_{s,s}=1$ for all $s\ge 0$.
    \label{hyp:cont1}
  \item For each $t\in (0,\infty)$, $\rptn_{\cdot, t}$ is absolutely
  continuous.  For each  $s\in [0,\infty)$, $\rptn_{s,\cdot}$ is absolutely
  continuous on bounded subintervals of $[s,\infty)$.
  \label{hyp:diff1}
  \item For a.e.\ $t\in (0,\infty)$,
  the function $(\partt \rptn)_{\cdot, t}$ is absolutely continuous on
  $\ob{0,t}$.  For a.e.\ $s\in [0,\infty)$, the function
  $(\parto \rptn)_{s,\cdot}$ is absolutely continuous on bounded
  subintervals of $\ob{s,\infty}$.  
  \label{hyp:diff2}
  \item 
  $\parto\partt\rptn = \partt\parto\rptn$ a.e.\ on the interior of
  $\Dcal$. 
  \label{hyp:diff3}
  \end{enumerate}
\end{assumptions}
We will see in \Cref{sec:prel-local-time} that the absolute continuity
statements are properties of the random function
$(s,t)\mapsto \tau_{[s,t],x}$ defined by \eqref{e:tauI-def} with
$I=[s,t]$ and $x$ fixed.  The derivatives in Item~\ref{hyp:diff2} have
open intervals of definition, while the standard definition of
absolute continuity concerns closed intervals. In this paper we say a
function $f$ defined on a bounded open interval $I\subset\RR$ is
absolutely continuous on $I$ if for $\epsilon>0$ there exists
$\delta>0$ such that for finitely many disjoint subintervals
$(a_{i},b_{i})$ of $I$ with endpoints in $I$,
$\sum |b_{i}-a_{i}|<\delta$ implies
$\sum |f(b_{i})-f(a_{i})|<\epsilon$.  Such functions are uniformly
continuous on bounded intervals and therefore the derivative
$\partt\rptn_{\cdot,t}$ in Item~\ref{hyp:diff2} extends to an
absolutely continuous function on the closed interval $[0,t]$ and
similarly $\parto\rptn_{s,\cdot}$ extends to $[s,\infty)$.  When we
write derivatives on the boundaries of their domains we mean these
extensions by continuity.

\subsection{The expansion}
\label{sec:expansion}

In this section we introduce the objects that enter into our
expansion, and state the expansion in \Cref{thm:logZ} below.

Define \emph{vertex functions} by 
\begin{equation}
\begin{aligned}
  \label{e:vertex-def}
  \vertex{s} &\bydef - \lim_{t\downarrow s} \parto 
  \lgrptn_{s,t} ,
  \qquad 0\le s <\infty,
  \\
  \vertex{s,t} &\bydef -\parto\partt\lgrptn_{s,t} , 
  \qquad 0 \le s < t <\infty.
\end{aligned}
\end{equation} 
These are a.e.\ equalities. By \Cref{hyp:rptn} and the paragraph that
follows them, the chain rule, which applies to a smooth function
composed with an absolutely continuous function, proves these
derivatives exist a.e.\ and provides formulas for them. The limit
defining $\vertex{s}$ exists by the discussion under \Cref{hyp:rptn}.

\begin{figure}
  \centering
  \begin{tikzpicture}[scale=.6]
    \node[dot] (s0) at (0,0) {};
    \node[dot] (s1) at (2,0) {};
    \node[dot] (s2) at (5,0) {};
    \node[dot] (s3) at (8,0) {};
    \node[dot] (s4) at (12,0) {};

    \node[dot] (s1p) at (3.5,0) {};
    \node[dot] (s2p) at (7,0) {};
    \node[dot] (s3p) at (11,0) {};
    \node[dot] (s4p) at (13,0) {};
    \node[dot] (s5p) at (17,0) {};

    \draw[black,thick] (s0) to[out=90, in =180] (2,2) to[out=0,in=90] (s1p);
    \draw[black,thick] (s1) to[out=90, in =180] (5,2) to[out=0,in=90] (s2p);
    \draw[black,thick] (s2) to[out=90, in =180] (8,2) to[out=0,in=90] (s3p);
    \draw[black,thick] (s3) to[out=90, in =180] (11,2) to[out=0,in=90] (s4p);
    \draw[black,thick] (s4) to[out=90, in =180] (14,2) to[out=0,in=90] (s5p);

    \node at (s0) [below]  {$s_{0}=s_{0}'$};
    \node at (s1) [below] {$s_{1}\phantom{'}$};
    \node at (s2) [below] {$s_{2}\phantom{'}$};
    \node at (s3) [below] {$s_{3}\phantom{'}$};
    \node at (s4) [below] {$s_{4}\phantom{'}$};

    \node at (s1p) [below] {$s_{1}'$};
    \node at (s2p) [below] {$s_{2}'$};
    \node at (s3p) [below] {$s_{3}'$};
    \node at (s4p) [below] {$s_{4}'$};
    \node at (s5p) [below] {$s_{5}=s_{5}'$};
  \end{tikzpicture}
  \caption{A lace with $m=5$ intervals}
  \label{fig:six-lace}
\end{figure}

For $m\in\NN$ a \emph{lace} $\lace$ is a sequence 
\begin{equation}
  L=\ob{(s_i,s'_{i+1})}_{i=0,\dots, m-1}
\end{equation}
of $m$ open intervals $(s_i,s'_{i+1})$ with $s_0 \bydef s_0'$,
$s_m\bydef s'_{m}$, and
\begin{equation}
    0 \leq s_{0}' < s_{1} < s_{1}' < s_{2} < s_{2}' < \dotsb < s_{m-1}' < s_{m}<\infty.
    \label{e:lace-def}
\end{equation}
The meaning of \eqref{e:lace-def} is illustrated by
Figure~\ref{fig:six-lace}.  The union of all of the intervals of a
lace is $(s_{0},s_{m}')$, and if any single interval is excluded from
the union the resulting set does not cover $(s_{0},s_{m}')$.

Let $\laces_{m}$ be the region in $\RR^{2m}$ defined by the
inequalities in \eqref{e:lace-def}. We identify $\laces_{m}$ with the
collection of all laces containing $m$ intervals. Let
\begin{equation}
   \laces_{0}
    \bydef
    \{s_{0}\mid\,0\leq s_{0}<\infty \}. 
\end{equation}

We associate to a lace $\lace$ a product
\begin{equation}
    \label{e:vertex-lace-def}
    \vertex{}(\lace)
    \bydef
    \begin{cases}
    \vertex{s_{0}},& \text{ if }\lace=\{s_0\} \in \laces_{0},\\
    \prod_{i=0}^{m-1}\vertex{s_i,s'_{i+1}},	
    & \text{ if }\lace=\ob{(s_i,s'_{i+1})}_{i=0,\dots, m-1} \in \laces_{m}, \ m\geq 1,
    \end{cases}
\end{equation}
of vertex functions.  The \emph{weight} $L\mapsto \weight (\lace)$ of a lace
is defined to be
\begin{equation}
    \label{e:weight}
    \weight (\lace) 
    \bydef
    \vertex{}(\lace)\times
    \begin{cases}
    1, & \lace \in \laces_{0} ,
    \\
    \tp (\lace) ,
    &\lace \in \laces_{m}, \ m \ge 1 ,
    \end{cases}
\end{equation}
where for $\lace=\ob{(s_i,s'_{i+1})}_{i=0,\dots, m-1}$ 
\begin{equation}
    \label{e:Pweight}
    \tp (\lace)
    \bydef
    \rptn_{s_{0}',s_{1}'}\, 
    \prod_{i=0}^{m-2}
    \frac{\rptn_{s_{i}',s_{i+2}'}}{\rptn_{s_{i}',s_{i+1}'}} .
\end{equation}
For $m=1$ the empty product in \eqref{e:Pweight} is defined to be one
by convention.

For $m \ge 0$ we define integration $\int_{\laces_{m}} d\lace$ over
$\laces_{m}$ to be integration with respect to Lebesgue measure on
$\laces_{m}$.  For example, if $m=0$ then $\int_{\laces_{m}}
d\lace = \int_{[0,\infty)}ds_{0}$.  Let $\laces_{m,\ell}$ be the
subset of $\laces_{m}$ such that $s_{m}\le \ell$, and let
$\Dcal_{\ell}$ be the subset of $\Dcal$ with $t\leq \ell$.

\begin{theorem}
\label{thm:logZ} Let $\rptn$ be such that (i) $\rptn$ satisfies
Assumption~\ref{hyp:rptn}, (ii) the function $\vertex{s,t}$ defined in
\eqref{e:vertex-def} is Lebesgue a.e.\ bounded on $\Dcal_{\ell}$.
Then for $\ell >0$
\begin{equation}
    \label{e:thm:logZ}
    \rptn_{0,\ell}
    =
    1
    +
    \sum_{m\ge 0}
    \int_{\laces_{m,\ell}} d\lace\;
    \weight (\lace)\;
    \rptn_{s_{m}',\ell}\;, 
\end{equation}
and the sum is absolutely convergent.
\end{theorem}

The proof of this theorem is given in \Cref{proof-thm:logZ-completed}
and is based on the two identities given in
\Cref{lem:first-step,lem:step}.  We will need the following fact from
real analysis, whose proof we omit.

\begin{lemma}
  \label{lem:Lip-AC}
  Let $I\subset \RR$ be bounded. If $f\colon I \to \RR$ is Lipschitz
  continuous and $g\colon I\to I$ is absolutely continuous, then
  $f\circ g$ is absolutely continuous. 
\end{lemma}

\subsection{The identity that starts the expansion}
\label{sec:first-step}

\begin{lemma}\label{lem:first-step} Under Assumptions~\ref{hyp:rptn},
\begin{equation}
    \rptn_{0,\ell}
    =
    1
    +
    \int_{(0,\ell)} ds_{0}\;
    \vertex{s_{0}}\;
    \rptn_{s_{0},\ell}
    +
    \int_{(0,\ell)} ds_{0}\;
    \int_{(s_{0},\ell)} d s_{1}'\;
    \vertex{s_{0},s_{1}'}\;
    \rptn_{s_{0},\ell} .
\end{equation}
\end{lemma}

\begin{proof} By Assumption~\ref{hyp:rptn}(\ref{hyp:diff1}),
\begin{equation}
  \label{e:lem:first-step05}
    \rptn_{0,\ell}
    =
    \rptn_{\ell,\ell}
    -
    \int_{(0,\ell)} ds_{0}\;
    \parto\rptn_{s_{0},\ell}\; .
\end{equation}
$\rptn$ is bounded below by a positive constant on $\Dcal_{\ell}$ by
Assumption~\ref{hyp:rptn}(\ref{hyp:cont1}) and the compactness of
$\Dcal_{\ell}$.  Hence, by $\rptn_{\ell,\ell}=1$,
\eqref{e:lem:first-step05} can be rewritten as
\begin{equation}
    \label{e:lem:first-step1}
    \rptn_{0,\ell}
    =
    1
    -
    \int_{(0,\ell)} ds_{0}\;
    \frac{\parto\rptn_{s_{0},\ell}}{\rptn_{s_{0},\ell}}\;
    \rptn_{s_{0},\ell} .
\end{equation}

By Assumption~\ref{hyp:rptn}(\ref{hyp:cont1}) the range of
$t\mapsto\rptn_{s_{0},t}$ for $t \in [s_{0},\ell]$ is the continuous
image of a compact set, and the range does not contain zero.  As
$z \mapsto 1/z$ is Lipschitz on compact subsets of $(0,\infty)$, it
follows from \Cref{lem:Lip-AC} that $t \mapsto \rptn_{s_{0},t}^{-1}$
is absolutely continuous in $t$ for $t$ in bounded subintervals of
$(s_{0},\infty)$.  Combined with
Assumption~\ref{hyp:rptn}(\ref{hyp:diff2}) this shows
$\rptn_{s_{0},t}^{-1}\, \parto\rptn_{s_{0},t}$ is absolutely
continuous in $t\in (s_{0},\ell)$ for a.e.\ $s_{0}$, hence for a.e.\
$s_{0}>0$
\begin{equation}
  \label{e:lem:first-step-2}
    \parto\log\rptn_{s_{0},\ell}
    =
    (\parto\log\rptn)_{s_{0},s_{0}}
    +
    \int_{(s_{0},\ell)} ds_{1}'\;
    \partt \parto\log \rptn_{s_{0},s_{1}'}.
\end{equation}
In this equation $(\parto\log\rptn)_{s_{0},s_{0}}$ is by definition
the continuous extension of \\
$(\parto\log\rptn)_{s_{0},t} = \rptn_{s_{0},t}^{-1}\,
\parto\rptn_{s_{0},t}$ as a function of $t>0$ to the boundary
$t=s_{0}$ of its domain. This is an instance of the convention we
declared below Assumption~\ref{hyp:rptn}.  By inserting
\eqref{e:lem:first-step-2} into \eqref{e:lem:first-step1} we obtain
\begin{align}
    \rptn_{0,\ell}
    &=
      1
    - \;
    \int_{(0,\ell)} ds_{0}\;
     \rptn_{s_{0},\ell}\; (\parto\log\rptn)_{s_{0},s_{0}}
    \nonumber\\
    &\phantom{=
    \rptn_{\ell,\ell} } \;  - \;
    \int_{(0,\ell)} ds_{0}\;
    \int_{(s_{0},\ell)} ds_{1}'\; \rptn_{s_{0},\ell}\; 
    \partt \parto\log \rptn_{s_{0},s_{1}'}.
\end{align}
The proof is completed by substituting in the definitions of
$\vertex{s_{0}}$ and $\vertex{s_{0},s_{1}'}$. For the last term the
interchange of derivatives is justified by
Assumption~\ref{hyp:rptn}(\ref{hyp:diff3}).
\end{proof}

\subsection{The identity that generates the expansion}
\label{sec:lem:step}

\begin{lemma}
  \label{lem:step}
  Suppose $\rptn$ satisfies Assumptions~\ref{hyp:rptn}.
  If the point $(u,v)\in\Dcal_{\ell}$ then
\begin{equation}
\label{e:lem:step1}
    \rptn_{u,\ell}
    =
    \rptn_{u,v}\;
    \rptn_{v,\ell}\;
    +
    \int_{(u,v)} ds_{+}\;
    \int_{(v,\ell)} ds_{+}'\;
    \vertex{s_{+},s_{+}'}\;
    \rptn_{u,s_{+}'}\;
    \frac{\rptn_{v,\ell}}{\rptn_{v,s_{+}'}} .
\end{equation}
\end{lemma}

\begin{proof}
  As $\rptn$ is bounded below by a positive constant on
  $\Dcal_{\ell}$, we can rewrite the left-hand side:
  \begin{equation}
    \label{e:5-16}
    \rptn_{u,\ell}
    =
    \rptn_{u,v}
    \rptn_{v,\ell}
    +
    \left(
    \frac{\rptn_{u,\ell}}{\rptn_{v,\ell}}
    -
    \rptn_{u,v}
    \right)
    \rptn_{v,\ell}.
\end{equation}
By Assumption~\ref{hyp:rptn}(\ref{hyp:diff1}), $\rptn_{s,s}=1$, and
the absolute continuity of $\rptn_{v,\ell}^{-1}$ in $\ell$ noted
after~\eqref{e:lem:first-step1}, equation (\ref{e:5-16}) can be
rewritten as
\begin{equation}
	\label{e:inter}
    \rptn_{u,\ell}
    =
    \rptn_{u,v}
    \rptn_{v,\ell}
    +
    \left(
    \int_{(v,\ell)} ds_{+}'\;
    \partt
    \frac{\rptn_{u,s'_{+}}}{\rptn_{v,s'_{+}}}
    \right)
    \rptn_{v,\ell} .
\end{equation}
Since $\exp$ and $\log$ are Lipschitz on compact subsets of their open
domains we can compute the derivative in~\eqref{e:inter} using
$f(x) = \exp\log f(x)$:
\begin{align}
    \rptn_{u,\ell}
    &=
    \rptn_{u,v}
    \rptn_{v,\ell}
    +
    \left(
    \int_{(v,\ell)} ds_{+}'\;
    \Big(\partt\log \ob{\frac{\rptn_{u,s_{+}'}}{\rptn_{v,s_{+}'}}}\Big)
    \frac{\rptn_{u,s'_{+}}}{\rptn_{v,s'_{+}}}
    \right)
    \rptn_{v,\ell} \\
    &=
    \rptn_{u,v}
    \rptn_{v,\ell}
    -
    \left(
    \int_{(v,\ell)} ds_{+}'\;
    \Big(
    \int_{(u,v)} ds_{+}\;
    \parto\partt\lgrptn_{s_{+},s_{+}'}
    \Big) \frac{\rptn_{u,s_{+}'}}{\rptn_{v,s_{+}'}}
    \right)
    \rptn_{v,\ell} ,
\end{align}
where we have used Assumption~\ref{hyp:rptn}(\ref{hyp:diff2}) in the
second step.  By Fubini's theorem this can be rewritten as the desired
result.
\end{proof}

\subsection{Proof of Theorem~\ref{thm:logZ}}
\label{proof-thm:logZ-completed}

Recall the definitions of $\weight(\lace)$, $\vertex{}(\lace)$ and
$\tp(\lace)$ in~\eqref{e:vertex-lace-def}--\eqref{e:Pweight} and
define
\begin{equation}
    \label{e:Rn-def}
    R_{n}
   \bydef
    \int_{\laces_{n,\ell}} d\lace\; \weight(\lace)\;
    \frac{\rptn_{s_{n-1}',\ell}}{\rptn_{s_{n-1}',s_{n}'}} ,
    \qquad n \ge 1.
\end{equation}

\begin{lemma}
  \label{lem:Rn}
  Suppose $\rptn$ satisfies Assumptions~\ref{hyp:rptn}. 
    Then
  \begin{equation}
    \label{e:thm:logZ1}
    \rptn_{0,\ell}
    =
    1 + \sum_{m=0}^{n-1} \int_{\laces_{m,\ell}} d\lace\;
    \weight (\lace)\;
    \rptn_{s_{m}',\ell}\;
    + R_{n} , \qquad n\geq 1.
\end{equation}
\end{lemma}
\begin{proof}
  We first prove \eqref{e:thm:logZ1} with $n=1$. By
  Lemma~\ref{lem:first-step}
  \begin{equation}
    \rptn_{0,\ell}
    =
    1
    +
    \int_{(0,\ell)} ds_{0}\;
    \vertex{s_{0}}\;
    \rptn_{s_{0},\ell}
    +
    \int_{(0,\ell)} ds_{0}\;
    \int_{(s_{0},\ell)} ds_{1}'\;
    \vertex{s_{0},s_{1}'}\;
    \rptn_{s_{0},\ell} .
  \end{equation}
  By the definition \eqref{e:weight} of $\weight (\lace)$ and the
  definition of integration over $\laces_{m,\ell}$ given below
  \eqref{e:Pweight} (recall also that $s_{0}'\bydef s_{0}$), this can be
  rewritten as
  \begin{equation}
    \rptn_{0,\ell}
    =
    1
    +
    \int_{\laces_{0,\ell}} d\lace\;
    \weight (\lace)\;
    \rptn_{s_{0}',\ell}\;
    +
    \int_{\laces_{1,\ell}} d\lace\;
    \vertex{s_{0}',s_{1}'}\;
    \rptn_{s_{0}',\ell} .
  \end{equation}
  This establishes \eqref{e:thm:logZ1} when $n=1$ as the final term is
  $R_{1}$.

  We now prove \eqref{e:thm:logZ1} holds for $n\geq 1$ by induction,
  using \eqref{e:thm:logZ1} as the inductive hypothesis. By
  Lemma~\ref{lem:step} with $(u,v)$ replaced by $(s_{n-1}',s_{n}')$,
  \begin{align}
    \rptn_{s_{n-1}',\ell}
    &=
      \rptn_{s_{n-1}',s_{n}'}\;
      \rptn_{s_{n}',\ell}\;
      \nonumber\\
    &\quad +
      \int_{(s_{n-1}',s_{n}')} ds_{+}\;
      \int_{(s_{n}',\ell)} ds_{+}'\;
      \vertex{s_{+},s_{+}'}\;
      \rptn_{s_{n-1}',s_{+}'} \;
      \frac{\rptn_{s_{n}',\ell}}{\rptn_{s_{n}',s_{+}'}} .
    \label{e:induct-on-n}
  \end{align}
  We insert \eqref{e:induct-on-n} into the definition \eqref{e:Rn-def}
  of $R_{n}$ and use the definition \eqref{e:weight} of
  $\weight (\lace)$ for the contribution from the first term
  $ \rptn_{s_{n-1}',s_{n}'}\; \rptn_{s_{n}',\ell}$:
  \begin{align}
    R_{n}
    &=
      \int_{\laces_{n,\ell}} d\lace\;
      \weight (\lace)
      \rptn_{s_{n}',\ell}\;
      +
      \Bigg(\int_{\laces_{n,\ell}} d\lace\;
      \vertex{}(\lace)\;
      \tp (\lace)\,
      \nonumber\\
    &\quad \times 
      \int_{(s_{n-1}',s_{n}')} ds_{+}\;
      \int_{(s_{n}',\ell)} ds_{+}'\;
      \vertex{s_{+},s_{+}'}\;
      \frac{\rptn_{s_{n-1}',s_{+}'}}{\rptn_{s_{n-1}',s_{n}'}}\;
      \frac{\rptn_{s_{n}',\ell}}{\rptn_{s_{n}',s_{+}'}}\Bigg) .
      \label{e:induct-on-n2}
  \end{align}
  Renaming $s_{+},s_{+}'$ as $s_{n},s_{n+1}'$ and combining
  the integrals in the second term into an integral over
  $\laces_{n+1,\ell}$ yields
  \begin{equation}
    R_{n}
    =
      \int_{\laces_{n,\ell}} d\lace\;
      \weight (\lace)\;
      \rptn_{s_{n}',\ell}\;
      +
      \int_{\laces_{n+1,\ell}} d\lace\;
      \vertex{}(\lace)\;
      \tp (\lace)\,
      \frac{\rptn_{s_{n}',\ell}}{\rptn_{s_{n}',s_{n+1}'}}    
   \end{equation}
   For the second term on the right of \eqref{e:induct-on-n2},
   $\vertex{s_{+},s_{+}'}\bydef\vertex{s_{n},s_{n+1}'}$ became part
   of $\vertex{}(\lace)$ and the ratio of $\rptn$'s became part of
   $\tp (\lace)$ when the range of integration became
   $\laces_{n+1,\ell}$. By the definition \eqref{e:Rn-def} of
   $R_{n+1}$, this is
   \begin{equation}
    \label{e:Rn-1}
    R_{n}=
      \int_{\laces_{n,\ell}} d\lace\;
      \weight (\lace)\;
      \rptn_{s_{n}',\ell}\;
      +
      R_{n+1} . 
  \end{equation}
  Inserting \eqref{e:Rn-1} into the inductive hypothesis
  completes the proof. 
\end{proof}

\begin{proof}[Proof of \Cref{thm:logZ}]
  We justify taking the $n\to\infty$ limit of \Cref{lem:Rn}.

  The factors $\rptn_{s,t}$ under the integrals in \eqref{e:thm:logZ1}
  are bounded above and below because they are strictly positive and
  continuous functions on the compact domain $\Dcal_{\ell}$. Together
  with the assumption that $\vertex{s,t}$ is uniformly bounded this
  proves that there is a constant $C=C(\ell)$ such that
  $\abs{\weight (\lace) \rptn_{s'_{n},\ell}} \le C^{n+1}$ for
  $\lace \in \laces_{n,\ell}$ and $n\ge 1$, where $\weight$ was
  defined in \eqref{e:weight}. Similarly the integrand in $R_{n}$ is
  bounded by $C^{n+1}$ for $\lace \in \laces_{n,\ell}$. Because
  $\rptn_{s,t}$ is bounded above and below we also have that
  $\abs{\weight(\lace)\rptn_{s'_{0},\ell}}$ is bounded by a constant
  $C'=C'(\ell)$ when $\lace\in\laces_{0,\ell}$.

  The Lebesgue measure of $\laces_{n,\ell}$ is the Lebesgue measure of
  all $2n$-tuples of ordered points in $(0,\ell)$, which is
  $\frac{1}{(2n)!}\ell^{2n}$. Therefore
  $|R_{n}| \le C^{n+1}\frac{1}{(2n)!}\ell^{2n}$ and the $m$th term in
  the sum over $m$ in \eqref{e:thm:logZ1} is bounded by
  $C^{m+1}\frac{1}{(2m)!}\ell^{2m}$.  Therefore the series in the
  right hand side of \eqref{e:thm:logZ} is absolutely convergent and
  equals $\rptn_{{0,\ell}}$ as claimed because
  $\lim_{n \rightarrow \infty } R_{n} = 0$ in \eqref{e:thm:logZ1}.
\end{proof}

\section{The lace expansion in finite volume}
\label{sec:LF}

In this section we continue with step one of
\Cref{sec:step-one}. Throughout this section $\Lambda\subset\ZZ^{d}$
denotes a fixed finite set. The main result is
\Cref{prop:lace-expansion}, which provides the finite-volume version
of \eqref{eq:Int-Conv}.  This proposition involves a function
$\PiL(x,y)$, and the formula \eqref{e:Pi-def} for $\PiL(x,y)$ is
called a lace expansion for reasons to be explained following
\Cref{prop:Pi-DB}. We begin by introducing some further definitions
and assumptions.

Given $Z^{\sssL}\colon [0,\infty)^{\Lambda}\to (0,\infty)$,
$\V{u}\mapsto Z^{\sssL}_{\V{u}}$, we choose the function
  $\rptn\colon \Dcal\to \RR$ of \Cref{sec:laces} to be the random function
\begin{equation}
    \label{e:rptn-choice}
    \rptn_{s,t}
    =
    \rptnL_{s,t} \bydef \ob{\frac{Z^{\sssL}_{\cdot}}{Z^{\sssL}_{\V{0}}}}\circ
      \V{\tau}^{\sssL}_{[s,t]}
    =
    \frac{1}	{Z^{\sssL}_{\V{0}}}
      Z^{\sssL}_{\V{\tau}^{\sssL}_{[s,t]}}.
\end{equation}
Recall \eqref{e:rptn} and note that
$\rptn_{s,t}=\rrptn_{\V{0},\V{\tau}^{\sssL}_{[s,t]}}$.
Henceforth $\rptn$ is given by \eqref{e:rptn-choice}.
Let $\GrL$ be the Green's function determined by \eqref{e:GF} with
this choice of $Z^{\sssL}$.  Recall that the weight $\weight(\lace)$
of a lace is defined in terms of $\rptn_{s,t}$ in~\eqref{e:weight}. We
write $\weightL(\lace)\bydef\weight(\lace)$ for the weight with the
choice \eqref{e:rptn-choice}.

Let $\laces_{m}(s)\subset\laces_{m}$ be the hypersurface defined by
$s_{0}'=s$.  For $m\geq 1$ we write $\int_{\laces_{m}(s)}d\lace$ for
integration with respect to Lebesgue measure on $\laces_{m}(s)$.
Then, by \eqref{e:lace-def}, we have
\begin{equation}
  \label{e:disintegration}
    \int_{\laces_{m}}d\lace
    =
    \int_{[0,\infty)}ds\int_{\laces_{m} (s)}d\lace .
\end{equation}
For $m=0$, since $\laces_{0}(s)$ consists of the single point
$s_{0}'=s$ we let $d\lace$ in the inner integral denote a unit Dirac
mass at $s$.

In the following assumptions, and hereafter, we write $Z$ in place of
$Z^{\sssL}$ when there is no ambiguity.
\begin{assumptions}
\label{hyp:w-rptn}

For all $a\in\Lambda$,
\begin{enumerate}
\item[\mylabel{as:Z1}{(Z1)}] $\V{t} \mapsto Z_{\V{t}}$ is strictly
  positive and continuous on $\co{0,\infty}^{\Lambda}$.
  
\item[\mylabel{as:Z2}{(Z2)}] $\V{t} \mapsto Z_{\V{t}}$ is
  $\Ccal^{2}$ on $[0,\infty)^{\Lambda}$.

\item[\mylabel{as:G0}{(G0)}]
  $ \int_{[0,\infty)} d\ell\; \Ea \,\Big[
  \rptnL_{0,\ell}\Big] < \infty $.

\item[\mylabel{as:F0}{(F0)}]
  $\sum_{m\ge 0}\int_{\laces_{m} (0)} d\lace\; \Ea \big[ |\weightL
  (\lace)|\; \big] < \infty$.
\end{enumerate}
\end{assumptions}
In \ref{as:Z2}, $\Ccal^{2}$ at the boundary means that the derivatives
have continuous extensions to the boundary.  Assumption \ref{as:F0}
enables us to define $\PiL\colon \Lambda \times \Lambda\to\RR$ by
\begin{equation}
    \label{e:Pi-def}
    \PiL (x,y)
    =
    \sum_{m\ge 0}
    \int_{\laces_{m}(0)} d\lace\;
    \EEa_{x}\,\big[ \weightL (\lace)\indic{\LX_{s_{m}'}=y}\big] .
\end{equation}
and \ref{as:G0} is equivalent to the Greens function $\GrL(x)$ being
finite for all $x\in\Lambda$.

\begin{proposition}
  \label{prop:lace-expansion}
  Under Assumptions~\ref{hyp:w-rptn},
  \begin{equation}
    \label{eq:lace-expansion}
    \GrL_{\V{0}} (a,b) 
    =
    \GL  (a,b) +
    \sum_{x,y \in \Lambda} \GL  (a,x)\PiL (x,y)
    \GrL_{\V{0}} (y,b) .  
  \end{equation}
\end{proposition}
In the literature of theoretical physics this relation between
$\GrL_{\V{0}}$, $\GL$ and $\PiL$ is called a Dyson equation
\cite[Equation (92)]{Dyson1949}.  The proof of the above proposition
occupies the rest of this section.

\subsection{Derivatives of local time}
\label{sec:prel-local-time}

For each $x\in\Lambda$, \eqref{e:tauI-def} defines a local time
$\V{\tau}_{[s,s'],x}$ that is absolutely continuous in $s$ for $s\leq
s'$ with $s'$ fixed, and similarly is absolutely continuous in $s'$
for fixed $s$ when $s'$ is restricted to a bounded interval.  Note
that
\begin{align}
    &
    \partial_{2} \tau_{[s,s'],x} 
    =
    \partial_{2} 
    \int_{[s,s']} \indic{\LX_{r}=x}\,dr\,  
    =
    \indic{\LX_{s'}=x} ,
    \label{e:LT-D-1}\\
    &
    \partial_{1} \tau_{[s,s'],x} 
    =
    -
    \indic{\LX_{s}=x} ,
    \label{e:LT-D-2}\\[3mm]
    &   
    \partial_{2} \partial_{1}\tau_{[s,s'],x} 
    =
    \partial_{1} \partial_{2}\tau_{[s,s'],x} 
    =
    0 .
    \label{e:LT-D-3}  
\end{align}
The first equation holds a.e.\ in $s'$ for $s \le s'$.  The derivative
does not depend on $s$, so it is absolutely continuous in $s$. By
similar reasoning \eqref{e:LT-D-2} holds a.e.\ in $s$ for $s\le s'$,
and as a consequence \eqref{e:LT-D-3} holds a.e.\ in $\{s \le s' \}$.

\subsection{Proof of Proposition~\ref{prop:lace-expansion}}
\label{sec:proof:prop:lace-expansion}

\begin{lemma}
  \label{lem:w-rptn}
  If $Z_{\tb}$ satisfies \ref{as:Z1} and \ref{as:Z2} of
  Assumptions~\ref{hyp:w-rptn} then $\weightL (L)$ is well-defined and
  $\rptnL_{s,t}$ defined by~\eqref{e:rptn-choice} satisfies the
  hypotheses of \Cref{thm:logZ}.
\end{lemma}
\begin{proof}
  The hypothesis that $\rptnL_{s,s}=1$ holds as
  $\V{\tau}_{[s,s]}=\V{0}$.  By \ref{as:Z2} and the compactness of
  $[0,\ell]^{\Lambda }$ the function $Z_{\V{t}}$ is Lipschitz in
  $\V{t}$.
  For each $s$, $\V{\tau}_{[s,t]}$ is absolutely continuous as a
  function of $t$ when $t$ is restricted to a bounded interval, and
  vice-versa by Section~\ref{sec:prel-local-time}. By
  \Cref{lem:Lip-AC} this implies that for each $s$, $\Z{[s,t]}$ is
  absolutely continuous as a function of $t$ when $t$ is restricted to
  a bounded interval and vice-versa. Combined with \ref{as:Z1} this
  proves the first of Assumptions~\ref{hyp:rptn}. Furthermore, by the
  chain rule, the composition $\Z{[s,t]}$ is differentiable in $t$ at
  points $(s,t)$ where $\V{\tau}_{[s,t]}$ has this property.
  Therefore for each $s$, $\Z{[s,t]}$ is differentiable in $t$ at all
  but a countable number of points (recall that simple random walk
  takes only finitely many jumps in any finite time interval), and
  hence is absolutely continuous in $t$, and vice-versa. This verifies
  the second item of \Cref{hyp:rptn}, and an analogous argument
  verifies the third item.  The fourth follows by \ref{as:Z2} and
  \eqref{e:LT-D-3}.

  Lastly, we must prove that $s,t \mapsto \vertex{s,t}$ is a.e.\
  bounded on $\Dcal_{\ell}$. 
  By \eqref{e:vertex-def} we must show that
  $\parto\partt\lgrptn_{s,t}$ is a.e.\ bounded on $\Dcal_{\ell}$.
  By \ref{as:Z2}, $\V{u}\mapsto Z_{\V{u}}$ is $\Ccal^{2}$ on
  $\cb{0,\ell}^{\Lambda}$. By \ref{as:Z1} and the compactness of
  $\Dcal_{\ell}$ the range of $\V{u}\mapsto Z_{\V{u}}$ is bounded away
  from zero. Therefore $F\colon \V{u}\mapsto \log Z_{\V{u}}$ is
  $\Ccal^{2}$ on $\cb{0,\ell}^{\Lambda}$ and
  $\lgrptn_{s,t}=F(\V{\tau}_{[s,t]})$. Let $F^{\sss (1)}_{x}(u)$ be
  the partial derivative of $F(\V{u})$ with respect to $u_{x}$ and let
  $F^{\sss (2)}_{xy}(u)$ be the second partial derivative of
  $F(\V{u})$ with respect to $u_{x}$ and $u_{y}$.  By the chain rule
  and~\eqref{e:LT-D-1}
  $\parto\partt\lgrptn_{s,t} = \parto \big(
  \sum_{x}F^{\sss (1)}_{x}(\V{\tau}_{[s,t]})\indic{X_{t}=x}\big)$. The sum
  over $x\in \Lambda$ is finite and $\indic{X_{t}=x}$ does not depend
  on $s$ so it is sufficient to prove that
  $\parto F^{\sss (1)}_{x}(\V{\tau}_{[s,t]})$ is a.e.\ bounded on
  $\Dcal_{\ell}$. By the chain rule and \eqref{e:LT-D-2}
  $\parto F^{\sss (1)}_{x}(\V{\tau}_{[s,t]})$ is a finite sum over $y$ of
  $F^{\sss (2)}_{xy}(\V{\tau}_{[s,t]})\indic{X_{s}=y}$. Therefore it is
  sufficient to prove that $F^{\sss (2)}_{xy}(\V{\tau}_{[s,t]})$ is
  bounded. By \ref{as:Z2} $F^{\sss (2)}_{xy}$ is continuous on
  $\cb{0,\ell}^{\Lambda}$ and by \Cref{sec:prel-local-time}
  $\V{\tau}_{[s,t]}$ is jointly continuous on
  $\Dcal_{\ell}$. Therefore the composition
  $F^{\sss (2)}_{xy}(\V{\tau}_{[s,t]})$ is continuous on the compact set
  $\Dcal_{\ell}$ and hence bounded as desired.
\end{proof}

\begin{proof}[Proof of \Cref{prop:lace-expansion}]
  We omit the superscript $\Lambda$ throughout, since $\Lambda$ is
  fixed.  At two points in the proof we will use the Markov property;
  the justifications for these applications are given in
  \Cref{sec:markov}.

  By definition  \eqref{e:GF} and \eqref{e:rptn-choice}, 
  \begin{equation}
    \label{e:start}
    G_{\V{0}}(a,b)
    =
    \int_{[0,\infty)} d\ell\;
    \Ea\,\Big[
    \rptn_{0,\ell}
    \indic{X_{\ell} = b}\Big] .
  \end{equation}
  \Cref{lem:w-rptn} implies we can expand $\rptn_{0,\ell}$ by
  Theorem~\ref{thm:logZ}.  Together with the definition
  \eqref{e:gabdef} of $\greens (a,b)$ this yields
  \begin{equation}
    \label{e:start2}
    G_{\V{0}}(a,b)
    =
    \greens (a,b)
    +
    \int_{[0,\infty)} d\ell\;
    \Ea\,\Big[
    \sum_{m \ge 0}\;
    \int_{\laces_{m,\ell}} d\lace\;
    \weight (\lace)\;
    \rptn_{s_{m}',\ell}\;     
    \indic{X_{\ell} = b}\Big] ,
  \end{equation}
  where $s_{m}'$ is defined by the lace $\lace$ as explained in
  \eqref{e:lace-def}. For convenience, define
  $U_{a,b} \bydef G_{\V{0}}(a,b) - \greens(a,b)$.  Using
  \eqref{e:disintegration} and \ref{as:F0} we obtain
  \begin{equation}
    \label{e:prop:lace-expansion2}
    U_{a,b} = 
    \int_{[0,\infty)} ds\;
    \sum_{m \ge 0}\;
    \int_{\laces_{m} (s)} d\lace\;
    \int_{[s_{m}',\infty)} d\ell\;
    \Ea\,\Big[
    \weight (\lace)\;
    \rptn_{s_{m}',\ell}\;     
    \indic{X_{\ell} = b}\Big].
  \end{equation}
  By the change of variable $\ell \mapsto s_{m}'+\ell$ in the integral
  with respect to $\ell$
  \begin{equation}
    \label{eq:unlabel1}
    U_{a,b}
    = 
    \int_{[0,\infty)} ds\;
    \sum_{m \ge 0}\;
    \int_{\laces_{m} (s)} d\lace\;
    \int_{[0,\infty)} d\ell\;
    \Ea\,\Big[
    \weight (\lace)\;
    \rptn_{s_{m}',s_{m}'+\ell}\;     
    \indic{X_{s_{m}'+\ell} = b}\Big] .
  \end{equation}
  For $\ell>0$, let
  $ h (\ell,y,b) \bydef \EEa_{y}\cb{ \rptn_{0,\ell}\; \indic{X_{\ell}
      = b}} $.  By conditioning on $\Fcal_{s_{m}'}$ in the last
  expectation in~\eqref{eq:unlabel1}, using
  $\weight(\lace)\in\Fcal_{s_{m}'}$, integrability by \ref{as:G0}, and
  the Markov property for
  $\EEa_{a} \big[ \rptn_{s_{m}',s_{m}'+\ell}\; \indic{X_{s_{m}'+\ell}
    = b}\big|\mc{F}_{s_{m}'}\big]$,
  \begin{equation}
    \label{e:prop:lace-expansion03}
    U_{a,b} = 
    \int_{[0,\infty)} ds\;
    \sum_{m \ge 0}\;
    \int_{\laces_{m} (s)} d\lace\;
    \int_{[0,\infty)} d\ell\;
    \Ea\,\Big[
    \weight (\lace)\;
    h (\ell,X_{s_{m}'},b)
    \Big] .
  \end{equation}
  By \ref{as:G0} and \ref{as:F0} the right-hand side converges
  absolutely and likewise for the following equations.  We bring the
  integral with respect to $\ell$ inside the expectation and rewrite
  $\int_{[0,\infty)} d\ell\; h (\ell,X_{s_{m}'},b)$ using the
  definition \eqref{e:start} of $G_{\V{0}}(a,b)$:
  \begin{equation}
    U_{a,b} =   
    \int_{[0,\infty)} ds\;
    \sum_{m \ge 0}\;
    \int_{\laces_{m} (s)} d\lace\;
    \Ea\,\Big[
    \weight (\lace)\;
    G_{\V{0}}(X_{s_{m}'},b)
    \Big] .
  \end{equation}
  By changing variables in the integral over $\laces_{m} (s)$ so that
  it becomes an integral over $\laces_{m} (0)$ we rewrite this as
  \begin{align}
    U_{a,b} =   
    \int_{[0,\infty)} ds\;
    \sum_{m \ge 0}\;
    \int_{\laces_{m} (0)} d\lace\;
    \Ea\,\Big[
    \weight (\lace + s)\; G_{\V{0}}(X_{s_{m}' + s},b)
    \Big] ,
    \label{e:prop:lace-expansion05}
  \end{align}
  where for $\lace \in \laces_m(0)$, $\lace+s$ is defined to be the
  lace in $\laces_m(s)$ obtained from $\lace $ by adding $s$ to each
  $s_i,s'_i$. For $\lace \in \laces_{m} (0)$ define
  $ f (x,b,\lace) \bydef \EEa_{x}\Big[ \weight (\lace)\;
  G_{\V{0}}(X_{s_{m}'},b) \Big]$ .  By conditioning on $\mc{F}_{s}$
  inside the expectation in \eqref{e:prop:lace-expansion05} and
  applying the Markov property to 
  $\Ea\,\Big[
    \weight (\lace + s)\; G_{\V{0}}(X_{s_{m}'+s},b)
   \big|\mc{F}_s \Big ]$ 
  we obtain
  \begin{align}
    U_{a,b} &= 
    \int_{[0,\infty)} ds\;
    \sum_{m \ge 0}\;
    \int_{\laces_{m} (0)} d\lace\;
    \Ea\,\big[
    f (X_{s},b,\lace)
    \big].
    \label{e:prop:lace-expansion1}
  \end{align}
  The expectation is equal to
  \begin{align}
    &\sum_{x,y \in \Lambda}
      \Ea\,\Big[E_{X_s}\big[w(L)G_{\V{0}}(X_{s_{m}'},b)\indic{X_{s_{m}'}=y}\big] 
      \indic{X_{s}=x}  
      \Big]
      \nonumber\\
    &\hspace{1cm}
      =\sum_{x,y \in \Lambda} G_{\V{0}}(y,b)\,\Ea\,\big[ \indic{X_{s}=x}  
      \big]E_x[w(L)\indic{X_{s_{m}'}=y}],
  \end{align}
  where we have used the fact that the sums over $x,y \in \Lambda$ are
  finite to take them outside the expectation in the first line.
  Recalling the definition \eqref{e:Pi-def} we see that
  \eqref{e:prop:lace-expansion1} can be written as
  \begin{align}
    U_{a,b} &= 
    \int_{[0,\infty)} ds\;
    \sum_{x,y \in \Lambda} G_{\V{0}}(y,b)\,\Ea\,\big[ \indic{X_{s}=x}  
    \big]\;
    \Pi(x,y).
\end{align}
By the definitions of $\greens (x,y)$ and $U_{a,b}$ this is the same
as
  \begin{equation*}
    G_{\V{0}}(a,b)
    =
    \greens (a,b)\;
    + 
    \sum_{x,y \in \Lambda}\;
    \greens (a,x)
    \Pi (x,y)
    G_{\V{0}}(y,b).\qedhere
  \end{equation*}
\end{proof}

\section{The terms \texorpdfstring{$\PiL_{m}$}{Pim}
of the lace expansion}
\label{sec:Pi-DB-outer}

Throughout this section $\Lambda\subset \ZZ^{d}$ is a fixed finite
set. Recall the definitions below \eqref{e:rptn-choice} and define
\begin{equation}
  \label{eq:Pi-m}
  \PiL_{m}(x,y) \bydef
  \int_{\laces_{m}(0)}d\lace \; \EEa_{x} 
  \left[ \weightL(\lace) \indic{\LX_{s_{m}'}=y}\right],
  \qquad
  m\geq 0.
\end{equation}
Thus $\PiL_{m}$ is the $m^\textrm{th}$ term in the series
\eqref{e:Pi-def} that defines $\PiL(x,y)$.  This section has two
parts. The first provides formulas for the weights $\weightL$, and the
second derives bounds on $\PiL_{m}$ for $m\geq 1$. The main result is
\Cref{prop:Pi-DB}, which bounds $\PiL_{m}$ in terms of $\GrL_{\V{0}}$;
these bounds are used in implementing step one of \Cref{sec:step-one}.

\subsection{Formulas for weights}
\label{sec:formulas-weights}

We give formulas for $\PiL_{0}$ and the factors $\vertexL{}$ that
enter into $\weightL$. Both computations are applications of the chain
rule to our choice~\eqref{e:rptn-choice} of $\rptn$ together with the
formulas~\eqref{e:LT-D-1}--\eqref{e:LT-D-3} for derivatives of the
local time. The formulas of this section are valid under
\Cref{hyp:w-rptn}. 

\subsubsection{The term \texorpdfstring{$\PiL_0$}{PiL0}}
\label{sec:PiL0}

Recall the definition \eqref{eq:selfloop-def} of
$\selfloopL{\cpl,\nu,x}$ where $Z^{\sssL}$ is now the function
entering in the definition~\eqref{e:rptn-choice} of $\rptn$.  At this
level of generality $Z^{\sssL}$ need not depend on $\cpl,\nu$, but we
retain them in our notation. The limit $\selfloopL{\cpl,\nu,x}$ exists
and is finite by \ref{as:Z1} and \ref{as:Z2}; see below
\eqref{e:vertex-def}.  The next result shows that
$\selfloopL{\cpl,\nu,x}$ is essentially the first (zeroth) term in the
finite volume lace expansion.

\begin{lemma}
  \label{lem:Selfloop}
  For all finite $\Lambda$,
  $ \PiL_{0}(x,y) = \selfloopL{\cpl,\nu,x} \indic{x=y}$.
\end{lemma}
\begin{proof}
  From \eqref{eq:Pi-m},
  \begin{equation}
    \label{eq:SL}
    \PiL_{0}(x,y) = \int_{\laces_{0}(0)}d\lace\, \EEa_{x} \left[
      \vertexL{s_0}\indic{X^{\sssL}_{s'_{0}}=y}\right].
  \end{equation}
  By definition, $d\lace$ for $\lace \in \laces_{0}(0) $ is a unit
  mass at $s_{0}=0$, and by definition $s_{0}'=s_{0}$.  Moreover,
  $X^{\sssL}_{0}=x$ under the measure $E_x$.  Hence \eqref{eq:SL}
  becomes
\begin{equation}
    \label{eq:SL2}
    \PiL_{0}(x,y) =  
\indic{x=y}\EEa_{x} \left[
      \vertexL{0}\right] 
      \end{equation}
   By the definition \eqref{e:vertex-def} of $\vertexL{0}$,
  \eqref{e:rptn-choice}, and $\partial_{t_{x}}Z^{\sssL}_{\V{0}}=0$, 
  \begin{align}
  \label{eq:V0}
    \vertexL{0} 
    &= -\lim_{s'\downarrow 0}\sum_{z\in\Lambda} \partial_{t_{z}}
      \log Z^{\sssL}_{\V{t}}
                     \big|_{\V{t}=\V{\tau}^{\sssL}_{[0,s']}}\, 
    \partial_{1} \tau^{\sssL}_{[0,s'],z} \\
  \label{eq:V1}
    &= \sum_{z\in\Lambda}\partial_{t_{z}}\log
      Z^{\sssL}_{\V{t}}
      \Big|_{\V{t}=\V{0}} \indic{X^{\sssL}_{0}=z}, \qquad \text{a.s.}
  \end{align}
  In obtaining \eqref{eq:V1} we used~\eqref{e:LT-D-2} and the
  right-continuity of the random walk.  Since $X^{\sssL}_{0}=x$ a.s.\
  under $\EEa_{x}$, the lemma follows by inserting the
    definition \eqref{eq:selfloop-def} of $\selfloopL{\cpl,\nu,x}$.
\end{proof}

\subsubsection{The vertex weight $\vertexL{}$}
\label{sec:PiL0vertex}

Recall from \eqref{e:vertex-def} that
$\vertexL{u,v} = \vertex{u,v} \bydef -\parto\partt\lgrptn_{u,v}$.
Define, for $x,y\in\Lambda$ and $u<v$,
\begin{equation}
  \label{eq:vweightc0}
  \vertexL{u,v} (x,y)
    \bydef
      \partial_{t_{x}} \partial_{t_{y}} \log Z^{\sssL}_{\V{t}}
       \big|_{\V{t}=\V{\tau}^{\sssL}_{[u,v]}}.
\end{equation}

\begin{lemma}
  \label{lem:vweightc}
  For all finite $\Lambda$ and all $u<v$,
  \begin{equation}
    \label{eq:vweightc}
    \vertexL{u,v} = 
    \sum_{x,y\in\Lambda}\vertexL{u,v}(x,y)\indic{\LX_{u}=x}\indic{\LX_{v}=y}.
  \end{equation}
\end{lemma}
\begin{proof}
  This follows from a calculation similar to the proof of \Cref{lem:Selfloop}.     
\end{proof}

\subsection{Bounds on \texorpdfstring{$\PiL_{m}$}{Pim}, \texorpdfstring{$m\geq 1$}{m>=1}}
\label{sec:Pi-DB}

Our bounds on $\PiL_{m}$ for $m\geq 1$ will rely on two
assumptions. Recall the definition~\eqref{eq:vweightc0} of
$\vertexL{u,v}(x,y)$.

\begin{assumptions}
  \label{hyp:G}
  \leavevmode
  \begin{enumerate}[label=H\arabic*, ref=H\arabic*,series=hyp]
  \item[\mylabel{as:G1}{(G1)}]
    For all $\V{t}\in[0,\infty)^{\Lambda}$, $\GrL_{\V{t}}\leq \GrL_{\V{0}}$. 
  \item[\mylabel{as:R1}{(R0)}] 
    There exists
    $\barvertL{}\colon \Lambda\times\Lambda\to\RR$ such that
    $\abs{\vertexL{u,v}(x,y)}\leq \barvertL{}(x,y)$
    for all $x,y\in\Lambda$ and $0\leq u<v<\infty$.
  \end{enumerate}
\end{assumptions}

Given vertices $x$ and $y$
in $\Lambda$, and $m\ge 1$, define
\begin{equation}
  \label{e:x-seq}
 \Lambda^{2m-2}_{x,y}\bydef \{((x_{i},x_{i}'))_{i=0,\dots, m}\in \Lambda^{2(m+1)}\, | \, x_{0}=x_{0}' = x \text{ and }
   x_{m}=x_{m}'=y\}. 
\end{equation}
Generic elements of $ \Lambda^{2m-2}_{x,y}$ will be denoted by
$\xb=((x_{i},x_{i}'))_{i=0,\dots, m}$.

\begin{figure}[]
  \centering
  \begin{tikzpicture}[baseline=8mm]
    \node[dot] (v0) at (0,0) {};
    \node[dot] (v1) at (2,0) {};
    \node[dot] (v0p) at (2,2) {};
    \node[dot] (v2) at (4,2) {};
    \node[dot] (v1p) at (4,0) {};
    \node[dot] (v3) at (6,0) {};
    \node[dot] (v2p) at (6,2) {};
    \node[dot] (v4) at (8,2) {};
    \node[dot] (v3p) at (8,0) {};
    \node[dot] (v5) at (10,0) {};
    \node[dot] (vmp) at (v0) {};
    \node[dot] (v5p) at (v5) {};

    \node at (vmp) [below] {$x$};
    \node at (v1) [below] {$x_{1}$};
    \node at (v0p) [above] {$x_{1}'$};
    \node at (v2) [above] {$x_{2}$};
    \node at (v1p) [below] {$x_{2}'$};
    \node at (v3) [below] {$x_{3}$};
    \node at (v2p) [above] {$x_{3}'$};
    \node at (v4) [above] {$x_{4}$};
    \node at (v3p) [below] {$x_{4}'$};
    \node at (v5p) [below] {$y$};

    \draw[black, very thick] (v0) -- (v1) -- (v0p) -- (v2) -- (v1p) -- (v3) --
    (v2p) -- (v4) -- (v3p) -- (v5);
    \draw[black,thick, decorate, decoration={zigzag,segment
      length = 5, amplitude=1}] (vmp) to (v0p);  
    \draw[black,thick, decorate, decoration={zigzag,segment
      length = 5, amplitude=1}] (v1) to (v1p);  
    \draw[black,thick, decorate, decoration={zigzag,segment
      length = 5, amplitude=1}] (v2) to (v2p); 
    \draw[black,thick, decorate, decoration={zigzag,segment
      length = 5, amplitude=1}] (v3) to (v3p);
    \draw[black,thick, decorate, decoration={zigzag,segment
      length = 5, amplitude=1}] (v4) to (v5p);
  \end{tikzpicture}
  \caption{The upper bound on $\PiL_{5}(x,y)$ from
      \Cref{prop:Pi-DB}. All vertices except $x$ and $y$ are summed
    over $\Lambda$. Lines connecting vertices represent functions:
    wavy lines represent $\barvertL{}$ and straight lines represent
    $G_{\V{0}}^{\sssL}$.}
  \label{fig:Pi-Bound}
\end{figure}

\begin{proposition}
  \label{prop:Pi-DB}
  Suppose \ref{as:Z1}--\ref{as:Z2}, \ref{as:G0}, and \Cref{hyp:G}
  hold. For $m\geq 1$ and $x,y\in\Lambda$,
  \begin{align}
    \nn
    \abs{\PiL_{m}(x,y)}  \leq
    &
      \sum_{\xb \in \Lambda_{x,y}^{2m-2}} 
      \GrL_{\V{0}}(x,x_{1}) \barvertL{}(x,x_{1}') \\ \label{eq:Pi-DB}
    &\times
    \prod_{j=1}^{m-1}\GrL_{\V{0}}(x_{j},x_{j}') \GrL_{\V{0}}(x_{j}',x_{j+1})
    \barvertL{}(x_{j},x_{j+1}') .
  \end{align}
\end{proposition}
See \Cref{fig:Pi-Bound} for a diagrammatic representation of the upper
bound, which explains our use of the term \emph{lace expansion}: the
upper bound is of exactly the form that occurs in discrete-time lace
expansion analyses of self-avoiding walk~\cite{BrSp85a,Slade}.  In
more detail, for the Edwards model, a computation (see
\eqref{e:WSAW-LV}) shows that we can choose $\barvertL{}$ to be a
constant times $\indic{x=y}$. This amounts to shrinking the wavy edges
in \Cref{fig:Pi-Bound} to points, and these are the diagrams occurring
in~\cite{BrSp85a,Slade}.

The next two subsections prove \Cref{prop:Pi-DB}. As $\Lambda$
is fixed it will be omitted from the notation.

\subsubsection{A preparatory lemma} 

Recall the definition \eqref{e:rptn-choice} of $\rptn_{s,t}$ and define
\begin{equation}
  \label{e:brptn-def}
  \brptn_{u,v}(w)
  \bydef
  \frac{\rptn_{u,w}}{\rptn_{u,v}}, \qquad u \le v \le w.
\end{equation}
By the definitions \eqref{e:weight} and \eqref{e:Pweight} of
$\weight (\lace)$ and $\Pweight (\lace)$ for
$\lace \in \laces_{m}$ with $m \ge 1$,
\begin{equation}
    \label{e:Pweight2}
    \begin{aligned}
      \weight (\lace) &= \vertex{} (\lace) \Pweight (\lace) ,
      \\
      \Pweight (\lace) &= \prod_{i=-1}^{m-2}
      \brptn_{s_{i}',s_{i+1}'}(s_{i+2}'), \hspace{10mm} s_{-1}' \bydef
      s_{0}'.
    \end{aligned}
\end{equation}
The term $\brptn_{s_{-1}',s_{0}'}(s_{1}')$ in the product is the
factor $\rptn_{s_{0}',s_{1}'} $ in \eqref{e:Pweight} (recall that
$\rptn_{s'_0,s'_0}=1$). For $k =1,2$ define $\Pweight_{-k} (\lace)$ by
replacing the upper limit $m-2$ in \eqref{e:Pweight2} by $m-2-k$. By
convention empty products are defined to be one.

\begin{lemma}
  \label{lem:induct} 
  Let $0 \le u_{1}\le u_{2} \le u_{3}$, and let $H\ge 0$ be
  $\Fcal_{u_{3}}$-measurable.  Then almost surely
\begin{equation}
  \label{e:lem:induct}
  \int_{\co{u_{3},\infty}}d\ell\, \EEa_{x} \left[ H\, 
    \brptn_{u_{1},u_{2}}(\ell)\,
    \indic{X_{\ell}=y} \Big| \,\Fcal_{u_{3}}\right]
  = 
  H\,
  \brptn_{u_{1},u_{2}}(u_{3})\,
  \G{[u_{1},u_{3}]}(X_{u_{3}},y)
\end{equation}
\end{lemma}

\begin{proof}
  By the definition~\eqref{e:brptn-def} of 
  $\brptn_{u_{1},u_{2}}(u_{3})$ 
  \begin{equation}
  	\label{e:brptn-1}
      \brptn_{u_{1},u_{2}}(\ell)
      =
      \frac{\rptn_{u_{1},\ell}}{\rptn_{u_{1},u_{2}}}
      =
      \frac{\rptn_{u_{1},u_{3}}}{\rptn_{u_{1},u_{2}}}
      \frac{\rptn_{u_{1},\ell}}{\rptn_{u_{1},u_{3}}}
      =
      \brptn_{u_{1},u_{2}}(u_{3})
      \brptn_{u_{1},u_{3}}(\ell) .
  \end{equation}
  Insert~\eqref{e:brptn-1} into the left-hand side
  of~\eqref{e:lem:induct}. Using the nonnegativity of $H$,
    $\brptn_{u_{1},u_{2}}$ and $\brptn_{u_{1},u_{3}}$ we take the
  $\Fcal_{u_{3}}$-measurable factor $H\, \brptn_{u_{1},u_{2}}(u_{3})$
  outside the conditional expectation and the integral of what remains
  is
  \begin{equation}
    \label{eq:Int-t2}
    \int_{\co{u_{3},\infty}}d\ell\, \Ea \left[
      \brptn_{u_{1},u_{3}}(\ell) 
     \indic{X_{\ell}=y} \big| \Fcal_{u_{3}}\right]
    =
    \G{[u_{1},u_{3}]}(X_{u_{3}},y) \qquad\text{a.s.,}
   \end{equation}
   by the Markov property as stated in \Cref{lem:clemma}.  
\end{proof}

\subsubsection{Proof of Proposition~\ref{prop:Pi-DB}}

Before giving the proof of \Cref{prop:Pi-DB}, we recall the following
consequence of the Fubini--Tonelli theorem that will be used in the
proof. If $X_{u}$ is a real-valued stochastic process satisfying
$\EEa [\int_{[a,b]}du\,|X_{u}|] = \int_{[a,b]}du\,\EEa [|X_{u}|]
<\infty$, then integration and conditional expectation can be
interchanged:
\begin{equation}
  \int_{[a,b]} du\,\EEa \left[ X_{u}\mid \mathcal{G}
  \right] = \EEa \left[
    \int_{[a,b]}du\, X_{u}\mid \mathcal{G}
  \right], \quad \textrm{ a.s.} \label{int_cond}
\end{equation}
\begin{proof}[Proof of \Cref{prop:Pi-DB}]
Let $\lace\in\laces_{m} (0)$. Given a sequence
$\xb \in \Lambda^{2m-2}_{x,y}$ as in \eqref{e:x-seq} and a time
$u\in\co{0,\infty}$ define the indicator function
\begin{equation}
  \label{eq:filter}
  \filter_{\lace,\xb,u} \bydef \prod_{j:s_{j}\leq u}
  \indic{X_{s_{j}}=x_{j}} \prod_{j':s_{j'+1}'\leq u}
  \indic{X_{s_{j'+1}'}=x_{j'+1}'} 
\end{equation}
of the event that the path $X$ is at the points $(x_{i},x_{i+1}')$ at
the times $(s_{i},s_{i+1}')$ up to $u$ in $\lace =
((s_{i},s_{i+1}'))_{i=0,\dots, m-1}$.  See Figures~\ref{fig:six-lace}
and~\ref{fig:Pi-Bound} and think of the solid lines in the latter
figure as a representation of paths $X$ with $X_{0}=x$,
$X_{s_{1}}=x_{1}$, $X_{s_{1}'}=x_{1}'$, etc.  For $y \in \Lambda$ we
have
\begin{equation}
    \label{e:filter-sum}
    \indic{X_{0}=x}\indic{X_{s'_{m}=y}}  =
    \sum_{\xb \in \Lambda^{2m-2}_{x,y}}\filter_{\lace,\xb,s_{m}'}
\end{equation}
since $X$ cannot be at the absorbing state $\ast$ at times earlier
than $s_{m}'$ on the event $\{X_{s'_{m}}=y\}$. We have also used that
$\{X_{0}=x\}=\{X_{s_{0}}=x\}$ since $\lace\in\laces_{m}(0)$.  Define
\begin{equation}
  \label{eq:NGN}
  \ngn_{m,\xb}
  \bydef
  \int_{\laces_{m}(0)} d\lace\, \EEa_{x} \Big[
  \filter_{\lace,\xb,s_{m}'} \,
  \Pweight (\lace)  
  \Big].
\end{equation}
Since $\indic{X_{0}=x}=1$ a.s.\ under $\EEa_{x}$, we can insert
\eqref{e:filter-sum} into the definition \eqref{eq:Pi-m} of $\Pi_{m}$
to obtain
\begin{align}
  \label{eq:Pi-F-1a}
  \abs{\Pi_{m}(x,y)} 
  &= \Big|
    \sum_{\xb\in\Lambda^{2m-2}_{x,y}} \int_{\laces_{m}(0)} d\lace\, \EEa_{x} \Big[
  \filter_{\lace,\xb,s_{m}'} \,
  \vertex{}(\lace)\Pweight (\lace)  
  \Big]\Big|
  \\\label{eq:Pi-F-1}
  &\leq \sum_{\xb \in \Lambda^{2m-2}_{x,y}}
  \prod_{j=0}^{m-1}\barvert{}(x_{j},x_{j+1}') \ngn_{m,\xb},
\end{align}
where we have used the triangle inequality and \ref{as:R1} to bound
the vertex functions in $\vertex{}(\lace)$, and $\Pweight(\lace)>0$ to
remove absolute values.  This reduces \Cref{prop:Pi-DB} to proving
\begin{equation}
  \label{eq:Pi-DB-NoV}
  \ngn_{m,\xb}
  \leq
  G_{\V{0}}(x_{0},x_{1}) 
  \prod_{j=1}^{m-1}G_{\V{0}}(x_{j},x_{j}') G_{\V{0}}(x'_{j},x_{j+1}),
\end{equation}
for $m\geq 1$ and $\xb\in\Lambda^{2m-2}_{x,y}$, which we will do by
induction on $m$. The base case $m=1$ follows by noting that
$\filter_{\lace,\xb,s_{1}'}=\indic{X_{s_{1}'}=y}$ under $\EEa_x$ and
recalling that $\int_{\laces_{1}(0)}\,d\lace =
\int_{0}^{\infty}\,ds_{1}'$, so $\ngn_{1,\xb}=G_{\V{0}}(x,y)$.

Suppose~\eqref{eq:Pi-DB-NoV} holds when $m=n$ for some $n\geq 1$.
By \eqref{e:lace-def} with $s_{0}'=0$, for
$\lace\in\laces_{n+1}(0)$ the measure $d\lace$ factorizes as
$d\lace' \,ds_{n}'\, ds_{n+1}$, where $d\lace'$ is Lebesgue measure
on
\begin{equation}
    \label{e:proof:prop:Pi-DB}
    \laces_{n}'(0)
    =
    \{(s_{1},s_{1}',s_{2},
    	\dots, s_{n-1},s_{n-1}',s_{n}) \mid
                 0<s_{1}<s_{1}'<
                 	\dots <s_{n-1}'<s_{n}\} 
\end{equation}
and $ds_{n}'\, ds_{n+1}$ is Lebesgue measure on the set of
$(s_{n}',s_{n+1})$ such that $s_{n}<s_{n}'<s_{n+1}$. Rewriting
$\ngn_{n+1,\xb}$ using this factorization yields
\begin{equation}
    \label{e:ngn-1}
    \ngn_{n+1,\xb}
    =
    \int_{\laces_{n}'(0)}d\lace'\,
    \int_{[s_{n},\infty)} ds_{n}'\,
    \int_{[s_{n}',\infty)} ds_{n+1}\,
    \EEa_{x} \Big[
    \filter_{\lace,\xb,s_{n+1}}
    \Pweight (\lace)
  \Big].
\end{equation}
The induction step involves estimating the integrals over $s_{n+1}$
and $s_{n}'$ by \Cref{lem:induct} and \ref{as:G1}. To bound the
$s_{n+1}$ integral note the range of integration starts at $s_{n}'$
and accordingly insert a conditional expectation with respect to
$\Fcal_{s_{n}'}$ under the expectation $\EEa_{x}$.  Bringing the
$s_{n+1}$ integral inside the expectation yields
\begin{align}
  \label{e:ngn-2a}
  \ngn_{n+1,\xb}
  &=
  \int_{\laces_{n}'(0)}d\lace'\,
  \int_{[s_{n},\infty)}ds_{n}'\,
  \EEa_{x} \Big[
  \ti_{\lace,\xb,s_{n}'}
  \Big] , \quad \textrm{where}
  \\
  \label{e:ngn-2b}  
  \ti_{\lace,\xb,s_{n}'}
  &\bydef 
  \int_{\co{s_{n}',\infty}}ds_{n+1}\, \EEa_{x} \Big[
    \filter_{\lace,\xb,s_{n+1}}
    \Pweight (\lace)
    \Big| \Fcal_{s_{n}'}\Big].
\end{align}
Recall that $\Pweight_{-1} (\lace)$ was defined below
\eqref{e:Pweight2}, and note that
\begin{equation}
    \filter_{\lace, \xb, s_{n+1}}
    =
    \filter_{\lace,\xb,s_{n}'} \indic{X_{s_{n+1}} = \,x_{n+1}} , 
    \hspace{6mm}
    \Pweight (\lace)
    =
    \Pweight_{-1} (\lace)
    \brptn_{s_{n-1}',s_{n}'}(s_{n+1}).
\end{equation}
We insert these identities into $\ti_{\lace,\xb,s_{n}'}$ and apply
\Cref{lem:induct} with
$(u_{1},u_{2},u_{3}) = (s_{n-1}',s_{n}',s_{n}')$ and
$H = \filter_{\lace,\xb,s_{n}'}\,\Pweight_{-1} (\lace)$. Since
$H\geq 0$, after using \ref{as:G1} with
$\V{t}=\V{\tau_{[s_{n-1}',s_{n}']}}$ we obtain
\begin{equation}
  \label{e:ngn-3m}
  \ti_{\lace,\xb,s_{n}'}
  \le
  \filter_{\lace,\xb, s_{n}'} \,
  \Pweight_{-1} (\lace)\,
  G_{\V{0}}(x_{n}',x_{n+1})
\end{equation}
because $\brptn_{s_{n-1}',s_{n}'}(s_{n}') = 1$. Hence by
\eqref{e:ngn-2a} and that $x_{n+1}'=x_{n+1}$ by \eqref{e:x-seq}, 
\begin{equation}
  \label{e:ngn-3}
  \ngn_{n+1,\xb}
  \le
  \int_{\laces_{n}'(0)}d\lace'\,
  \int_{\co{s_{n},\infty}}ds_{n}'\,
  \EEa_{x} \Big[
  \filter_{\lace,\xb,s_{n}'} \,
  \Pweight_{-1} (\lace)\,
  \Big]\,
  G_{\V{0}}(x_{n}',x_{n+1}).
\end{equation}
For the
$s_{n}'$ integral in~\eqref{e:ngn-3} the procedure is similar so we
will be brief.  Insert a conditional expectation with respect to
$\Fcal_{s_{n}}$ under the expectation in \eqref{e:ngn-3}, bring the
integral over $s_{n}'$ inside $\EEa_{x}$, and then insert
\begin{equation}
    \filter_{\lace,\xb,s_{n}'}
    =
    \filter_{\lace,\xb,s_{n}}\,\indic{X_{s_{n}'}=\,x_{n}'} ,
    \hspace{6mm}
    \Pweight_{-1} (\lace)
    =
    \Pweight_{-2} (\lace)
    \brptn_{s_{n-2}',s_{n-1}'}(s_{n}') .
\end{equation}
We apply \Cref{lem:induct} with $(u_{1},u_{2},u_{3}) =
(s_{n-2}',s_{n-1}',s_{n})$ and $H=\filter_{\lace,\xb,s_{n}}$, with the
result, again after using \ref{as:G1},
\begin{align}
  \label{e:ngn-4}
  \ngn_{n+1,\xb}
  &\le
  \int_{\laces_{n}'(0)}d\lace'\,
  \EEa_{x} \Big[
  \filter_{\lace,\xb,s_{n}} \,
  \Pweight_{-2} (\lace)\,
  \brptn_{s_{n-2}',s_{n-1}'}(s_{n}) 
  \Big]\,
  \nn \\
  &\qquad \times 
  G_{\V{0}}(x_{n},x_{n}')
  G_{\V{0}}(x_{n}',x_{n+1}).
\end{align}
By \eqref{e:Pweight2}
  $ \Pweight_{-2} (\lace)\, \brptn_{s_{n-2}',s_{n-1}'}(s_{n}) $ equals
  $\Pweight (\lace')\mid_{s_{n}'=s_{n},x_{n}'=x_{n}}$.
  By~\eqref{e:proof:prop:Pi-DB} the measure spaces
  $(\laces'_n(0),d\lace')$ and $(\laces_{n}(0),d\lace)$ are the
  same. Therefore
\begin{align}
  \label{e:ngn-5}
  \ngn_{n+1,\xb}
  \le\left(
  \int_{\laces_{n}(0)}d\lace\,
  \EEa_{x} \Big[
  \filter_{\lace,\xb,s_{n}} \,   
  \Pweight (\lace)\,
  \Big]_{s_{n}'=s_{n},x_{n}'=x_{n}}\right)
  G_{\V{0}}(x_{n},x_{n}')
  G_{\V{0}}(x_{n}',x_{n+1}).
\end{align}
By~\eqref{eq:NGN} the quantity in brackets is $\ngn_{m,\xb}$ with
$m=n$. Applying the inductive hypothesis~\eqref{eq:Pi-DB-NoV} to this
term completes the proof.
\end{proof}

\section{Preparation for the infinite volume limit}
\label{sec:Bounds-2}

This section continues with step one of \Cref{sec:step-one}. The main
result is \Cref{cor:aPi-bound}, which is a bound on $\aPiL$, which is
the finite volume version of the term $\aPi$ in \eqref{eq:Int-Conv}. A
crucial aspect of the bound given by \Cref{cor:aPi-bound} is that it
is uniform in $\Lambda$.  The bound relies on \Cref{hyp:G2}, which
play a continuing role in the remainder of the paper.

\subsection{Convolution estimates}
\label{sec:Convolution-Estimates}

Recall that $\mnorm{x} = \max\{\abs{x},1\}$. The next lemma says that
when the sum over $w\in\ZZ^{d}$ is sufficiently convergent there is a bound as
if $w=0$. 
\begin{lemma}[Equation~(4.17) of~\cite{HHS2003}]
	\label{lem:HHS-1}
	Let $d\geq 5$, $u,v\in \ZZ^d$. There exists a $C>0$ such that 
    \begin{equation}
    \nonumber
	\sum_{w\in\ZZ^d} \mnorm{w}^{4-2d} \mnorm{w-v}^{2-d}
        \mnorm{w-u}^{2-d} 
        \leq C \mnorm{v}^{2-d}\mnorm{u}^{2-d}.
	\end{equation}
\end{lemma}

The next estimate says the convolution of two functions decays
according to whichever has the weaker decay, provided at least one of
them is integrable.
\begin{lemma}[Proposition~1.7(i) of~\cite{HHS2003}]
  \label{lem:HHS-2}
  Let $f,g\colon \ZZ^d \to \RR$ be such that
  $|f(x)| \leq \mnorm{x}^{-a}$, $|g(x)|\leq \mnorm{x}^{-b}$,
  $a\geq b>0$. There exists a $C>0$ such that
  \begin{equation}
    \nonumber
    |(f\ast g)(x)| \leq
    \begin{cases}
      C \mnorm{x}^{-b}, & a>d, \\
      C\mnorm{x}^{d-(a+b)}, & \textrm{$a<d$ and $a+b>d$}.
    \end{cases}
  \end{equation}
\end{lemma}

\Cref{fig:Conv-E} gives a diagrammatic representation of the next lemma.
\begin{figure}[]
  \centering
  \begin{equation}
	\nonumber
  \begin{tikzpicture}[baseline=8mm]
    \node[dot] (v3) at (6,0) {};
    \node[dot] (v2p) at (6,2) {};
    \node[dot] (v4) at (8,2) {};
    \node[dot] (v3p) at (8,0) {};
    \node[dot] (v5) at (10,0) {};

    \node at (v3) [below] {$v$};
    \node at (v2p) [above] {$u$};
    \node at (v4) [above] {$a$};
    \node at (v3p) [below] {$b$};

    \node at (v5) [below right] {$y$};

    \draw[black, very thick] 
    (v2p) -- (v4) -- (v3p) -- (v5);

    \draw[black,very thick] (v3) to[out=15, in=165]   (v3p);
    \draw[black,very thick] (v3) to[out=-15, in=195] (v3p);
    \draw[black,very thick] (v4) to[out=330,in=120]  (v5);
   \draw[black,very thick] (v4) to[out=300,in=150] (v5);
  \end{tikzpicture}
  \leq C \hspace{3mm}
    \begin{tikzpicture}[baseline=8mm]
    \node[dot] (v3) at (6,0) {};
    \node[dot] (v2p) at (6,2) {};
    \node[dot] (v5) at (10,0) {};

    \node at (v3) [below] {$v$};
    \node at (v2p) [above] {$u$};

    \node at (v5) [below right] {$y$};

    \draw[black, very thick]  (v2p) -- (v5);

    \draw[black,very thick] (v3) to[out=15, in=165]  (v5);
    \draw[black,very thick] (v3) to[out=-15, in=195] (v5);
  \end{tikzpicture}
\end{equation}
\caption{A diagrammatic depiction of \Cref{lem:Conv-E}. Solid lines
  represents factors $\mnorm{x_{2}-x_{1}}^{2-d}$.  The vertices
  $u,v,y$ are fixed, but $a$ and $b$ are summed over $\ZZ^{d}$. }
\label{fig:Conv-E}
\end{figure}

\begin{lemma}
\label{lem:Conv-E}
Fix $u,v,y\in \ZZ^d$, $d\geq 5$. There exists a $C>0$ such that
\begin{align}
    \label{eq:Bound-1}
  &\hspace{-10mm}\sum_{a,b\in \ZZ^d} \mnorm{u-a}^{2-d}
    \mnorm{y-a}^{4-2d} \mnorm{b-a}^{2-d} \mnorm{v-b}^{4-2d} \mnorm{y-b}^{2-d} 
  \\\nonumber
  &\leq C \mnorm{y-u}^{2-d} \mnorm{y-v}^{4-2d}.
	\end{align}
\end{lemma}
\begin{proof}
  Lemma~\ref{lem:HHS-1} can be applied to the sum over $a$ to upper
  bound the left-hand side of \eqref{eq:Bound-1} as if $a=y$,
  that is, by a constant $C>0$ times
  \begin{equation}
    	\nonumber
		\mnorm{y-u}^{2-d} \sum_{b\in \ZZ^d} \mnorm{y-b}^{2-d}
                \mnorm{v-b}^{4-2d} \mnorm{y-b}^{2-d}.
  \end{equation}
  The sum over $b$ is a convolution of two functions that decay at
  rate $\gamma = 2d-4$.  As $\gamma$ exceeds $d$ when $d\geq 5$,
  Lemma~\ref{lem:HHS-2} implies the claimed upper bound.  \qedhere
\end{proof}

\subsection{\texorpdfstring{$\aPiL$}{aPIL} and uniform bounds on \texorpdfstring{$\aPiL$}{Psi}}
\label{sec:Pi-Bounds}

\begin{assumptions}
\label{hyp:G2}
For all $\Lambda\subset\ZZ^{d}$ finite,
\begin{enumerate}
\item[\mylabel{as:G2}{(G2)}]
  For $\Lambda'\subset\Lambda$ and $x,y\in\Lambda'$,
  $G^{\sss(\Lambda')}_{\V{0}}(x,y)\leq \GrL_{\V{0}}(x,y)$;
\item[\mylabel{as:R1p}{(R1)}] 
  There exists $\eta>0$
  independent of $\Lambda$, such that for $0\leq u<v<\infty$
  \begin{equation}
    \label{eq:R1p}
    \abs{\vertexL{u,v}(x,y)}
    \leq \eta
    (\indic{x=y} +
    \GrL_{\V{0}}(x,y)^{2}), \qquad x,y\in\Lambda.
  \end{equation} 
\end{enumerate}
\end{assumptions}
The assumption \ref{as:R1p} supersedes \Cref{as:R1}\ref{as:R1} by
stipulating the specific form
\begin{equation}
  \label{eq:barvertL}
  \barvertL{}
  =
  \eta
  (\indic{x=y} +
  \GrL_{\V{0}}(x,y)^{2})
\end{equation}
for the bound $\barvertL{}$ of \ref{as:R1}. This form is motivated by
our applications, as will become clear in \Cref{sec:verification}.
Note that \ref{as:G2} implies
$\lim_{\Lambda\uparrow\ZZ^d} \GrL_{\V{0}} =
\sup_{\Lambda\uparrow\ZZ^d} \GrL_{\V{0}}$.
The propositions of this section will be made under the assumption
that this limit satisfies a $\cirb$-IRB, i.e., that
\begin{equation}
  \label{KIRBnow}  
  \GrI_{\V{0}}(x,y)
  \bydef
  \sup_{\Lambda\uparrow\ZZ^d} \GrL_{\V{0}}(x,y)
  \le
  \cirb\greens(y-x), \qquad x,y\in\ZZ^{d}.
\end{equation}
The key aspect of the next proposition is that the bound is
independent of $\Lambda$ and proportional to $(c\eta)^m$.  Recall that
$\Pi_{m}^{\sssL}$ is defined by \eqref{eq:Pi-m}.

\begin{proposition}
  \label{prop:Pi-bound}
  Suppose $d\geq 5$, and that \ref{as:Z1}--\ref{as:Z2},
  \ref{as:G1}--\ref{as:G2}, and \ref{as:R1p} hold, and that a
  $\cirb$-IRB \eqref{KIRBnow} holds. Then there are constants
  $c_{1},c_{2}>0$ depending only on $d$, $\Jump$ and $\cirb$ such that
  for each $m\geq 1$
  \begin{equation}
    \label{eq:Pi-m-bound}    
    \abs{\PiL_{m}(x,y)}\leq c_{1}(c_{2}\eta )^{m} \mnorm{y-x}^{-3(d-2)}.
  \end{equation}
\end{proposition}
\begin{proof}
  The basic input in our estimates is that by~\ref{as:G1}
  and~\ref{as:G2}, $\GrL_{\V{t}}(x,y)\leq \GrI_{\V{0}}(x,y)$, so the
  $\cirb$-IRB and \eqref{SRWG2} imply
  \begin{equation}
    \label{eq:G-FV}
    \GrL_{\V{t}}(x,y) \leq \cirb\tilde C_{\Jump} \mnorm{y-x}^{2-d},
  \end{equation}
  and hence, letting $\cirb_{1}=\max\{\cirb \tilde C_{\Jump},1\}$,
  by~\ref{as:R1p} in the form \eqref{eq:barvertL} and~\eqref{eq:G-FV},
  \begin{equation}
    \label{eq:vertex-FV}
    \barvertL{}(x,y)    
    \leq \eta                           
    (\indic{x=y}+\cirb_{1}^{2})\mnorm{y-x}^{4-2d}.
  \end{equation}
  For $u,u',v,v' \in\ZZ^{d}$, define\vspace{-1em}
  \begin{equation}
    A^{\sss(\Lambda)}(u,u';v,v')
    \bydef 
    \GrL_{\V{0}}(u,u')\barvertL{}(u,v')\GrL_{\V{0}}(u',v) 
    \eqqcolon
    \begin{tikzpicture}[baseline=4mm]
      \node[dot] (v1) at (1,0) {};
      \node[dot] (v0p) at (1,1) {};
      \node[dot] (v2) at (2,1) {};
      \node[dot] (v1p) at (2,0) {};
      \node at (v1) [below] {$u\phantom{'}$};
      \node at (v0p) [above] {$u'$};
      \node at (v2) [above] {$v$};
      \node at (v1p) [below] {$v'$};
      \draw[black, very thick]
       (v1) -- (v0p) -- (v2);
      \draw[black,thick, decorate, decoration={zigzag,segment
       length = 5, amplitude=1}] (v1) to (v1p);  
    \end{tikzpicture},
  \end{equation}
  where the right-hand side follows the diagrammatic notation of
  \Cref{fig:Pi-Bound}.  Recall the notation $\xb$ defined in
  \eqref{e:x-seq}, and note that \ref{as:G2} combined with a
  $\cirb$-IRB implies \ref{as:G0} holds.  Hence we can apply
  Proposition~\ref{prop:Pi-DB}, and this proposition can be rewritten
  as
  \begin{equation}
    \label{eq:Pib1}
    \abs{\PiL_{m}(x,y)}
    \leq \sum_{\xb\in\Lambda^{2m-2}_{x,y}} 
    \GrL_{\V{0}}(x_{0}',x_{1})\barvertL{}(x_{0},x_{1}')
    \prod_{j=1}^{m-1}A^{\sss(\Lambda)}(x_{j},x_{j}';x_{j+1},x_{j+1}') .
  \end{equation}
  To check this claim compare Figure~\ref{fig:Pi-Bound} with
  \begin{equation}
    \begin{tikzpicture}[baseline=8mm]
      \node[dot] (v0p) at (1,0) {};
      \node[dot] (v1) at (1,1) {};
      \node[dot] (v1p) at (2,1) {};
      \node[dot] (v2) at (2,0) {};
      \node at (v1) [above] {$x_{0}\phantom{'}$};
      \node at (v0p) [below] {$x_{0}'$};
      \node at (v2) [below] {$x_{1}\phantom{'}$};
      \node at (v1p) [above] {$x_{1}'$};
  
      \draw[black, very thick]
      (v0p) -- (v2);
      \draw[black,thick, decorate, decoration={zigzag,segment
        length = 5, amplitude=1}] (v1) to (v1p);  
    \end{tikzpicture}
    \begin{tikzpicture}[baseline=8mm]
      \node[dot] (v1) at (1,0) {};
      \node[dot] (v0p) at (1,1) {};
      \node[dot] (v2) at (2,1) {};
      \node[dot] (v1p) at (2,0) {};
      \node at (v1) [below] {$x_{1}\phantom{'}$};
      \node at (v0p) [above] {$x_{1}'$};
      \node at (v2) [above] {$x_{2}\phantom{'}$};
      \node at (v1p) [below] {$x_{2}'$};
  
      \draw[black, very thick]
      (v1) -- (v0p) -- (v2);
      \draw[black,thick, decorate, decoration={zigzag,segment
        length = 5, amplitude=1}] (v1) to (v1p);  
    \end{tikzpicture}
    \begin{tikzpicture}[baseline=8mm]
      \node[dot] (v0p) at (1,0) {};
      \node[dot] (v1) at (1,1) {};
      \node[dot] (v1p) at (2,1) {};
      \node[dot] (v2) at (2,0) {};
      \node at (v1) [above] {$x_{2}\phantom{'}$};
      \node at (v0p) [below] {$x_{2}'$};
      \node at (v2) [below] {$x_{3}\phantom{'}$};
      \node at (v1p) [above] {$x_{3}'$};

      \draw[black, very thick]
      (v1) -- (v0p) -- (v2);
      \draw[black,thick, decorate, decoration={zigzag,segment
        length = 5, amplitude=1}] (v1) to (v1p);  
    \end{tikzpicture}
    \begin{tikzpicture}[baseline=8mm]
      \node[dot] (v1) at (1,0) {};
      \node[dot] (v0p) at (1,1) {};
      \node[dot] (v2) at (2,1) {};
      \node[dot] (v1p) at (2,0) {};
      \node at (v1) [below] {$x_{3}\phantom{'}$};
      \node at (v0p) [above] {$x_{3}'$};
      \node at (v2) [above] {$x_{4}\phantom{'}$};
      \node at (v1p) [below] {$x_{4}'$};

      \draw[black, very thick]
      (v1) -- (v0p) -- (v2);
      \draw[black,thick, decorate, decoration={zigzag,segment
        length = 5, amplitude=1}] (v1) to (v1p);  
    \end{tikzpicture}
    \begin{tikzpicture}[baseline=8mm]
      \node[dot] (v0p) at (1,0) {};
      \node[dot] (v1) at (1,1) {};
      \node[dot] (v1p) at (2,1) {};
      \node[dot] (v2) at (2,0) {};
      \node at (v1) [above] {$x_{4}\phantom{'}$};
      \node at (v0p) [below] {$x_{4}'$};
      \node at (v2) [below] {$x_{5}\phantom{'}$};
      \node at (v1p) [above] {$x_{5}'$};
  
      \draw[black, very thick]
      (v1) -- (v0p) -- (v2);
      \draw[black,thick, decorate, decoration={zigzag,segment
        length = 5, amplitude=1}] (v1) to (v1p);  
    \end{tikzpicture}
  \end{equation}
  and recall that $x_{0}=x_{0}' = x$ and $x_{m}=x_{m}'=y$ and there is
  a sum over the remaining $x_{i},x_{i}'$. To estimate the summands
  in~\eqref{eq:Pib1} we introduce
  \begin{equation}
    \label{eq:barA}
    \bar A(u,u';v,v')
    \bydef
    \frac{\cirb_{1}^{2}(1+\cirb_{1}^{2})}{\mnorm{u-u'}^{d-2}
      \mnorm{u-v'}^{2d-4} \mnorm{u'-v}^{d-2}}. 
  \end{equation}

  Inserting the bounds~\eqref{eq:G-FV} (with $\V{t}=\V{0}$)
  and~\eqref{eq:vertex-FV} into $A^{\sss(\Lambda)} (u,u';v,v')$ we
  obtain
  \begin{align}
    \label{eq:barA1}
    A^{\sss(\Lambda)}(u,u';v,v')
    &\leq \eta
    \bar A(u,u';v,v') , \\
    \label{eq:barA2}
    \GrL_{\V{0}}(x_{0}',x_{1})\barvertL{}(x_{0},x_{1}')
    & \leq \eta
    \cirb_{1}^{-1}\bar A(x_{0},x_{0}';x_{1},x_{1}') ,
  \end{align}
  where \eqref{eq:barA2} holds because $x_{0}'=x_{0}$.  Inserting
  \eqref{eq:barA1}--\eqref{eq:barA2} into \eqref{eq:Pib1} yields
  \begin{equation}
    \label{eq:Pibb}
    \abs{\PiL_{m}(x,y)}
    \leq
    \cirb_{1}^{-1} 
    \eta^{m}
    U_{m}(x,y)
  \end{equation}
  where
  \begin{equation}
    U_{m}(x,y)
    \bydef 
      \sum_{\xb\in\ZZ^{d(2m-2)}_{x,y}}
      \prod_{j=0}^{m-1} \bar A(x_{j},x_{j}';x_{j+1},x_{j+1}') .
    \label{eq:Pib2}
  \end{equation}
  In obtaining \eqref{eq:Pibb} sums over vertices in $\Lambda$ have
  been extended to sums over $\ZZ^{d}$, which gives an upper bound as
  all terms are non-negative.  To prove \eqref{eq:Pi-m-bound} holds
  for $m\geq 1$ it therefore suffices to establish the upper bound
  \begin{align}
    U_{m}(x,y)
    &\leq
    c_{1}c_{2}^{m}\mnorm{y-x}^{-3(d-2)}.
    \label{eq:Pib4}
  \end{align}
  We prove~\eqref{eq:Pib4} with $c_{1}=\cirb_{1}^{2}(1+\cirb_{1}^{2})$
  and $c_{2}=\max\{1,C(\cirb_{1})\}$ (where $C$ is a constant defined
  in~\eqref{eq:Pib3} below) by induction. For $m=1$ there is no sum in
  \eqref{eq:Pib2}, so the bound follows from $x_{0}=x_{0}' = x$,
  $x_{1}=x_{1}'=y$, and \eqref{eq:barA}.

  Suppose the upper bound has been established for some $n-1\geq 1$.
  By multiplying both sides of Lemma~\ref{lem:Conv-E} by
    $\mnorm{u-v}^{2-d}$ and inserting \eqref{eq:barA} there is a
  $C=C(\cirb_{1})>0$ such that for $u,v,y \in \ZZ^{d}$,
  \begin{equation}
    \label{eq:Pib3}
    \sum_{a,b\in\ZZ^{d}}\bar A(v,u;a,b)\bar A(a,b;y,y)
    \leq C \bar A(v,u;y,y) .
  \end{equation}
  Let $m=n$. By using~\eqref{eq:Pib3} to estimate the sum over
  $x_{n-1},x_{n-1}'$ in the definition of $U_{n}$ and then using the
  induction hypotheses we have
  \begin{align}
    U_{n}(x,y)
    &=
    \sum_{\xb\in\ZZ^{d(2n-2)}_{x,y}}
      \prod_{j=0}^{n-1}
      \bar A(x_{j},x_{j}';x_{j+1},x_{j+1}') \\
    &\leq
    C \sum_{\xb\in\ZZ^{d(2n-4)}_{x,y}}
      \prod_{j=0}^{n-2}
      \bar A(x_{j},x_{j}';x_{j+1},x_{j+1}') \\\nn
    &\leq
    c_{1}c_{2}^{n}
      \mnorm{y-x}^{-3(d-2)},
  \end{align}
  where in the second line we have redefined $x_{n-1}\bydef x_{n}$ and
  $x_{n-1}'\bydef x_{n}'$. The final line follows by recalling that
  $x_{n}=x_{n}'=y$.
\end{proof}

Recall that $\Pi^{\sssL}=\sum_{m\ge 0}\Pi_{m}^{\sssL}$. Define
$\aPiL$, the finite-volume precursor to $\aPi$ from
\Cref{sec:infrared}, by
\begin{equation}
    \label{eq:aPi}
    \aPiL(x,y) \bydef \sum_{m\geq 1}\PiL_{m}(x,y) =\PiL_{}(x,y)-\PiL_{0}(x,y).
\end{equation}
Summing~\eqref{eq:Pi-m-bound} over $m\geq 1$ immediately gives the
following.

\begin{corollary}
  \label{cor:aPi-bound}
  Under the hypotheses of \Cref{prop:Pi-bound}, if $c_2\eta<1$ then
  \begin{equation}
    \label{eq:Pi-Decay}
    |\aPiL(x,y)| \leq \frac{c_1c_2\eta}{1-c_2\eta} 
    \mnorm{y-x}^{-3(d-2)}, \qquad x,y\in\ZZ^{d}.
  \end{equation}
\end{corollary}

\section{The lace expansion in infinite volume}
\label{sec:nIVL}

The main result of this section is \Cref{prop:IVL*}, which constructs
$\selfloop{\cpl,\nu}$ and $\aPi_{\cpl,\nu}$ such that
\eqref{eq:Int-Conv} holds.  This completes a key part of step one of
\Cref{sec:step-one}.  The proof uses \Cref{cor:aPi-bound} and the
algebraic structure of \Cref{prop:lace-expansion} to take the infinite
volume limit of \Cref{prop:lace-expansion}.  In particular we prove
the existence of the infinite volume limit $\PiI$ of $\PiL$.

\subsection{The infinite volume limit of $\GrL$}
\label{sec:infin-volume-limit-3}

We begin by establishing some properties of $\GrI_{\V{0}}$
which was defined in~\eqref{KIRBnow}.

\begin{lemma}
  \label{lem:Ginf}
  If \ref{as:Z1}, \ref{as:G2} and a $\cirb$-IRB \eqref{KIRBnow}
    holds then $\GrI_{\V{0}}(a,b)$ is non-negative and
    $\ZZ^{d}$-symmetric.
\end{lemma}

\begin{proof}
  Non-negativity is clear from \ref{as:Z1} and the
  definition \eqref{e:GF} of $\GrL$.

  By the $\cirb$-IRB and monotone convergence provided by \ref{as:G2}
  $\lim_{\Lambda\uparrow \ZZ^{d}}$ $\GrL_{\V{0}}(a,b)$ exists for any
  choice of exhaustion $\Lambda_{n}\uparrow\ZZ^{d}$. Given
  $\Lambda_{n}\uparrow\ZZ^{d}$ and $\Lambda'_{n}\uparrow\ZZ^{d}$ there
  exists $\tilde{\Lambda}_{n}\uparrow\ZZ^{d}$ such that
  $\tilde{\Lambda}_{n =1,3,5,...}$ is a subsequence of $\Lambda_{n}$
  and $\tilde{\Lambda}_{n =2,4,6,...}$ is a subsequence of
  $\Lambda'_{n}$. Since these three sequences have the same limit, the
  limit is independent of the exhaustion.  Independence of the
  exhaustion implies $\GrI_{\V{0}}(a,b)$ is translation invariant: the
  limit of $\GrL_{\V{0}}(a,b)$ through $\Lambda_{n}$ equals the limit
  of $\GrL_{\V{0}}(a',b')$ through $(\Lambda_{n}+e)$, where $a'=a+e$,
  $b'=b+e$, and $e$ a unit vector in $\ZZ^{d}$. Simultaneously, this
  latter limit is the same as the limit of $\GrL_{\V{0}}(a',b')$
  through $(\Lambda_{n})$. This implies that
  $\GrI_{\V{0}}(a,b) = \GrI_{\V{0}}(0,b-a)$.  A similar argument shows
  $\GrI_{\V{0}}(x):=\GrI_{\V{0}}(0,x)$ is
  $\ZZ^{d}$-symmetric. Therefore $\GrI_{\V{0}}(x,y)$ is
  $\ZZ^{d}$-symmetric; see the discussion above \Cref{hyp:Jnew}.
\end{proof}

Since $\GrI_{\V{0}}(x,y)$ is translation invariant, a $\cirb$-IRB of
the form \eqref{KIRBnow} implies $\GrI_{\V{0}}(x)$ satisfies a
$\cirb$-IRB of the form \eqref{eq:KIRB}. Thus in the sequel there is
no ambiguity when we say $\GrI_{\V{0}}$ satisfies a $\cirb$-IRB
without further specification.

\subsection{The infinite volume limit of $\PiL$}
\label{sec:infin-volume-limit-2}

In this section we prove the existence of the infinite volume limit of
$\PiL$. Recall that $\PiL=\sum_{m\ge 0}\PiL_{m}$ and $\PiL_{m}$ is
defined by \eqref{eq:Pi-m}. By \Cref{lem:Selfloop} the $m=0$ term is
$\PiL_{0}(x,y) = \selfloopL{\cpl,\nu,x} \indic{x=y}$ where
$\selfloopL{\cpl,\nu,x}$ is defined in \eqref{eq:selfloop-def}.  The
next assumption postpones proving that $\selfloopL{\cpl,\nu,x}$ has an
infinite volume limit to the next section.

\begin{assumptions}
  \leavevmode
  \begin{enumerate}
  \item[\mylabel{as:Z4}{(Z3)}] If a $\cirb$-IRB holds, then
    $\selfloopL{\cpl,\nu,x}$ is bounded uniformly in $x$ and $\Lambda$, and the
    limit
    $\selfloop{\cpl,\nu} = \selfloopI{g,\nu,x} \bydef \lim_{\Lambda\uparrow \ZZ^d}
    \selfloopL{\cpl,\nu,x}$ in \eqref{eq:selfloop-def} exists and is
    independent of $x$.
\end{enumerate}
\end{assumptions}

\begin{lemma}
  \label{lem:piInfb}
  Assume the hypotheses of \Cref{prop:Pi-bound} and \ref{as:Z4}.  If
  $\eta$ is sufficiently small, then for $x,y\in\ZZ^{d}$
  \begin{equation}
    \label{eq:piInfb}
    \abs{\PiL(x,y)}= O(\mnorm{y-x}^{-3(d-2)}),
  \end{equation}
  uniformly in $x,y$ and $\Lambda$.
\end{lemma}
\begin{proof}
  This is immediate from \eqref{eq:aPi}, \Cref{cor:aPi-bound},
  \Cref{lem:Selfloop} and \ref{as:Z4}.
\end{proof}

\begin{lemma}
  \label{lem:Pisub}
  Assume the hypotheses of \Cref{prop:Pi-bound} and \ref{as:Z4} and
  that $\eta$ is sufficiently small.  Then for any sequence of volumes
  $\Lambda_{n}\uparrow \ZZ^{d}$ there exists a subsequence
  $\Lambda_{n_k}$ such that
  $\Pi(x,y) \bydef \lim_{k\to\infty}\Pi^{\sss (\Lambda_{n_k})}(x,y)$
  exists pointwise in $x,y\in\ZZ^{d}$.
\end{lemma}
\begin{proof}
  Extend the definition of $\PiL\colon\Lambda\times\Lambda\to \RR$ to
  $\ZZ^{d}\times\ZZ^{d}$ by letting $\PiL(x,y)=0$ if $x\notin\Lambda$
  or $y\notin\Lambda$. By \Cref{lem:piInfb}, $\abs{\PiL(x,y)}$ is
  $O(\mnorm{y-x}^{-3(d-2)})$ uniformly in $\Lambda$.  Thus for any
  $x,y\in\ZZ^{d}$ and any increasing sequence of volumes
  $\Lambda_{n}\uparrow \ZZ^d$, there exists a subsequence
  $\Lambda_{n_k(x,y)}$ such that
  $\Pi^{\sss (\Lambda_{n_k(x,y)})}(x,y)$ converges as $k\to\infty$. By
  a diagonal argument we can refine this sequence such that the limit
  exists for all $x,y\in\ZZ^{d}$.
\end{proof}

\begin{lemma}
  \label{lem:ConvInf}
  Assume the hypotheses of \Cref{prop:Pi-bound} and \ref{as:Z4} and
  that $\eta$ is sufficiently small.  For a sequence
  $\Lambda_{n}\uparrow\ZZ^{d}$ for which $\Pi^{\sss (\Lambda_{n})}$
  converges pointwise to $\Pi$,
  \begin{equation}
    \label{eq:ConvInf}
    S^{\sss (\Lambda_{n})}\Pi^{\sss (\Lambda_{n})}G^{\sss (\Lambda_{n})}_{\V{0}}(x,y) \to \greens\Pi\GrI_{\V{0}}(x,y)\qquad x,y\in\ZZ^{d},
  \end{equation}
  and the product on the right-hand side is absolutely convergent, so there is no
  ambiguity in the order of the products.
\end{lemma}
\begin{proof}
  By \Cref{lem:piInfb}, $\PiL(x,y)$ is uniformly bounded above by a
  multiple of $U(x,y)\bydef \mnorm{y-x}^{-3(d-2)}$, $\GL$ is bounded
  above by $\greens$ and, by \ref{as:G2}, $\GrL_{\V{0}}$ is bounded
  above by $\GrI_{\V{0}}$.  Both $\GI(x,y)$ and $\GrI_{\V{0}}(x,y)$
  are non-negative and bounded above by a multiple of
  $\mnorm{y-x}^{-d+2}$.  Hence the products $\GI U (x,y)$ and
  $U\GrI_{\V{0}}(x,y)$ are both absolutely convergent by
  \Cref{lem:HHS-2}, and decay at least as fast as a multiple of
  $\mnorm{y-x}^{-d+2}$. Applying \Cref{lem:HHS-2} once more with
  $d\geq 5$ shows $\GI U \GrI_{\V{0}} (x,y)$ is given by an absolutely
  convergent double sum.

  Since $\Pi^{\sss (\Lambda_{n})}\to \Pi$ pointwise by hypothesis,
  $S^{\sss (\Lambda_{n})}\to \greens$ pointwise (see
  \eqref{e:GsrwIV}), and
  $G^{\sss (\Lambda_{n})}_{\V{0}}\to\GrI_{\V{0}}$ pointwise by
  \Cref{lem:Ginf}, \eqref{eq:ConvInf} follows by the dominated
  convergence theorem.
\end{proof}

Recall the definition of $\DI$ from \eqref{e:Delta-infty-matrix-def}.

\begin{lemma}
  \label{lem:2side}
  Assume the hypotheses of \Cref{prop:Pi-bound} and \ref{as:Z4}, that
  $\eta$ is sufficiently small, and that $\Lambda_{n}\uparrow\ZZ^{d}$
  is such that $\Pi^{\sss (\Lambda_{n})}\to\Pi$ pointwise. Then
  $-(\DI+\Pi)$ is a two-sided inverse of $\GrI_{\V{0}}$ and
    $\Pi(x,y)=\Pi(y,x)$ for $x,y \in \ZZ^{d}$.
\end{lemma}
\begin{proof}
  By \eqref{eq:lace-expansion}, \Cref{lem:Ginf} and \Cref{lem:ConvInf},
  \begin{equation}
    \label{eq:2side}
    \GrI_{\V{0}} = \greens + \greens\Pi\GrI_{\V{0}}.
  \end{equation}
  Multiplying \eqref{eq:2side} on the left by $-\DI$ and using
  \eqref{e:free-green-inf-1} yields
  \begin{equation}
    \label{eq:2side-1}
    -(\DI + \Pi)\GrI_{\V{0}}(x,y) = \indic{x=y}.
  \end{equation}
  In applying \eqref{e:free-green-inf-1} we have used that
  $\DI(\greens\Pi\GrI_{\V{0}})= (\DI\greens)(\Pi\GrI_{\V{0}})$, which
  holds as $\DI(x,\cdot)$ is finite range by \ref{as:J4n}. Thus
    $-(\DI + \Pi)$ is a left-inverse of $\GrI_{\V{0}}$.

  Letting $A^{t}$ denote the transpose of a matrix $A$, note that
  \begin{equation}
    \label{eq:2side-2}
    -\indic{x=y}
    = (\GrI_{\V{0}})^{t}  (\DI + \Pi)^{t} (x,y)
    = \GrI_{\V{0}} (\DI + \Pi)^{t}(x,y),
  \end{equation}
  as $(\GrI_{\V{0}})^{t}=\GrI_{\V{0}}$ by \Cref{lem:Ginf}. Note
  $\Pi(x,y) = O(\mnorm{y-x}^{-3(d-2)})$, as $\Pi$ is a pointwise limit
  of functions satisfying this uniform bound by
  \Cref{lem:piInfb}. Since $\DI(x,\cdot)$ is finite-range by
  \ref{as:J4n}, this implies $(\DI+\Pi)(x,y)$ is
  $O(\mnorm{y-x}^{-3(d-2)})$, and hence $(\DI+\Pi)^{t}(x,y)$ is also
  $O(\mnorm{y-x}^{-3(d-2)})$. Thus
  $(\DI+\Pi)\GrI_{\V{0}} (\DI+\Pi)^{t}$ is absolutely convergent and
  unambiguously defined by \Cref{lem:HHS-2} and $-(\DI + \Pi)^{t}$
    is a right-inverse of $\GrI_{\V{0}}$. By \Cref{lem:inv}
  $-(\DI + \Pi)$ is two-sided-inverse to $\GrI_{\V{0}}$ and
  \begin{equation*}
    \label{eq:2side-3}
    \DI+\Pi =  (\DI+\Pi)^{t}, 
  \end{equation*}
  and consequently $\Pi(x,y)=\Pi(y,x)$ as desired.
\end{proof}

\begin{proposition}
  \label{prop:Pi-inv}
  Assume the hypotheses of \Cref{prop:Pi-bound} and \ref{as:Z4} and
  that $\eta$ is sufficiently small. The limit
  $\PiI(x,y)\bydef \lim_{\Lambda\uparrow\ZZ^{d}}\PiL(x,y)$
  exists and is $\ZZ^{d}$-symmetric.
\end{proposition}
\begin{proof}
  Let $\Pi$ and $\tilde{\Pi}$ be pointwise limit points of exhaustions
  $\Lambda_n\uparrow \ZZ^d$ and $\tilde{\Lambda}_n\uparrow \ZZ^d$. Let
  $A=-(\DI+\Pi)$ and $\tilde{A}=-(\DI+\tilde{\Pi})$. By
  \Cref{lem:piInfb} $A$ and $\tilde{A}$ are
  $O(\mnorm{y-x}^{-3(d-2)})$.  Then
  $\sum_{u,v\in\ZZ^{d}}|A(x,u|\,|\GrI_{\V{0}}(u,v)|\,|\tilde{A}(v,y)|<\infty$
  by \Cref{lem:HHS-2}.  By \Cref{lem:2side} $A$ and $\tilde{A}$ are
  two-sided inverses of $\GrI_{\V{0}}$. By \Cref{lem:inv}
  $A=\tilde{A}$.  By \Cref{lem:Pisub} limit points exist for every
  exhaustion and we have just shown that the limit point is unique so
  $\PiI(x,y)\bydef \lim_{\Lambda\uparrow\ZZ^{d}}\PiL(x,y)$ exists.

  Let $T$ be an automorphism of $\ZZ^{d}$. By \Cref{lem:Ginf} and the
  definition of $\ZZ^{d}$-symmetry above \Cref{hyp:Jnew}
  $\GrI_{\V{0}}(Tx,Ty) = \GrI_{\V{0}}(x,y)$ for all $x,y$. Let
  $A^{\prime}(x,y) = A(Tx,Ty)$ for $A$ as above. By changing the
  summation variable from $y$ to $T^{-1}y$,
  $\sum_{y} \GrI_{\V{0}}(x,y)A^{\prime}(y,z)$ equals
  $\sum_{y} \GrI_{\V{0}}(x,T^{-1}y)A(y,Tz)$ which equals $\sum_{y}$
  $\GrI_{\V{0}}(Tx,y)$ $A(y,Tz)$ by the $\ZZ^{d}$-symmetry of
  $\GrI_{\V{0}}$. Since $A$ is a right inverse this sum simplifies to
  $\indic{Tx=Tz}=\indic{x=z}$.  We conclude that $A^{\prime}$ is a
  right-inverse to $\GrI_{\V{0}}$. By a similar calculation
  $A^{\prime}$ is a left-inverse so $A^{\prime}$ is a two-sided
  inverse to $\GrI_{\V{0}}$. Repeating the argument in the first
  paragraph we have $A^{\prime}=A$ so $A=-(\DI+\Pi^{\sssI})$ is
  $\ZZ^{d}$-symmetric as claimed.
\end{proof}

\subsection{The infinite-volume lace expansion equation}
\label{sec:obtaining-crefeq}

In the next proposition $\aPiI(x)\bydef\aPiI(0,x)$ is the infinite
volume limit of $\aPiL$ defined by \eqref{eq:aPi} and
$\selfloopI{\cpl,\nu}$ is as defined in \eqref{eq:selfloop-def}, in
which the limit exists by \Cref{as:Z4}.  Recall
$\GrI_{\V{0}}(x) = \GrI_{\V{0}}(0,y-x)$.  In Item~\ref{BHK1-1**c} of
the Proposition below we prove \eqref{eq:Int-Conv}.

\begin{proposition}
  \label{prop:IVL*} 
  Assume the hypotheses of \Cref{prop:Pi-bound} and \ref{as:Z4} hold.
  Then there exist $\alpha>0$, $\eta_{0}>0$, and $\aPiI$ such that for
  all $\eta\in (0,\eta_{0})$,
  \begin{enumerate}[label=(\roman*)]
  \item \label{BHK1-1**a} $\aPiI(x)$ exists and is a
    $\ZZ^{d}$-symmetric function of $x$.
  \item \label{BHK1-1**b} $\abs{\aPiI(x)}\leq 
    \alpha \eta\mnorm{x}^{-3 (d-2)}$.
  \item \label{BHK1-1**c} 
    $ (\hJ-\selfloopI{\cpl,\nu})\GrI_{\V{0}}(x) = \indic{x=0} +
    \jumprate\ast \GrI_{\V{0}}(x) + \aPi^{\sssI}\ast \GrI_{\V{0}}(x)$.
  \item \label{BHK1-1**d} $\hJ - \selfloopI{\cpl,\nu} \geq \frac{1 + O
      (\eta)} {\cirb\GI(0)}$.
\end{enumerate}
\end{proposition}

\begin{proof}
  Item \ref{BHK1-1**a}: Note $\aPiL(x,y) = \PiL(x,y) - \PiL_{0}(x,y)$.
  Both terms on the right-hand side have $\ZZ^{d}$-symmetric infinite
  volume limits by \Cref{prop:Pi-inv}, \Cref{lem:Selfloop} and
  \ref{as:Z4}. Therefore $\aPiI(x)$ exists and is a
  $\ZZ^{d}$-symmetric function.

  Item \ref{BHK1-1**b} follows from the finite-volume estimate given
  by \Cref{cor:aPi-bound}.  
  
  Item \ref{BHK1-1**c}: By \Cref{lem:2side}
  $-(\DI+\PiI)\GrI_{\V{0}} = \mathbf{I}$. In terms of the one variable
  functions $\PiI(x):=\PiI(0,x)$ and
  $\GrI_{\V{0}}(x):=\GrI_{\V{0}}(0,x)$ with convolution replacing
  matrix products this is $(J - \PiI)*\GrI_{\V{0}}(x) = \indic{x=0}$.
  We insert $\Jump(x) = - \hJ \indic{x=0} + \jumprate(x)$ from
  \eqref{e:hJ-def} and
  $\PiI(x) = \selfloopI{\cpl,\nu}\indic{x=0} + \aPiI(x)$ from
  \Cref{as:Z4}. After rearranging we obtain \ref{BHK1-1**c}.

  Item \ref{BHK1-1**d}: We evaluate \ref{BHK1-1**c} at $x=0$, insert
  item~\ref{BHK1-1**b} using $d \ge 5$ to obtain
  $|\GrI_{\V{0}}\ast \aPiI(0)| = O (\eta)$ and insert
  $\jumprate\ast \GrI_{\V{0}}(0) \geq 0$.  The result is the desired
  bound
  \begin{equation}
    \hJ-\selfloopI{\cpl,\nu}
    \geq
    \frac{1 + O (\eta)}{\GrI_{\V{0}}(0)}
    \geq
    \frac{1 + O (\eta)}{\cirb \GI(0)} 
    ,
  \end{equation}
  where the second inequality is implied by the $\cirb$-IRB
  hypothesis.
\end{proof}

\section{Final hypotheses and proof of asymptotic behaviour}
\label{sec:completion-proof}

In~\Cref{rem:dep} we listed the lemmas that collectively prove the
infrared bound of \Cref{thm:Main-Pre-models}.  As outlined in
\Cref{sec:outl-rema-paper} we now revise the hypotheses of these
lemmas to replace their specialisation to \ctwsaw and $\pf$ model by
\Cref{as:FA-1} and~\ref{as:FA-2} and we prove these revised lemmas.
The main result \Cref{thm:gen-main} is that these assumptions imply
the desired asymptotic form for the Green's function
$G_{\cpl,\nu_{c}}(x)$.  In \Cref{sec:verification} we will verify that
the Edwards model and the $n=1,2$-component $\pf$ models satisfy these
assumptions.

\subsection{Final hypotheses}
\label{sec:final-hypoth-proof}

This subsection summarizes the hypotheses under which we will draw
conclusions about the asymptotic behaviour of the Green's function.
For each finite $\Lambda\subset \ZZ^{d}$ and parameters $\cpl>0$ and
$\nu\in\RR$ let
$Z^{\sssL}_{\cpl,\nu}\colon [0,\infty)^{\Lambda}\to (0,\infty)$,
$\V{t} \mapsto Z^{\sssL}_{\cpl,\nu,\V{t}}$ satisfy
\begin{assumptions}
  \label{as:FA-1} 
  Assume $\Jump$ is such that \ref{as:J1}--\ref{as:J4n} hold, and
  \begin{itemize}
  \item[(i)] for each $\cpl>0,\nu\in \RR$, \ref{as:Z1}, \ref{as:Z2},
    \ref{as:Z4}, \ref{as:G1}--\ref{as:G2} hold for
    $Z^{\sssL}_{\cpl,\nu}$;
  \item[(ii)] there exists $\nu_{1}>0$ such that for all $g\ge 0$
    $G^{\sssL }_{\cpl,\nu_{1}}(x)$ is summable in $x$ uniformly in
    $\Lambda$;
  \item[(iii)] there exists $c_*>0$ such that for each
    $\cpl>0,\nu\in \RR$, \ref{as:R1p} holds for $Z^{\sssL}_{\cpl,\nu}$
    with $\eta=c_*\cpl$.
  \end{itemize}
\end{assumptions}

Throughout this section we write $G_{\cpl,\nu}=\GrI_{\cpl,\nu}$, which
exists as a possibly infinite monotone limit by \ref{as:G2}. The
\emph{susceptibility} is defined by
\begin{equation}
  \label{eq:susc}
  \chi_{\cpl}(\nu) \bydef \sum_{x\in\ZZ^{d}}G_{\cpl,\nu}(x),
\end{equation}
and the \emph{critical value} $\nu_{c}(\cpl)$ of $\nu$ is defined by
\begin{equation}
  \label{eq:nu-crit-def}
  \nu_{c}(\cpl) \bydef \inf\{\nu\in\RR \mid \chi_{\cpl}(\nu)<\infty\}.
\end{equation}
In particular, $\nu_{c}\le \nu_{1}$ with $\nu_{1}$ given in
\Cref{as:FA-1}.

\begin{assumptions}
  \label{as:FA-2}
  Assume that $\cpl>0$, that $\nu_{c}(\cpl)\leq 0$,
  and that
  \begin{enumerate}[label=(\roman*),series=hypIV]
  \item[\mylabel{as:G3}{(G3)}] $G_{\cpl,\nu}(x)$ is
     non-increasing in $\nu\in\RR$. Moreover,
    \begin{enumerate}
    \item For $\nu\in (\nu_{c}(\cpl),\infty)$ and
    $x\in\ZZ^{d}$, $G_{\cpl,\nu}(x)$ is continuous 
    in $\nu$.
  \item If $G_{\cpl,\nu}\leq 3\greens$ for some
    $\nu\in [\nu_{c},\infty)$, then
    $\{G_{\cpl,\nu'}(x)\}_{x\in\ZZ^{d}}$ is a uniformly equicontinuous
    family of functions of $\nu'\in[\nu,\infty)$.
    \end{enumerate}
  \item[\mylabel{as:G4}{(G4)}] For $\nu\in(\nu_{c}(\cpl),\infty)$,
    $\sup_{x\in\ZZ^{d}:|x|\ge r}G_{\cpl,\nu}(x)/\greens
    (x)\rightarrow 0$ as $r\rightarrow \infty$.
  \item[\mylabel{as:G5}{(G5)}] $G_{\cpl,\cpl}\leq 2\greens$.
  \item[\mylabel{as:Z5}{(Z4)}] For each finite $\Lambda$ and $\cpl>0$,
    $Z^{\sssL}_{\cpl,\nu}\colon [0,\infty)^{\Lambda}\to (0,\infty)$ is
    continuous in $\nu\in \RR$ pointwise in
    $\V{t}\in [0,\infty)^{\Lambda}$ and uniformly bounded in $\V{t}$
    for each $\nu$.
  \item[\mylabel{as:Z6}{(Z5)}] 
     \begin{enumerate}
     \item For $\cpl>0$, $\selfloop{\cpl,\nu}$ is continuous in
       $\nu \in (\nu_c(g),\infty)$.
     \item If $G_{\cpl,\nu}\leq 3\greens$ for some $\nu$, then
       $\selfloop{\cpl,\nu'}$ is continuous for
       $\nu'\in\co{\nu,\infty}$.  If, additionally,
       $\selfloop{\cpl,\nu}\leq 0$ and $\nu\in(\nu_{c}(\cpl),\cpl]$,
       then $\selfloop{\cpl,\nu}=O(\cpl)$.
     \end{enumerate}
   \item[\mylabel{as:Z7}{(Z6)}] Under the hypothesis
     $\nu_{c}=-\infty$, $\selfloop{\cpl,\nu}\to\infty$ as
     $\nu\to-\infty$.
  \end{enumerate}
\end{assumptions}

These assumptions suffice for our applications, and are intended to be
a starting point on the road to better assumptions.  The next lemma is
a step in this direction.

\begin{lemma}
  \label{lem:G4}
  \ref{as:Z1}, \ref{as:G1}, \ref{as:G2} and
  $\sum_{x\in\ZZ^{d}}G^{\sssI}(0,x)<\infty$ implies $G^{\sssI}(0,x)$
  decays exponentially in $x$ as $x\rightarrow\infty$. In particular
  \ref{as:Z1}, \ref{as:G1}, \ref{as:G2} imply \ref{as:G4}.
\end{lemma}

\noindent To prove this we establish a Simon inequality.

\begin{proposition}[Simon inequality]
  \label{prop:simon}
  Let $G^{\sssI}$ be the infinite volume limit of the Green's function
  of a model that satisfies \ref{as:Z1}, \ref{as:G1},
  \ref{as:G2}. Assume $G^{\sssI}(0,x)$ is summable in $x \in \ZZ^{d}$.
  For $a,b \in \ZZ^{d}$ and $\Lambda' \subset \ZZ^{d}$ such that
  $\Lambda'$ contains $a$ but not $b$,
  \begin{equation}
    \label{eq:simon}
    G^{\sssI}(a,b)
    \leq
    \sum_{x' \in \Lambda',x\in\ZZ^{d}\setminus\Lambda'}
    G^{\sssI}(a,x')
    \Jump^{\sss{(\partial\Lambda'})}(x',x)
    G^{\sssI}(x,b),
  \end{equation}
  where
  $\Jump^{\sss{(\partial\Lambda'})}(x',x) =
  \Jump(x-x')\indic{x'\in\Lambda'}\indic{x\in
    \ZZ^{d}\setminus\Lambda'}$, i.e.,
  $\Jump^{\sss{(\partial\Lambda'})}(x',x)$ is non-zero only for jumps
  out of $\Lambda'$.
\end{proposition}

Simon~\cite{SimonIneq} proved the progenitor for this inequality for
the Ising model, and showed that models that satisfy Simon
inequalities are such that whenever the two-point correlation is
summable the two-point function decays exponentially.
Lieb~\cite{Lieb} gave an important improvement in the Simon
inequality, which was extended by Rivasseau~\cite{Rivasseau} to
two-component models. The improvement was to replace $G^{\sssI}(a,x')$
by $G^{\sss{(\Lambda')}}(a,x')$. We are unable to obtain this
improvement in the generality of \Cref{prop:simon}; see~\cite[Theorem
6.1]{BFS82} for a random walk proof of this improvement for $\pf$
models.

The next proof paraphrases~\cite{SimonIneq} in the notation of this
paper. We assumed the finite range condition \ref{as:J4n} in order to
appeal to \cite{SimonIneq} to prove \ref{as:G4}.

\begin{proof}[Proof of \Cref{lem:G4}]
  Recall from \Cref{as:J4n} \ref{as:J4n} that there exists a range $R$
  such that $\Jump^{\sss{(\partial\Lambda')}}(x',x)=0$ for
  $|x'-x|\geq R$.  For $r\geq 1$ choose $\Lambda'$ to be the ball
  $\{x:|x|\leq rR\}$ in $\ZZ^{d}$, let
  $F(r) = \sum_{x'\in\Lambda'} \sum_{x\not\in\Lambda'} G^{\sssI}(0,x')
  \Jump^{\sss{(\partial\Lambda'})}(x',x)$, and let
  $f(s) = \sum_{x:|x|\ge sR}G^{\sssI}(0,x)$. The choice of $\Lambda'$
  and the range $R$ of $\Jump^{\sss{(\partial\Lambda')}}$ imply that
  we can upper bound $F(r)$ by replacing $x\not\in\Lambda'$ by
  $|x|>rR$ and $x'\in\Lambda'$ by $|x'|\geq(r-1)R$. This yields
  $F(r) \le \hJ\,f(r-1)$, where $\hJ$ is given by \eqref{e:hJ-def}.
  Note that for $x'\in \Lambda'$ and $|x|\geq (r+1)R$,
  $G^{\sssI}(0,x') \Jump^{\sss{(\partial\Lambda'})}(x',x)=0$, and
  hence by summing \eqref{eq:simon} over $|b|\geq nrR$ with $a=0$ we
  have $f(nr) \le F(r)f(nr-(r+1)) \leq \hJ\,f(r-1) f((n-2)r)$.  By
  iteration we obtain $f(nr) \le (\hJ\,f(r-1))^{n/2}f(0)$ for $n$
  even.  By the summability hypothesis $f(0)$ is finite and
  $f(r-1) \downarrow 0$ as $r\uparrow\infty$. Therefore, for $r$
  sufficiently large, $\hJ\,f(r-1) \leq \frac{1}{2}$. With this choice
  of $r$ and $n$ even $f(nr) \le 2^{-n/2}f(0)$ which implies
  $G^{\sssI}(0,x)$ decays exponentially in $x$ as desired and, by
  \eqref{SRWG}, \ref{as:G4} is an immediate consequence.
\end{proof}

\begin{proof}[Proof of \Cref{prop:simon}]
  Let $\Lambda \subset \ZZ^{d}$ such that
  $a \in \Lambda'\subset\Lambda$. Let
  $\mc{S}=\inf \{t\geq 0:\LX_{t} \not \in \Lambda'\}$ be the time of
  first exit from $\Lambda'$, and note that $\LX_{\mc{S}}$ is the
  position of $\LX$ immediately after it jumps for the first time out
  of $\Lambda'$.  Let
  \begin{equation}
    \label{eq:simon-P}
    P(a,x)
    \bydef
    \Ea\,\Big[\frac{\Z{[0,\mc{S}]}}{Z_{\V{0}}}
    \indic{\LX_{\mc{S}} = x}\Big].
  \end{equation}
  Recall that the finite volume Greens function $\GrL(a,b)$ is
  $\GrL_{\V{t}}(a,b)$ as defined by \eqref{e:GF} with $\V{t}$ set to
  $\V{0}$.  Using \ref{as:Z1} we interchange the integral with the
  expectation in \eqref{e:GF} and obtain
  \begin{equation}
    \GrL(a,b)
    =
    \Ea\,\Big[
    \int_{[0,\infty)} d\ell \,
    \frac{\Z{[0,\ell]}}{Z_{\V{0}}}
    \indic{\LX_{\ell} = b}\Big].
  \end{equation}
  The hypothesis on $a,b$ implies that $\ell>\mc{S}$. By conditioning
  on $\mc{F}_{\mc{S}}$
  \begin{equation}
    \GrL(a,b)
    =
      \Ea\,\bigg[
      \frac{\Z{[0,\mc{S}]}}{Z_{\V{0}}}
      \int_{[\mc{S},\infty)} d\ell \,
      \Ea\,\Big[
      \frac{Z_{\V{\tau}_{[0,\mc{S}]}+\V{\tau}_{[\mc{S},\ell]}}}{\Z{[0,\mc{S}]}}
      \indic{\LX_{\ell} = b}
      \Big\vert \mc{F}_{\mc{S}}
      \Big]
      \bigg].
  \end{equation}
  By \Cref{lem:clemma} with $I=[0,\mc{S}]$ and
  $\G{I}^{\sssL}(\LX_{\mc{S}},b) \leq\GrL(\LX_{\mc{S}},b)$ from
  \ref{as:G1} we have
  $\GrL(a,b) \leq \Ea\,\big[ \frac{\Z{[0,\mc{S}]}}{Z_{\V{0}}}
  \GrL(\LX_{\mc{S}},b) \big]$ which is the same as
  \begin{equation}
    \label{eq:GPG}
    \GrL(a,b)
    \leq
    \sum_{x\in\Lambda}
    P(a,x) \GrL(x,b).
  \end{equation}
  To complete the proof it suffices to show that
  \begin{equation}
    \label{eq:P-bound}
    P(a,x)
    \le
    \sum_{x'\in\Lambda'} \GrL(a,x') \Jump^{\sss{(\partial\Lambda'})}(x',x)
  \end{equation}
  because inserting \eqref{eq:P-bound} into \eqref{eq:GPG} and using
  \ref{as:G2} to take the infinite volume limit gives
  \eqref{eq:simon}.  \emph{For the remainder of this proof we write
    $X=\LX$.}  By summing over the possible values of $X_\ell$, and
  the (Poisson) number of jumps in $(\ell,\ell+\delta]$, one can
  easily show that for $x\in \Lambda\setminus\Lambda'$, $\ell>0$ and
  $\delta > 0$
  \begin{equation}
    \Ea \left[ \indic{X_{\ell+\delta}=x} \vert \mc{F}_{\ell}
    \right]
    \indic{\mc{S} > \ell}
    =
    \Jump(x-X_\ell) \delta
    \indic{\mc{S} > \ell}
    +
    O(\delta^{2}),
  \end{equation}
  where the $O(\delta^{2})$ term is uniform in
  $\ell,x, \,\text{a.s.}\, \omega$.  We let
  $Y_{\ell}=\frac{Z_{\V{\tau}_{[0,\ell]}}}{Z_{\V{0}}}$ and use this to
  obtain
  \begin{align}
    \Ea
    \left[
    \indic{X_{\ell+\delta}=x}
    \indic{\mc{S} > \ell}
    Y_{\ell}
    \right]
    &=
      \Ea
      \left[
      \Ea
      \left[
      \indic{X_{\ell+\delta}=x}
      \vert
      \mc{F}_{\ell}
      \right]
      \indic{\mc{S} > \ell}
      Y_{\ell}
      \right]
      \nonumber\\
    &=
      \delta \Ea
      \left[
      \Jump(x-X_\ell)
      \indic{\mc{S} > \ell}
      Y_{\ell}
      \right]
      +
      O(\delta^{2}).
  \end{align}
  Since for $\ell$ in a compact set $[0,M]$ the vector of local times
  ranges within a compact subset $[0,M]^\Lambda$, we have by
  \ref{as:Z1} that $Y_\ell$ is bounded by a constant for
  $\ell\in [0,M]$.  Therefore $O(\delta^{2})$ is uniform in $x$ and
  $\ell\in [0,M]$.  Note that $Y_{\ell}$ is pathwise continuous in
  $\ell$ by \ref{as:Z1} and \eqref{e:tauI-def} and since also the
  probability that the walker jumps in a small interval goes to zero
  as the length of the interval goes to 0, the expectation is
  continuous in $\ell$.  We insert $\ell = \ell_{m}=m\delta$, sum over
  $m=0,1,...,\lfloor M/\delta \rfloor$ and take the limit as
  $\delta\downarrow 0$ to obtain
  \begin{equation}
    \lim_{\delta\downarrow 0}
    \sum_{m=0}^{\lfloor M/\delta \rfloor}
    \Ea
    \left[
      \indic{X_{\ell_{m+1}}=x}
      \indic{\mc{S} > \ell_{m}}
      Y_{\ell_{m}}
    \right]
    =
    \int_{[0,M]}
    \Ea
    \left[
      \Jump(x-X_{\ell})
      \indic{\mc{S} > \ell}
      Y_{\ell}
    \right]\, d\ell .
  \end{equation}
  On the other hand, taking the sum inside the expectation, and using
  Dominated convergence we see that the left hand side converges to
  $ \Ea \left[ \indic{X_{\mc{S}}=x} \indic{\mc{S} \le M} Y_{\mc{S}}
  \right]$ because it partitions
  $\{\mc{S} \le \delta(\lfloor M/\delta \rfloor +1)\}$ into
  $\{\mc{S} \in (m\delta, (m+1)\delta]\}$, and $Y_{\ell}$ is pathwise
  continuous, and $X_{\ell}$ is right-continuous. We let
  $M\uparrow \infty$ to obtain
  \begin{equation}
    \Ea \left[ \indic{X_{\mc{S}}=x} Y_{\mc{S}}\right]
    =
    \int_{[0,\infty)}
    \Ea
    \left[
  \Jump(x-X_{\ell})
      \indic{\mc{S} > \ell}
      Y_{\ell}
    \right]\, d\ell .
  \end{equation}
  Recalling that $Y_{\ell}=\frac{Z_{\V{\tau}_{[0,\ell]}}}{Z_{\V{0}}}$
  the left hand side is $P(a,x)$ by definition \eqref{eq:simon-P}.  By
  inserting
  $\Jump(x-X_{\ell})= \sum_{x'\in
    \Lambda'}\indic{X_{\ell}=x'}\Jump(x-x')$ we have
  \begin{equation}
    P(a,x)
    =
    \sum_{x'\in\Lambda'}
    \int_{[0,\infty)} d\ell \,
    \;
    \Ea\,\Big[
    \frac{Z_{\V{\tau}_{[0,\ell]}}}{Z_{\V{0}}}
    \indic{X_{\ell} = x'}\indic{\mc{S} > \ell}\Big]
    \Jump^{\sss{(\partial\Lambda')}}(x',x).
  \end{equation}
  We insert $\indic{\mc{S} > \ell}\leq 1$ and \eqref{e:GF} with
  $\V{t}=\V{0}$ to obtain \eqref{eq:P-bound} and thereby complete the
  proof.
\end{proof}

\subsection{Model independent lemmas}
\label{sec:model-indep-lemm}

In this section we prove, under revised hypotheses, the model
independent \Cref{lem:BHK-1-1*,lem:BHK-1-2*,lem:BHK-2,lem:Fcontinuous}
in the list of~\Cref{rem:dep}.

\begin{lemma}
  \label{lem:BHK-1*-gen}
  Lemmas \ref{lem:BHK-1-1*} and \ref{lem:BHK-1-2*} hold when the
  \ctwsaw and $\pf$ model are replaced by Assumptions \ref{as:FA-1}
  and \ref{as:Z6}.
\end{lemma}

\begin{proof}
  Recall that $\decon_{\cpl,\nu}$, $\tilde G_{\cpl,\nu}(x)$, and
  $\tilde\aPi_{\cpl,\nu}(x)$ are defined in \eqref{eq:Int-Conv-m-decon}
  in terms of $\selfloop{\cpl,\nu}$ and
  $\aPi_{\cpl,\nu}(x):=\aPi_{\cpl,\nu}^{\sssI}(x)$, and in
  \eqref{e:dgreensz-def} we defined
  $\decon^{\dgreens}_{z}(x)= - \indic{x=0} + z \jumprate(x)$.

  \textbf{Proof of revised \Cref{lem:BHK-1-1*}}. By \Cref{prop:IVL*},
  \Cref{as:FA-1}(iii) and hypothesis $K=3$ we immediately obtain
  \Cref{lem:BHK-1-1*} with the desired revision. Note that this
  revised lemma implies there is a constant $c_{d}$ such that for
  $d\geq 5$
  \begin{equation}
    \label{eq:lem:BHK-1*-gen-1}
    0 \le \www (\cpl,\nu) \le c_{d}, \quad \text{and }
    \abs{\aPi_{\cpl,\nu}(x)}                                 
    \leq
    \cv \alpha c_*\mnorm{x}^{-3(d-2)}
    \leq
    \cv \alpha c_*\mnorm{x}^{-(d+4)}.
  \end{equation}
  
  \textbf{Proof of revised \Cref{lem:BHK-1-2*}}.
  Part~\ref{BHK1-2}: recall the definition of
  $\Dcal_{C}$ from above \Cref{lem:BHK-1-2*} and recall that we have chosen
  $\decon=\decon_{\cpl,\nu}$. We have to prove \cref{as:BHK2-1}:
  $\decon_{\cpl,\nu}$ is $\ZZ^d$-symmetric; \cref{as:BHK2-2}:
  $\sum_{x\in\ZZ^{d}} \decon_{\cpl,\nu}(x) \leq 0$; \cref{as:BHK2-3}:
  there exists $C_{0}$ and there exists
  $z=z(\cpl,\decon_{\cpl,\nu})\in\cb{0,\hJ^{-1}}$ such that
  \begin{equation}
    \abs{\decon_{\cpl,\nu}(x)-\decon^{\dgreens}_{z}(x)}
    \leq
    C_{0}\cv\mnorm{x}^{-(d+4)} .
    \label{e:BHK-1-2*a}
  \end{equation}

  \Cref{as:BHK2-1} holds by \ref{as:J3} and \Cref{prop:IVL*}.  To
  obtain \cref{as:BHK2-2}, we sum~\eqref{eq:Int-Conv-m} over $x$ and
  interchange the sum over $x$ with the sum in the convolution in
  \eqref{eq:Int-Conv-m}. Since $\nu > \nu_{c}$ the sums are absolutely
  convergent and the interchange is valid. The result is
  \begin{equation}
    \label{eq:decon-susc}
    \sum_{x\in\ZZ^{d}} \decon_{\cpl,\nu}(x) 
    =
    -\ob{\sum_{x\in\ZZ^{d}}\tilde G_{\cpl,\nu}(x)}^{-1} < 0,
  \end{equation}
  as desired. The inequality follows from $G_{\cpl,\nu}(x)>0$ and
  $\www(\cpl,\nu)>0$.

  \Cref{as:BHK2-3}.
  \begin{align}
    \abs{\decon(x)-\decon^{\dgreens}_{z}(x)}
    \leq
    \abs{\big(\www(\cpl,\nu) - z\big)}\jumprate(x)
    +
    \www(\cpl,\nu)\abs{\aPi_{\cpl,\nu}(x)}
    \label{e:BHK2-3a}
  \end{align}
  For $\www(\cpl,\nu) \le \hJ^{-1}$, the choice $z=\www(\cpl,\nu)$
  satisfies \eqref{e:BHK-1-2*a} with $C_{0}\geq c_{d}\alpha c_*$ by
  \eqref{eq:lem:BHK-1*-gen-1}. Otherwise, by the first equation in
  \eqref{eq:lem:BHK-1*-gen-1}, $\www(\cpl,\nu) > \hJ^{-1}$ and we will
  now prove that $z=\hJ^{-1}$ satisfies \eqref{e:BHK-1-2*a}.
  Accordingly set $z=\hJ^{-1}$ until the end of the proof of this part.
  \Cref{as:BHK2-2} bounds how much $\www(\cpl,\nu)$ can exceed
  $z$ as in the final inequality of
  \begin{align}
    0
    &\leq
      \www(\cpl,\nu) - z
      =
      \hJ^{-1}
      \sum_{x} \decon^{\dgreens}_{\www(\cpl,\nu)}(x)
      \nonumber\\
    &=
      \hJ^{-1}
      \sum_{x}
      \Big(
      \decon_{\cpl,\nu}(x)-\www(\cpl,\nu)\aPi_{\cpl,\nu}(x)
      \Big)
      \leq
      \hJ^{-1}
      \sum_{x}
      -\www(\cpl,\nu)\aPi_{\cpl,\nu}(x),
      \label{eq:zbd2}
  \end{align}
  where equalities follow from $\sum_{x}\jumprate(x)=\hJ$, the
  definition of $\decon^{\dgreens}_{z}(x)$ with $z$ replaced by
  $\www(\cpl,\nu)$ and the definition of $\decon_{\cpl,\nu}(x)$. This
  together with \eqref{e:BHK2-3a} implies
  \begin{align}
    \abs{\decon(x)-\decon^{\dgreens}_{z}(x)}
    \leq
    \Big(
    \sum_{x'}
    \www(\cpl,\nu)\abs{\aPi_{\cpl,\nu}(x')}      
    \Big)\frac{
    \jumprate(x)}{\hJ}
    +
    \www(\cpl,\nu)\aPi_{\cpl,\nu}(x).
    \label{e:BHK2-3b}
  \end{align}
  We insert \eqref{eq:lem:BHK-1*-gen-1}.  By \ref{as:J4n} there is a
  constant $c_{\jumprate}$ such that
  $\big(\sum_{x'}\mnorm{x'}^{-3(d-2)}\big) \jumprate(x) \hJ^{-1}$ $\leq
  c_{\jumprate}\mnorm{x}^{-(d+4)}$. Therefore the right hand side is
  bounded by $C_{0}\cv\mnorm{x}^{-(d+4)}$ for
  $C_{0}\geq c_{d}\alpha c_*(c_{\jumprate}+1)$.  The proof of
  item~\ref{as:BHK2-3} and therefore of part~\ref{BHK1-2} is complete.

  \Cref{lem:BHK-1-2*}, part~\ref{BHK1-1}: We must show that
  $\selfloop{\cpl,\nu}=O(\cpl)$. By \ref{as:Z6} and the hypothesis
  $\nu\in[\nu_{c},\cpl]$ this holds if $\selfloop{\cpl,\nu}\leq 0$. If
  $\selfloop{\cpl,\nu}\geq 0$ then $\www(\cpl,\nu) > \hJ^{-1}$ and
  part (ii) follows by inserting
  $\www(\cpl,\nu)=(\hJ-\selfloop{\cpl,\nu})^{-1}$ into \eqref{eq:zbd2}
  and solving the inequality for $\selfloop{\cpl,\nu}$.
\end{proof}

Next we prove \Cref{lem:BHK-2}. The hypotheses need no revision
because they do not reference our models.  This lemma extends
~\cite[Lemma~2]{BHK}, where the Laplacian is nearest neighbour, to the
finite range context of \Cref{hyp:Jnew}~\ref{as:J4n}. The proof is,
\emph{mutatis mutandis}, that of~\cite[Lemma~2]{BHK}, so we discuss only part
that needed care.

\begin{proof}[Proof of \Cref{lem:BHK-2}]
  Note this reference uses $-\Delta$ for what we denote by $\decon$,
  and that the formula for $\mu$ and its range is stated in the body
  of the proof of Lemma 2 in~\cite[below (28)]{BHK}.
  
  The most significant step to check is the Edgeworth expansion (24)
  in the proof of~\cite[Lemma~4]{BHK}. According to~\cite{BHK} this is
  equation (1.5b) of~\cite[Theorem~2]{Uchiyama} with $m=4$.  This
    is misleading even for the nearest neighbour Laplacian because
  (1.5b) is not the same as (24), but the proof of~\cite[Lemma~4]{BHK}
  remains valid with (1.5b) in place of (24) so we momentarily set
  this aside. The equation (1.5b) of~\cite[Theorem~2]{Uchiyama}
  continues to hold under our \Cref{hyp:Jnew}.  In particular,
  \ref{as:J3} implies that in (1.5b) the norm $\|\cdot\|$ is the
  Euclidean norm $|\cdot|$ and the odd Edgeworth coefficients
  $U_{1},U_{3}$ vanish. By the discussion
  below~\cite[Theorem~2]{Uchiyama} and (2.4)
  of~\cite[Theorem~2]{Uchiyama} the Edgeworth coefficients
  $U_{2}(\tilde\omega^{x})$ and $U_{4}(\tilde\omega^{x})$ in (1.5b)
  are continuous functions of the unit vector $\tilde\omega^{x}$ and
  are therefore bounded.

  We return to the problem with ~\cite[equation (24)]{BHK}. The
  coefficients in (1.5b) of~\cite[Theorem~2]{Uchiyama} depend on the
  direction $\tilde\omega^{x}$ by which $x$ approaches infinity,
  whereas ~\cite[equation (24)]{BHK} has no directional
  dependence. However the boundedness of this directional dependence
  is all that is used in the proof of~\cite[Lemma~2,
  Lemma~4]{BHK}. \end{proof}

\begin{lemma}
  \label{lem:Fcontinuous-new}
  Lemma \ref{lem:Fcontinuous} holds when the \ctwsaw and $\pf$ model
  are replaced by Assumption \ref{as:FA-1} and \ref{as:G3},
  \ref{as:G4} of Assumption \ref{as:FA-2}.
\end{lemma}

\begin{proof} 
  Recall from \eqref{e:Fnu-def*} that
  $F(\nu) \bydef \sup_{x\in\ZZ^{d}}f(x,\nu)$ where
  $f(x,\nu) = G_{\cpl,\nu}(x)/\greens(x)$. We have to prove that $F$
  is continuous on $(\nu_{c},g]$.  It suffices to
  prove that $F$ is continuous on $(\tilde{\nu}_{c},g]$ for
  $\tilde{\nu}_{c}>\nu_{c}$. As in step (i) in the proof of
  \cite[Proposition~2.2]{HHS2003}
  we reduce the supremum defining $F$ to a finite set.  By \ref{as:G4}
  there exists $r$ such that
  $\sup_{|x|>r} f(x,\tilde{\nu}_{c}) \leq \frac{1}{2} f(0,g)$.  For
  $\nu\in (\tilde{\nu}_{c},g]$ we have that
  $F(\nu) = \sup_{|x|\leq r}G_{\cpl,\nu}(x)/\greens(x)$ because 
  \ref{as:G3} implies
  \begin{equation}
    \sup_{|x|>r} f(x,\nu)
    \leq
    \sup_{|x|>r} f(x,\tilde{\nu}_{c})
    \leq
    \frac{1}{2} f(0,g)
    \leq
    \frac{1}{2} f(0,\nu).
  \end{equation}
  For fixed $x$, $f(x,\nu)$ is continuous in $\nu$ by ~\ref{as:G3}.
  Since the supremum of finitely many continuous functions is
  continuous, $F$ is continuous as desired.
\end{proof}

\subsection{Model-dependent lemmas and proof of
  \Cref{thm:Main-Pre-models}}
\label{sec:model-depend-lemm}

By the previous section the list in Remark~\ref{rem:dep} of lemmas
needed for Proposition \ref{prop:infrared} has been reduced to the
model-dependent Lemmas \ref{lem:Z-1-earlier}, \ref{lem:Z-1-earlier-2},
and \ref{lem:large-nu}.  The conclusions of these lemmas are contained
in Assumptions \ref{as:G2}, \ref{as:G3}, and \ref{as:G5}. In
particular, they are contained in Assumptions \ref{as:FA-1} and
\ref{as:FA-2} without \ref{as:Z7}.  According to Remark~\ref{rem:dep}
Theorem \ref{thm:Main-Pre-models} requires the same list together with
Lemmas \ref{lem:new-nucfin} and \ref{lem:Z5}. The conclusion of Lemma
\ref{lem:new-nucfin} is Assumption \ref{as:Z7}. Furthermore, according
to Remark~\ref{rem:dep}, the continuity properties used in the proof
of Lemma \ref{lem:Z5} are in Assumption~\ref{as:Z5}. Thus we have
proved the following

\begin{theorem}
  \label{thm:main-pre-ass}
  Proposition \ref{prop:infrared} holds when the \ctwsaw and the $\pf$
  model are replaced by Assumption \ref{as:FA-1} and \ref{as:FA-2}
  without \ref{as:Z7}. Theorem \ref{thm:Main-Pre-models} holds when
  the \ctwsaw and the $\pf$ model are replaced by Assumption
  \ref{as:FA-1} and \ref{as:FA-2}: in particular, under these
  assumptions, there is $\cpl_{0}=\cpl_{0}(d,\Jump)>0$ such that if
  $0<\cpl<\cpl_{0}$, then $\nu_{c}(\cpl)$ is finite and
  \begin{equation}\label{eq:Main-Pre00-2nd}
    G_{\cpl,\nu_{c}}(x) \le 2\greens (x)
    , {}\quad  x\in\ZZ^{d}.
  \end{equation}
\end{theorem}

\subsection{Proof of asymptotic behaviour}
\label{sec:proof-main-theorem}

We begin with two lemmas. The first, Lemma~\ref{lem:HinL1}, is an
extension of a lemma from~\cite{BHK}, and hence we only describe where
care must be taken in obtaining this extension.

\begin{lemma}
\label{lem:HinL1}
Let $\decon$ and $H$ be as in \Cref{lem:BHK-2}. If
$\sum_{x\in\ZZ^{d}}\decon(x)<0$ then $H(x) \in \ell^{1} (\ZZ^{d})$.
\end{lemma}

\begin{proof}
  By the definition \eqref{e:dgreensz-def} of $\dgreens_{z}$ as a
  series
  \begin{align}
    \sum_{x\in \ZZ^{d}}\abs{\dgreens_{z}(x)}
    &\leq
      \sum_{x\in \ZZ^{d}}\sum_{n\ge 0} (\abs{z} \jumprate)^{\ast n}(x)
      \nonumber\\
    &=
      \sum_{n\ge 0} \abs{z}^{n} \sum_{x\in \ZZ^{d}}
      \jumprate^{\ast n}(x)
      =
      \sum_{n\ge 0} \abs{z}^{n} \hJ^{n},
      \nonumber
  \end{align}
  where $\hJ=\sum_{x} \jumprate(x)$ by \ref{as:J1} and
  \eqref{e:hJ-def}. The interchange of sums is justified because all
  terms on the right hand side are positive.  The right hand side is
  absolutely convergent iff $\abs{z}<\hJ^{-1}$ and when it is
  absolutely convergent $\abs{\dgreens_{z}(x)}$ decays exponentially
  like $(\abs{z}\hJ)^{O(\mnorm{x})}$ as $\mnorm{x}\rightarrow\infty$ by
  \ref{as:J4n}.  

  From $\decon \ast H= -\1$ we can generate a series similar to
  $\dgreens_{z}$ but it is inadequate because, unlike $\Jump_{+}(x)$,
  $D(x)\not\geq 0$ for $x\not=0$. However, this sign problem was
  solved in \cite{BHK}, where 
  in the proof of~\cite[Lemma~2]{BHK}, it is shown that
  \begin{equation*}
    H = (-D\ast \dgreens_{\mu})^{-1}\ast \dgreens_{\mu},
  \end{equation*}
  where we have expressed the equation
  preceding~\cite[Equation~(32)]{BHK} in the notation of the present
  paper. In~\cite{BHK} it is shown that the first term
  $(-D\ast \dgreens_{\mu})^{-1}$ is an element of the Banach algebra
  $B$ defined at the beginning of~\cite[Section~4]{BHK}, i.e., the set
  of functions $f$ on $\ZZ^{d}$ that are $\ell^{1}$ and have
  $\sup_{x}|f(x)| |x|^{d}$ finite. Since
  $\sum_{x\in\ZZ^{d}} D(x) < 0$, $\mu$ defined in \Cref{lem:BHK-2}
  satisfies $\abs{\mu}<\hJ^{-1}$. Therefore, by the preceding
  paragraph, $\dgreens_{\mu}$ is an element of $B$ since it decays
  exponentially in $\mnorm{x}$.  Hence the convolution defining $H$ is
  an element of $B$; in particular it is $\ell^{1}$ as desired.
\end{proof}

\begin{lemma}
  \label{lem:DRC}
  Under \Cref{as:FA-1} and~\ref{as:FA-2},
  if $d\geq 5$ and $\cpl$ is sufficiently small, then
  $\{\decon_{g,\nu}(x)\}_{x\in\ZZ^{d}}$ is an equicontinuous family of
  functions for $\nu\in\co{\nu_{c},\infty}$. Moreover
  $\decon_{g,\nu}$ is continuous in $\nu\in\co{\nu_{c},\infty}$ as
    an $\ell^{1}$-valued function.
\end{lemma}

\begin{proof}
  To prove that $\decon_{g,\nu}(x)$ is defined and continuous in
  $\nu\in\co{\nu_{c},\infty}$ we now discuss the definitions in
  \eqref{eq:Int-Conv-m-decon} for $\nu =\nu_{c}$ as well as
  $\nu > \nu_{c}$.  By \Cref{thm:main-pre-ass} the infrared bound
  $G_{\cpl,\nu_{c}}\le 2\greens$ holds for $\nu \ge \nu_{c}$.  This
  implies the hypotheses of \Cref{prop:IVL*} hold with $\eta=c_*\cpl$
  for $\nu\ge \nu_{c}$. By item~\ref{BHK1-1**d} of \Cref{prop:IVL*} we
  conclude that
  $\decon_{\cpl,\nu}\bydef \decon^{\dgreens}_{\www(\cpl,\nu)} +
  \tilde\aPi_{\cpl,\nu}$ exists for $\nu \ge \nu_{c}$ as desired. This
  definition together with item~\ref{BHK1-1**c} of \Cref{prop:IVL*}
  asserts that
  \begin{equation}\label{eq:Int-Conv-m*}
    \decon_{\cpl,\nu} \ast \tilde G_{\cpl,\nu} (x)
    = - \indic{x=0} .
  \end{equation}
  Moreover, by \ref{as:J4n}, items~\ref{BHK1-1**b} and~\ref{BHK1-1**d}
  of \Cref{prop:IVL*}, and the lower bound on $w(\cpl,\nu)$ following
  \eqref{eq:lem:BHK-1*-gen-1}, there is a $c_{1}>0$ such that for
  $\nu\geq \nu_{c}$,
  \begin{equation}
    \label{eq:Ddecay}
    \abs{D_{\cpl,\nu}(x)}\leq c_{1}\mnorm{x}^{-d-4}.
  \end{equation}

  For $\nu_{1},\nu_{2}\in [\nu_{c},\infty)$, 
  \eqref{eq:Int-Conv-m*} implies that
  \begin{equation}
    \label{eq:Dcont}
    \decon_{\nu_{2}}\ast(\tilde G_{\nu_{1}}-\tilde
    G_{\nu_{2}})\ast \decon_{\nu_{1}} + (\decon_{\nu_{1}}-\decon_{\nu_{2}})=0,
  \end{equation}
  where we have omitted the subscript $\cpl$.  Note that the omission
  of the order of the convolutions in this equation is valid as the
  iterated convolutions are absolutely convergent by \eqref{eq:Ddecay}, 
  the infrared bound $G_{\cpl,\nu_{c}}\le 2\greens$, and \Cref{lem:HHS-2}.
  Therefore
  \begin{align}
    \nonumber
    \abs{\decon_{\nu_{2}}(x)-\decon_{\nu_{1}}(x)} &\leq
    \sup_{y\in\ZZ^{d}}\abs{\tilde G_{\nu_{1}}(y)-\tilde G_{\nu_{2}}(y)}
    \norm{\decon_{\nu_{1}}}_{1}\norm{\decon_{\nu_{2}}}_{1} \\
    \label{eq:Dcont2}
    &\leq C \sup_{y\in\ZZ^{d}}\abs{\tilde G_{\nu_{1}}(y)-\tilde
      G_{\nu_{2}}(y)} 
  \end{align}
  for some $C>0$ by \eqref{eq:Ddecay}.  By \ref{as:Z6} part (b),
  $\www (\cpl,\nu)$ is continuous in $\nu$ for $\nu \ge \nu_{c}$.
  Therefore the functions $\tilde G_{\cpl,\nu}(x)$ are equicontinuous
  on $\co{\nu_{c},\infty}$ by \ref{as:G3} part (b).  This proves that
  the functions $\decon_{g,\nu}(x)$ are equicontinuous in
  $\nu\in\co{\nu_{c},\infty}$ as desired.

  The second claim, that $\nu\mapsto \decon_{\cpl,\nu}$ is continuous
  in $\ell^{1}$, follows from the first. This is so because
  $\sum_{x\in\ZZ^{d}}|\decon_{\cpl,\nu}(x)|$ converges uniformly in
  $\nu$ by \cref{BHK1-1**b} and \cref{BHK1-1**d} of
  \Cref{prop:IVL*}. 
\end{proof}

\begin{lemma}
  \label{lem:deconsum0}
  Consider a model satisfying \Cref{as:FA-1} and~\ref{as:FA-2}.  If
  $d\geq 5$ and $\cpl$ is sufficiently small then
  $\sum_{x\in\ZZ^{d}}\decon_{\cpl,\nu_{c}}(x) = 0$.
\end{lemma}
\begin{proof}
  This follows from \Cref{lem:HinL1,lem:DRC} and the definition of
  $\nu_{c}$. (Recall that we showed $H=G$ in the proof of
  \Cref{thm:main-pre-ass}).
\end{proof}

\begin{theorem}
  \label{thm:gen-main}
  For models satisfying \Cref{as:FA-1} and~\ref{as:FA-2}, if $d\geq 5$
  there exists $\cpl_{0}=\cpl_{0}(d,\Jump)$ such that if
  $0<\cpl<\cpl_{0}$, then there are constants $C>0$, $\epsilon>0$ such
  that
  \begin{equation}
    \label{eq:gen-main}
    G_{\cpl,\nu_{c}}(x) \sim \frac{C}{\mnorm{x}^{d-2}} +
    O\ob{\frac{1}{\mnorm{x}^{d-2+\epsilon}}} .
  \end{equation}
\end{theorem}
\begin{proof}
  By \Cref{thm:main-pre-ass}, \ref{as:Z6} part (b)
  $\selfloop{\cpl,\nu}$ is (right-)continuous in $\nu$ at $\nu_{c}$
  and by the revised \Cref{lem:BHK-1-2*} part of
    \Cref{lem:BHK-1*-gen},
  $\selfloop{\cpl,\nu}=O(g)$ for $\nu\in[\nu_{c},\cpl]$. Therefore
  $\www(\cpl,\nu_{c})=\hJ^{-1}(1+O(g))$ and it suffices to prove
  \eqref{eq:gen-main} for $\tilde G_{\cpl,\nu_{c}}$.

  Let
  $\HaraJ(x) \bydef \www(\cpl,\nu_{c})\big(\jumprate(x) +
  \aPi_{\cpl,\nu_{c}}(x)\big)$, and note that
  \begin{equation}
    \label{eq:Hara-1}
    \decon_{\cpl,\nu}(x)=-\indic{x=0}+\HaraJ(x),                   
  \end{equation}
  by the definition of $\decon_{\cpl,\nu}$, see
  \eqref{eq:Int-Conv-m-decon}.  Let
  $\hat \HaraJ(k) = \sum_{x\in\ZZ^{d}}\HaraJ(x)e^{ik\cdot x}$ be the
  Fourier transform of $\HaraJ$.  By \cite[Theorem~1.4]{Hara2008},
  \eqref{eq:gen-main} holds for $\tilde G_{\cpl,\nu_{c}}$ if there is
  a $\rho>0$ such that
  \begin{enumerate}
  \item[\mylabel{as:H1}{(H1)}] $\hat \HaraJ(0)=1$,
  \item[\mylabel{as:H2}{(H2)}]
    $\abs{\HaraJ(x)}\leq K_{1}\mnorm{x}^{-(d+2+\rho)}$, some
    $K_{1}>0$,
  \item[\mylabel{as:H3}{(H3)}]
    $\sum_{x\in\ZZ^{d}}\norm{x}_{2}^{2+\rho}\abs{\HaraJ(x)}\leq
    K_{2}$, some $K_{2}>0$,
  \item[\mylabel{as:H4}{(H4)}] there is a $K_{0}>0$ such that
    $\hat \HaraJ(0) - \hat \HaraJ(k)\geq K_{0}\norm{k}_{2}^{2}$,
    $k\in\cb{-\pi,\pi}^{d}$.
  \end{enumerate}
  By \Cref{thm:main-pre-ass}, $G_{\cpl,\nu_{c}}$ satisfies an infrared
  bound. Hence \ref{as:H2}--\ref{as:H3} follow from \Cref{prop:IVL*},
  \ref{as:J4n}, and the assumption $d\geq 5$. Furthermore \ref{as:H1}
  follows from $\sum_{x\in\ZZ^{d}}\decon_{g,\nu_{c}}=0$, i.e.,
  \Cref{lem:deconsum0}. Thus the proof of \eqref{eq:gen-main} is
  reduced to proving \ref{as:H4}.

  To prove \ref{as:H4} let
  $\hat\jumprate(k) = \sum_{x\in\ZZ^{d}}\jumprate(x)e^{ik\cdot x}$ be
  the Fourier transform of $\jumprate$ and note that $\hJ$ defined by
  \eqref{e:hJ-def} equals the Fourier transform $\hat\jumprate(k)$
  evaluated at $k=0$ by \ref{as:J1}.  By absorbing
  $\www(\cpl,\nu_{c})$ into $K_{0}$ and $\hJ = \hat\jumprate(0)$,
  \ref{as:H4} can be re-expressed as
  \begin{equation*}
    \hJ - \hat\jumprate(k)
    +
    (\hat{\aPi}_{\cpl,\nu_{c}}(0)-\hat{\aPi}_{\cpl,\nu_{c}}(k))
    \geq K_{0}\norm{k}_{2}^{2}.
  \end{equation*}
  By item~\ref{BHK1-1**b} of \Cref{prop:IVL*}, the $\ZZ^{d}$-symmetry
  of $\aPi_{\cpl,\nu_{c}}$,
  $1-\cos(k\cdot x) \leq c_{1}(k\cdot x)^{2}$ and
  $\sum_{x}\mnorm{x}^{-3(d-2)}\abs{x}^{2} < \infty$ for $d\ge 5$,
  \begin{equation*}
    \abs{\hat{\aPi}_{\cpl,\nu_{c}}(0)-\hat{\aPi}_{\cpl,\nu_{c}}(k)}
    \le
    \sum_{x\in\ZZ^{d}} \abs{\aPi_{\cpl,\nu_{c}}(x)}(1-\cos(k\cdot x))
    =
    O(\cpl) \abs{k}^{2}
  \end{equation*}
  Thus to prove \ref{as:H4} it suffices to show
  $\hJ - \hat\jumprate(k)\geq c'\norm{k}_{2}^{2}$ for some $c'>0$; the
  desired bound then follows by taking $\cpl$ small enough.  The
  stated lower bound on $\hJ-\hat\jumprate(k)$ follows from
  \cite[Lemma~2.3.2]{lawler_limic_2010}, whose hypotheses are provided
  by \ref{as:J4n} and the irreducibility assumption \ref{as:J2}.
\end{proof}

\section{Verification of hypotheses}
\label{sec:verification}

Recall that in \Cref{sec:proof-prop:infrared,sec:proof-thm:Main-Pre-models} we
reduced the proof of \Cref{thm:main,thm:Main-Pre-models} and
\Cref{prop:infrared} to the lemmas listed in \Cref{rem:dep}. Then, in
\Cref{sec:completion-proof} we revised these lemmas by replacing the
\ctwsaw and $n=1,2$\, $\pf$ models by \Cref{as:FA-1} and~\ref{as:FA-2}
in the hypotheses and we proved these revised lemmas.  Therefore the
next \Cref{lem:specialisation} completes the proof of
\Cref{thm:main,thm:Main-Pre-models} and \Cref{prop:infrared}.

\begin{lemma}
    \label{lem:specialisation}
    The Edwards and $n=1,2$\, $\pf$ models defined in
    Sections~\ref{sec:latt-edwards-model} and \ref{sec:pf-models}
    satisfy Assumptions~\ref{as:FA-1} and \ref{as:FA-2}.
\end{lemma}
The remainder of this section proves this lemma, first for the Edwards
model and then for the $\pf$ model. For each model we also prove
\Cref{prop:og}.

\subsection{Edwards model}
\label{sec:wsaw-ver}

By the argument below \Cref{lem:Z-1-earlier} \Cref{as:FA-1}(ii)
  holds and $\nu_{c}\leq 0$ as required by \Cref{as:FA-2}.  By the
\Cref{def:SAW} of the Edwards model it is clear that \ref{as:Z1},
\ref{as:Z2}, and \ref{as:Z5} hold.

Short calculations starting from the
definition~\eqref{eq:selfloop-def} of $\selfloopL{\cpl,\nu,x}$ and the
definition~\eqref{eq:vweightc0} of $\vertexL{s,s'}(x,y)$ show that
\begin{equation}
  \label{e:WSAW-LV}
  \selfloopL{\cpl,\nu,x} = -\nu \indic{x\in \Lambda},
  \qquad
  \vertexL{s,s'}(x,y) = -2\cpl \indic{x=y\in \Lambda}
\end{equation}
and therefore the infinite volume limits are
$\selfloop{\cpl,\nu} = -\nu$ and
$\vertex{s,s'}(x,y) = -2\cpl \indic{x=y}$. From \eqref{e:WSAW-LV} we
immediately obtain \ref{as:Z4}, \ref{as:Z7}, and \Cref{as:FA-1}(iii).
The continuity statements in \ref{as:Z6} are clear
from~\eqref{e:WSAW-LV}, and the claim $\selfloop{\cpl,\nu}=O(\cpl)$ in
part (b) of \ref{as:Z6} follows as $\selfloop{\cpl,\nu}\leq 0$ implies
$\nu\geq 0$, and (b) includes the hypothesis
$\nu\in (\nu_{c}(\cpl),\cpl]$.

By \eqref{eq:defSAW} and \eqref{e:rptn} it follows that
\begin{equation}
  \rrptn_{\V{t},\V{s}}\bydef \exp\left\{-\cpl\sum_{x\in
      \Lambda}(2t_xs_x+s_x^2)-\nu\sum_{x\in \Lambda}s_x\right\}
\end{equation}
which is decreasing in $t_x$ for each $x$, so \ref{as:G1} holds. To
verify \ref{as:G2}, note that
$\{X_\ell^{\sss(\Lambda)}=b\} = \{\TL>\ell,X_\ell^{\sss(\infty)}=b\}$
and on this event
$\V{\tau_{[0,\ell]}}^{\sss(\Lambda)}=\V{\tau_{[0,\ell]}}^{\sss(\infty)}$.
Hence
\begin{equation}
     \rrptn_{\V{0},\V{\tau}^{\sss(\Lambda)}_{[0,\ell]}}\indic{\LX_{\ell} = b}
     =
     \rrptn_{\V{0},\V{\tau}^{\sss(\infty)}_{[0,\ell]}}
     \indic{\ZX_{\ell} = b}
     \indic{\TL>\ell},
\end{equation}
which is non-negative and increasing in $\Lambda$ since $\TL$ is. 

\ref{as:G3}. Monotonicity in $\nu$ is clear from
\eqref{eq:defSAW}.  We defer the proofs of \ref{as:G3} parts (a) and
(b) until after \Cref{lem:continuity} below, as they are very similar
to the detailed proof of \Cref{lem:continuity}.

\Cref{lem:G4} shows that \ref{as:G4} is a consequence of \ref{as:Z1},
\ref{as:G1}, \ref{as:G2} which we have already established.
\ref{as:G5} is clear since $\GrI_{\cpl,\cpl}\leq \greens$ by
\eqref{eq:defSAW}. This concludes the proof of Lemma
\ref{lem:specialisation} for the \edwards.

\begin{proof}[Proof of \Cref{prop:og} for \edwards]
  By $\selfloop{\cpl,\nu}=-\nu$ and continuity from \ref{as:Z6} part
  (b), we have $\nu_{c} = - \selfloop{\cpl,\cpl}$. Therefore the
  desired result follows from Lemma \ref{lem:BHK-1-2*} part (ii)
  which is a consequence of Lemma~\ref{lem:BHK-1*-gen} and Lemma
  \ref{lem:specialisation}.
\end{proof}

\subsection{\texorpdfstring{$\pf$}{phi4} theory with \texorpdfstring{$n=1,2$}{n=1,2}}
\label{sec:verification-pf}

Recall the definition of $\pair{\cdot}^{\sssL}_{\cpl,\nu,\V{t}}$
from~\eqref{eq:pfdef}. In this section we abbreviate this to
$\pair{\cdot}^{\sssL}_{\V{t}}$.  By the
definition~\eqref{eq:selfloop-def} of $\selfloopL{\cpl,\nu,x}$ and the
definition~\eqref{eq:vweightc0} of $\vertexL{s,s'}(x,y)$
straightforward calculations show that
\begin{align}
  \label{eq:phi4-selfloop}
  \selfloopL{\cpl,\nu,x}
  &=
    \left(
        - \nu - \cpl\pair{\abs{\varphi_{x}}^{2}}_{\V{0}}^{\sssL}
    \right)\indic{x \in \Lambda}
  \\
  \label{eq:phi4-vertex}
  \vertexL{s,s'}(x,y)
  &=
    \Big(-2\cpl\indic{x=y} + \cpl^{2}
    \pair{\abs{\varphi_{x}}^{2};\abs{\varphi_{y}}^{2}}_{\V{\tau}^{\sssL}_{[s,s']}}^{\sssL}\Big)
    \indic{x,y \in \Lambda},
\end{align}
where
$\pair{A;B}^{\sssL} \bydef \pair{AB}^{\sssL} -
\pair{A}^{\sssL}\pair{B}^{\sssL}$.

\begin{lemma}
  \label{lem:G-1}
  For the $n$-component $\pf$-model with $n=1,2$, $a,b\in\Lambda$ and
  $\V{t}\in[0,\infty)^{\Lambda}$, $\GrL_{\cpl,\nu,\V{t}}(a,b)$ is
  non-decreasing in $\nu$, in each component of $\V{t}$, and in
    $\Lambda$.
\end{lemma}
\begin{proof}
  These statements follow by writing $\GrL_{\cpl,\nu,\V{t}}(a,b)$ in
  terms of
  $\pair{\varphi_{a}\cdot\varphi_{b}}_{\cpl,\nu,\V{t}}^{\sssL}$ using
  \Cref{{thm:just}}.  We prove that this expectation has the claimed
  monotonicity in $\nu$ by showing that the derivative with respect to
  $\nu$ is nonpositive. Up to a positive constant of proportionality
  this derivative is \emph{minus} the sum over $y\in \Lambda$ of
  $
  \pair{\varphi_{a}\cdot\varphi_{b};\varphi_{y}\cdot\varphi_{y}}_{\cpl,\nu,\V{t}}^{\sssL}$.
  As desired this correlation is nonnegative by the GKS inequality for
  $n=1$ and the Ginibre inequality for $n=2$. See \cite[Lemmas~11.3
  and~11.4]{FFS92}. Monotonicity in $\V{t}$ is proved in a similar way
  by differentiating
  $\pair{\varphi_{a}\cdot\varphi_{b}}_{\cpl,\nu,\V{t}}^{\sssL}$ with
  respect to a component $t_{y}$ of $\V{t}$. Monotonicity in $\Lambda$
  follows from monotonicity in $\V{t}$ because letting
  $t_{y} \uparrow \infty$ is the same as removing the point $y$ from
  $\Lambda$.
\end{proof}

\begin{proposition}[Lebowitz Inequality]
  \label{prop:phi4-Lebowitz}
  Consider the $\pf$ model with $n=1,2$ components. Then for all
  $\Lambda$, $x,y, u\in \Lambda$ and $\V{t} \in [0,\infty)^{\Lambda}$,
  \begin{equation*}
    0 \leq
    \pair{\varphi_{x}\cdot\varphi_{y} ; 
      \varphi_{u}\cdot \varphi_{u}}_{\V{t}}^{\sssL}             
    \le
    2 \pair{\varphi_{x}\cdot
      \varphi_{u}}_{\V{t}}^{\sssL}\pair{\varphi_{y}\cdot
      \varphi_{u}}_{\V{t}}^{\sssL} . 
  \end{equation*}
\end{proposition}
\begin{proof}
  The lower bound is the GKS inequality for $n=1$ and the Ginibre
  inequality for $n=2$. See \cite[Lemmas~11.3 and~11.4]{FFS92}. The
  upper bound for $n=1$ is the Lebowitz inequality~\cite{lebowitz1974}
  and for $n=2$ was proved by
  Bricmont~\cite[Theorem~2.1]{Bricmont1977}.  See also
  \cite[Theorem~12.1]{FFS92}.
\end{proof}

In this paragraph we prove Assumption \ref{as:FA-1}. From
\Cref{def:Zpf} it is clear that \ref{as:Z1}, \ref{as:Z2}
hold. Assumptions \ref{as:G1}, \ref{as:G2} follow immediately from
\Cref{lem:G-1}.  Assumption \ref{as:Z4} follows from
\eqref{eq:phi4-selfloop}, \ref{as:G2} and Lemma
\ref{lem:Ginf}. Therefore Assumption \ref{as:FA-1}(i) holds.  By the
argument below \Cref{lem:Z-1-earlier} Assumption \ref{as:FA-1}(ii)
holds. Assumption \ref{as:FA-1}(iii) is verified below by Lemma
\ref{lem:r}.

We begin the proof of Assumption \ref{as:FA-2} by noting that
$\nu_{c}\leq 0$ by the argument below \Cref{lem:Z-1-earlier}.  The
first sentence in \ref{as:G3} follows immediately from \Cref{lem:G-1},
and the remaining parts (a) and (b) of \ref{as:G3} are proved below in
Lemma \ref{lem:continuity}. By \Cref{lem:G4} \ref{as:G4} is a
consequence of \ref{as:Z1}, \ref{as:G1}, and \ref{as:G2}.  Property
\ref{as:G5} is clear since $\GrI_{\cpl,\cpl}\leq \greens$ is obtained
by inserting $Z_{\V{t}+\V{s}}\le e^{-\nu s}Z_{\V{t}}$ into
\eqref{e:GF}; see below Lemma \ref{lem:Z-1-earlier}.  From
\Cref{def:Zpf} it is clear that \ref{as:Z5} holds. The continuity
statements in \ref{as:Z6} follow from \ref{as:G3}.  The
$\selfloop{\cpl,\nu}=O(\cpl)$ part of \ref{as:Z6} follows from
\eqref{eq:phi4-selfloop} as the $3$-IRB implies the expectation is
bounded uniformly in $\cpl$.

We have now proved Assumptions \ref{as:FA-1} and \ref{as:FA-2} except
for \ref{as:Z7}.  By \Cref{thm:main-pre-ass} and the hypothesis
$\nu_{c}=-\infty$ contained in \ref{as:Z7}, we have
$G_{\cpl,\nu} \leq 2 \greens$ for $\nu > -\infty$. Therefore
\ref{as:Z7} follows from \eqref{eq:phi4-selfloop}.

\begin{lemma}\label{lem:continuity}
\leavevmode
\begin{enumerate}
\item $G_{\cpl,\nu}(x)$ is Lipschitz as a function of
  $\nu \in (\nu_{c},\infty)$.
\item If $d\ge 5$ and $G_{\cpl,\nu_{c}}$ satisfies a $\cirb$-IRB for
  some $\cirb$ then $G_{\cpl,\nu}(x)$ is uniformly Lipschitz as a
  function of $\nu \in [\nu_{c},\infty)$, and hence
  $\{G_{\cpl,\nu}(x)\}_{x\in\ZZ^{d}}$ is uniformly equicontinuous in
  $\nu\in [\nu_{c},\infty)$.
\end{enumerate}
\end{lemma}

\begin{proof}

  (1) For any finite volume $\Lambda$ and any $\nu$, by
  \Cref{prop:phi4-Lebowitz}, \eqref{e:just} and \ref{as:G2}
  \begin{align}
    \pair{\varphi_{0}\cdot\varphi_{x};
    |\varphi_{y}|^{2}}_{\V{0}}^{\sss(\Lambda)}
    &\le
      2 \pair{\varphi_{0}\cdot
      \varphi_{y}}_{\V{0}}^{\sss(\Lambda)} \pair{\varphi_{x}\cdot
      \varphi_{y}}_{\V{0}}^{\sss(\Lambda)}
      \leq
      2 n^{2} G_{\cpl,\nu}(y) G_{\cpl,\nu}(y-x)
      \nonumber\\
    \label{eq:cont-pf}
    &\le
      n^{2} G_{\cpl,\nu}^{2}(y) + n^{2} G_{\cpl,\nu}^{2}(y-x) ,
  \end{align}
  and the final inequality is the elementary inequality
  $2uv\leq u^{2}+v^{2}$ for $u,v\in\RR$.  Since
  \begin{equation}
    -\frac{\partial}{\partial
      \nu}G^{\sss(\Lambda)}_{\cpl,\nu}(x)
    =
    \frac{1}{2}
    \sum_{y \in
      \Lambda}\pair{\varphi_{0}\cdot\varphi_{x};
      \varphi_{y}^{2}}_{\V{0}}^{\sss(\Lambda)}
  \end{equation}
  we have, for $\nu$ and $a$ such that $\nu \ge a \ge \nu_{c}$,
  \begin{equation}
    \label{eq:differentiability}
    \big|\frac{\partial}{\partial \nu}G^{\sss(\Lambda)}_{\cpl,\nu}(x)\big|
    \le
    n^{2}\sum_{y \in \ZZ^{d}}G_{\cpl,\nu}^{2} (y)
    =
    c_{a},
  \end{equation}
  where $c_{a}=n^{2}\sum_{y \in \ZZ^{d}}G_{\cpl, a}^{2} (y)$ and
  $c_{a}$ is finite for $a>\nu_{c}$ because $G_{\cpl, a}(y)$ is
  summable by \eqref{eq:nu-crit-def}.  For $a>\nu_{c}$ and
  $\nu ,\nu ' \in [a,\infty)$, by writing
  $G^{\sss(\Lambda)}_{\cpl,\nu'}(x) - G^{\sss(\Lambda)}_{\cpl,\nu}(x)$
  as the integral of its derivative we have
  $|G^{\sssL}_{\cpl,\nu'}(x) - G^{\sssL}_{\cpl,\nu}(x)| \le c_{a}
  |\nu'-\nu|$.  Taking $\Lambda \uparrow \ZZ^{d}$ by \ref{as:G2}, we
  obtain $|G_{\cpl,\nu'}(x) - G_{\cpl,\nu}(x)| \le c_{a} |\nu'-\nu|$
  and therefore $G_{\cpl,\nu}(x)$ is Lipschitz as claimed.
 
  (2) We repeat part (1) with $a=\nu_{c}$. Since $d \ge 5$,
  $c_{\nu_{c}}$ is finite by the $\cirb$-IRB, so $G_{\cpl,\nu}(x)$ is
  Lipschitz on $[\nu_{c},\infty)$ with uniform constant $c_{\nu_{c}}$.
\end{proof}

\begin{proof}[Proof of \ref{as:G3} parts (a) and (b) for the \edwards]
  \label{rem:G3saw}
  We first claim that for
  any finite volume $\Lambda$ and any $\nu\in\RR$,
  \begin{equation}
    \label{eq:G3a}
    -\frac{d}{d\nu} \GrL_{\cpl,\nu}(x) \leq \GrL_{\cpl,\nu}\ast \GrL_{\cpl,\nu}(x) \leq
   \frac{1}{2}\sum_{y\in\ZZ^{d}}( (\GrL_{\cpl,\nu}(y-x))^{2}+(\GrL_{\cpl,\nu}(y))^{2}),
  \end{equation}
  where the second inequality is the elementary $2ab\leq
  a^{2}+b^{2}$. Granting the claim, note that by \ref{as:G2} and
  translation invariance this proves \eqref{eq:differentiability} for
  the Edwards model, and the remainder of the proof is essentially
  identical to the proof above.
  
  We now prove the claimed first inequality in~\eqref{eq:G3a}.
  By the definitions \eqref{e:GF} and \eqref{e:tauI-def}, the
  left-hand side of \eqref{eq:G3a} is
  \begin{equation}
    \label{eq:G3a-proof-1}
    \sum_{x' \in \Lambda}
    \int_{[0,\infty)} d\ell \,
    \int_{[0,\ell]} d\ell'\,
    \;
    \Ea\,\Big[
    \rptn_{0,\ell}
    \indic{\LX_{\ell'}=x'} \indic{\LX_{\ell} = x}\Big] .
  \end{equation}
  where $\rptn_{s,t} = Z^{\sssL}_{\V{\tau}^{\sssL}_{[s,t]}}/Z^{\sssL}_{\V{0}}$
  as in \eqref{e:rptn-choice}, and, for the Edwards model,  $Z^{\sssL}_{\V{0}}=1$.
  We reverse the order of integration over $\ell, \ell' $ and insert
  $\rptn_{0,\ell} = \rptn_{0,\ell'}\brptn_{0,\ell'} (\ell)$, where
  $\brptn_{0,\ell'} (\ell) = \big(\rptn_{0,\ell}\big/\rptn_{0,\ell'}\big)$
  as in \eqref{e:brptn-def}.
  By \Cref{lem:induct} with $H=\rptn_{0,\ell'}$ and
  $(u_{1}, u_{2}, u_{3}) = (0,\ell',\ell')$
  the result is
  \begin{equation}
    \sum_{x' \in \Lambda}
    \int_{[0,\infty)} d\ell' \,
    \;
    \Ea\,\Big[
    \rptn_{0,\ell'}
    \GrL _{\V{\tau}^{\sss(\Lambda)}_{[0,\ell']}}(X^{\sssL}_{\ell'},x)
    \indic{\LX_{\ell'}=x'} \Big] ,
  \end{equation}
  By \ref{as:G1} and \eqref{e:GF} read from right to left
  we obtain the first inequality in \eqref{eq:G3a} as desired.
\end{proof}

For $0<u<v$ define
\begin{align}
  &
    \barvertL{u,v} (x,y)
    \bydef
    2\cpl
    \cb{
    \indic{x=y}
    +
    n^{2}\cpl
    \Big(\GrL_{\V{\tau}_{[u,v]}^{\sssL}}(x-y)\Big)^{2}},
    \\
  &
    \barvertL{} (x,y)
    \bydef
    \barvertL{0,0} (x,y).
   \label{e:rbar}
\end{align}
The next lemma verifies \Cref{as:FA-1}(iii).
\begin{lemma}
  \label{lem:r}
  Suppose $0<u<v$, $x,y\in\Lambda$. Then
  \begin{equation}
    \abs{\vertexL{u,v} (x, y)}
    \le
    \barvertL{u,v} (x,y)
    \le
    \barvertL{} (x,y) .
  \end{equation}
\end{lemma}
\begin{proof}
  As $\Lambda$ is fixed we will omit it from the notation. 
  Applying the triangle inequality to~\eqref{eq:phi4-vertex}
  and using \Cref{prop:phi4-Lebowitz} and \eqref{e:just}, 
  \begin{equation}
    \vertex{u,v} (x, y) \le \barvert{u,v}(x, y),
  \end{equation}
  The remaining inequality $\barvert{u,v}(x,y) \le \barvert{}(x,y)$
  follows by \Cref{lem:G-1}.
\end{proof}

\begin{proof}[Proof of \Cref{prop:og} for the $\pf$ model]
  The same argument as for the \edwards\ shows
  $\selfloop{\cpl,\nu_{c}}$ is $O(\cpl)$. Since $\selfloop{\cpl,\nu} =
  -\cpl\pair{\varphi_{x}^{2}}_{\V{0}}-\nu$, and
  $\cpl\pair{\varphi_{x}^{2}}_{\V{0}}$ is $O(\cpl)$ at
  $\nu_{c}$ since an infrared bound holds, the claim follows.
\end{proof}

\appendix
\section{Random walk and the Markov property}

\subsection{Properties of continuous-time random walk}
\label{appendix}

\begin{proof}[Proof of Lemma~\ref{lem:free-green}]
  We first prove that $\DL$ is invertible.  Let
  $f,h\colon\Lambda\rightarrow\RR$. The quadratic form associated to
  $\DL$ is given by
  \begin{equation}
    \label{e:Lambda-dirichlet-form}
    (f,-\DL h)
    \bydef
    \sum_{x \in \Lambda} f_{x} (-\DL h)_{x} .
  \end{equation}
  For $f\colon\Lambda\rightarrow\RR$ define the extension by zero:
  $\tilde{f} = f$ on $\Lambda$ and $\tilde f_x=0$ for
  $x\notin\Lambda$.  We \emph{claim} that
  \begin{equation}
    \label{e:Lambda-dirichlet-form2}
    (f,-\DL h)
    =
    \frac{1}{2}\sum_{x,y \in \ZZ^{d}} \Jump (x-y) 
    (\tilde{f}_{x} - \tilde{f}_{y}) (\tilde{h}_{x} - \tilde{h}_{y}).
  \end{equation}
  By choosing $h=f$ we obtain
  \begin{equation}
    (f,-\DL f) = \frac{1}{2}\sum_{x\neq
      y\in\ZZ^{d}}\Jump (x-y)|\tilde{f}_{x}-\tilde{f}_{y}|^{2} > 0,\quad
    f\neq 0 .\label{posdef-Delta} 
  \end{equation}
  The strict inequality holds because $\tilde{f}_{y}=0$ for
  $y\notin\Lambda$ and for every point $v\in \Lambda$ there is a walk
  with transitions of nonzero rate that starts at $v$ and reaches a
  point not in $\Lambda$. This positivity implies that the eigenvalues
  of $-\DL$ are strictly positive and therefore $-\DL$ is invertible
  as desired.  Thus it suffices to prove the claim
  \eqref{e:Lambda-dirichlet-form2}.

  To prove \eqref{e:Lambda-dirichlet-form2} we start with the
  right-hand side which contains
  \begin{equation}
    (\tilde{f}_{x} - \tilde{f}_{y}) (\tilde{h}_{x} - \tilde{h}_{y})
    =
    \tilde{f}_{x}(\tilde{h}_{x} - \tilde{h}_{y})
    +
    \tilde{f}_{y} (\tilde{h}_{y} - \tilde{h}_{x}),
  \end{equation}
  so by the symmetry under exchanging $x$ and $y$ we can rewrite the
  right-hand side of \eqref{e:Lambda-dirichlet-form2} as
  \begin{align}
    \frac{1}{2}\sum_{x,y \in \ZZ^{d}} \Jump (x-y)
    (\tilde{f}_{x} - \tilde{f}_{y}) (\tilde{h}_{x} - \tilde{h}_{y})
    &=
      \sum_{x,y \in \ZZ^{d}} \Jump (x-y) 
      \tilde{f}_{x}(\tilde{h}_{x} - \tilde{h}_{y}) 
      \nonumber\\
    \label{e:Lambda-dirichlet-form3}
    &=
    \sum_{x,y \in \ZZ^{d}} \Jump (x-y) 
    \tilde{f}_{x}(- \tilde{h}_{y})  .
  \end{align}
  For the final equality we used the zero row sum property
  $\sum_{y}\Jump (x-y)=0$.  Recall from
  \eqref{e:Delta-infty-matrix-def} that $\Jump (x-y) = \DI_{x,y}$ and
  that $\DL_{x,y}$ is the restriction of $\DI_{x,y}$ to
  $\Lambda$. Therefore, in \eqref{e:Lambda-dirichlet-form3} we insert
  \begin{equation}
    \sum_{y \in \ZZ^{d}}\Jump (x-y)
    (-\tilde{h}_{y})
    =
    \sum_{y \in \ZZ^{d}} (-\DI_{x,y})\tilde{h}_y
    =
    \sum_{y \in \Lambda}(-\DI_{x,y}) h_y
    =
    (- \DL h)_{x}
  \end{equation}
  which proves \eqref{e:Lambda-dirichlet-form2} and hence completes
  the proof that $\DL$ is invertible.

  Next we prove \eqref{e:free-green-1}.  By definition
\begin{equation}
   \GL (a,b)= \int_{0}^\infty dt\,
    \Pprod_{a}(X_{t} = b)
    = \int_{0}^\infty dt\, (e^{t\Delta_*})_{a,b} 
    =\int_{0}^\infty dt\, (e^{t\DL})_{a,b}\label{restricted_mat}.  
\end{equation}
where the last equality holds since
$(\Delta_*^k)_{a,b}=((\DL)^k)_{a,b}$, and where for a square matrix
$A$, $e^{tA}$ denotes the matrix exponential
$\sum_{k=0}^{\infty}\frac{t^{k}}{k!}A^{k}$.  The right hand side of
\eqref{restricted_mat} is $(-\DL)^{-1}_{a,b}$ as desired: since $-\DL$
is real symmetric with positive eigenvalues, this follows by
diagonalizing and integrating.
\end{proof}

\begin{proof}[Proof of Lemma \ref{lem:free-green-inf}]
  \label{proof:lem:free-green-inf}
  Recall the definition of $\dgreens_{z}(x)$ from
  \eqref{e:dgreensz-def}, and let
  $\dgreens(x) = \dgreens_{\hJ^{-1}}(x)$. 
  Let $T_n$ denote the time of the $n$th jump of $\ZX$ and $T_0=0$.
  By \eqref{e:Sinf}
  \begin{align}
    \nonumber
    \GI(x)
    &= E_0\Big[\sum_{n=0}^\infty (T_{n+1}-T_n)\indic{\ZX_{T_n}=x}\Big]\\
    \label{mmm1}
    &=\sum_{n=0}^\infty E_0[T_{n+1}-T_n]E_0\Big[\indic{\ZX_{T_n}=x}\Big]\\
    \label{e:mmm2}
    &=\hJ^{-1}\sum_{n=0}^\infty
            E_0\Big[\indic{\ZX_{T_n}=x}\Big]=\hJ^{-1}\dgreens(x).
  \end{align}
  By using the strong Markov property to restart at time $T_{1}$
    (see \cite[Section 4.3]{lawler_limic_2010}), we have
  \begin{equation}
    \dgreens(x) = \indic{x=0} + \sum_{y}
      \hJ^{-1}\jumprate(y)
    \dgreens(x-y),
  \end{equation}
  and by \eqref{e:mmm2} this can be rewritten as
  \begin{equation}
    \hJ\GI(x)=\indic{x=0}+ \sum_{y}
    \jumprate(y) 
    \GI (x-y).
  \end{equation}
  Collecting terms gives the first claim.  To verify
  \eqref{e:free-green-inf-2}, we use \eqref{mmm1}:
  \begin{equation*}
    \GI (x)=\sum_{n=0}^\infty E_0[T_{n+1}-T_n]P_0(\ZX_{T_n}=x)=\sum_{n=0}^\infty \hJ^{-1}
             (\hJ^{-1}\jumprate)^{*n}(x).\qedhere
  \end{equation*}
\end{proof}

\subsection{The Markov Property}
\label{sec:markov}

Recall that $\Omega_{1}$ is the space of paths defined in
  \Cref{sec:conv-techn-choice} and $(\Omega_{1},\mc{F},P_x)$ is the
  probability space for random walk $X_{t}$ such that $X_{0}=x$.  For
$s\ge 0$ define the $\mc{F}$-measurable map
$ \theta_{s}\colon \Omega_1 \rightarrow \Omega_1 $ by
$\theta_s((x_t)_{t\ge 0})=(x_{s+t})_{t\ge 0}$.  The following is a
standard formulation of the Markov property.

\begin{proposition}
  \label{prop:standard-markov}
  Let $H\colon \Omega_1 \to \RR$ be $\mathcal{F}$-measurable and
  integrable with respect to $E_a$ for each $a\in \ZZ^d$, and let
  $
    h (x)
    =
    E_{x}\big[
    H
    \big] .$
  Then for every $x\in \ZZ^d$ and $s\ge 0$,
  \begin{equation}
    E_x\big[
    H\circ\theta_{s} 
    \big|\mc{F}_{s}\big]
    =
    h (X_{s}), \quad P_x\text{-a.s.}
  \end{equation}
\end{proposition}

\subsubsection{The Markov property as used to obtain
  \texorpdfstring{\eqref{e:prop:lace-expansion03}}{Proposition Lace
    exansion03}}
To justify this application of the Markov property, for $\ell>0$ and
$b\in \ZZ^d$ let
$H_{\ell,b}\bydef \rptn_{0,\ell}\;\indic{X_{\ell} = b}$ and
$ h (\ell,y,b) \bydef E_{y} \cb{H_{\ell,b}}.$ Then, by
\Cref{prop:standard-markov},
\begin{equation}
    E_{a} \big[ \rptn_{s,s+\ell}\;
    \indic{X_{s+\ell} = b}\big|\mc{F}_{s}\big]
    =
    h (\ell,X_{s},b), \quad \text{ $P_a$-a.s.}
\end{equation}

\subsubsection{The Markov property as used to obtain \texorpdfstring{\eqref{e:prop:lace-expansion1}}{prop:lace-expansion} }

To justify this application of the Markov property, for
$\lace \in \laces_{m} (0)$ let $H_{\lace,b}\bydef\weight (\lace)G_{\V{0}}(X_{s_{m}'},b)$, and 
$    f (x,b,\lace)
    \bydef
    E_{x}
    \cb{H_{\lace,b}}.$
Then, by Proposition~\ref{prop:standard-markov},
\begin{equation}
    E_{a}\,\big[
    \weight (\lace +s)\; G_{\V{0}}(X_{s_{m}'+s},b)
    \big|\mc{F}_{s}
    \big]
    =
    f (X_{s},b,\lace).
\end{equation}

\subsubsection{The Markov property in the proof of
  \Cref{lem:induct}}

Fix $b \in \Lambda$. We assume that the function
  $\GrL: (\V{r},x) \mapsto \GrL _{\V{r}}(x,b)$ is defined for
  $(\V{r},x) \in [0,\infty)^{\Lambda}\times \Lambda$ and is bounded on
  this domain. Let $\mathcal{S}$ be a $(\mc{F}_t)_{t \ge 0}$-stopping
  time which is $P_{a}$-a.s. finite.  By definition the random
  variable $\G{I}^{\sssL}(X_{\mc{S}},b)$ is the composition of $\GrL$
  with $(\V{\tau_{I}},X_{\mc{S}})$.

\begin{lemma}
  \label{lem:clemma}  For
  $b\in \Lambda$, and any $\mc{F}_{\mc{S}}$-measureable Borel set
  $I \subset [0,\mc{S}]$,
\begin{equation}
    \label{eq:clemma}
    \G{I}^{\sssL}(X_{\mc{S}},b) = E_{a} \Big[
    \int_{[\mc{S},\infty)} d\ell\,
    \;
    \frac{Z_{\V{\tau}_{I} + \V{\tau}_{[\mc{S},\ell]}}}{\Z{I}}
    \indic{X_{\ell}=b} \big\vert \mc{F}_{\mc{S}}
    \Big]\quad \text{$P_{a}$-a.s.}
\end{equation}
\end{lemma}

The proof of \Cref{lem:clemma} requires two preparatory ingredients
stated for a general probability space $(\Omega,\mc{F},\Pprod)$.

\begin{lemma}
\label{lem:funnyE}
Let $W:\Omega \to \RR$ be integrable with respect to $\Pprod$,
$A\in \mc{F}$ with $\Pprod(A)>0$ and $\Eprod_A=\Eprod[\cdot|A]$.  Then
\begin{equation}
  \1_A\Eprod[W|\mc{F}]=\1_A\Eprod_A[W|\mc{F}], \qquad \Pprod\textrm{-a.s.}
\end{equation}
\end{lemma}

\begin{proof}
  Let $L$ and $R$ denote the left- and right-hand sides, respectively.
  Then both $L$ and $R$ are $\mc{F}$-measurable. Let $B\in \mc{F}$.
  Then $\Eprod[\1_B(L-R)]=0$ since
  \begin{align}
    \Eprod[\1_B\1_A\Eprod_A[W|\mc{F}]]
    &=\Eprod_A[\1_B\1_A\Eprod_A[W|\mc{F}]]\Pprod(A)\\
    &=\Eprod_A[\1_B\1_A W]\Pprod(A)\\
    &=\Eprod[\1_B\1_AW]=\Eprod[\1_B\1_A\Eprod[W|\mc{F}]]. 
  \end{align}
  Taking $B_1=\{L>R\}\in \mc{F}$ and $B_2=\{L<R\}\in \mc{F}$ completes
  the proof.
\end{proof}

The following lemma is a standard result in the case that
$\mathcal{G}=\sigma(W_1)$, see \cite[Example 5.1.5]{Durrett2010}.
Since we have been unable to find this particular formulation in the
literature, we give a proof below.
\begin{lemma}
  \label{lem:conditional} 
  For measurable spaces $(S_1,\mc{S}_1)$ and $(S_2,\mc{S}_2)$ let
  $W_1\colon\Omega \ra S_1$ and $W_2\colon\Omega\ra S_2$ be
  measurable, and let $f\colon S_1\times S_2\ra \RR$ be
  Borel-measurable on the corresponding product space
  $(S_1\times S_2,\mc{S})$ and either bounded or non-negative and such
  that $\Eprod[f(W_1,W_2)]$ is finite.  Define $h\colon S_1 \ra \RR$
  by
    \begin{equation}
      h(w_1)=\Eprod[f(w_1,W_2)].
    \end{equation}
    If $W_2$ is independent of $\mc{G}\subset \mc{F}$ and $W_1$ is
    $\mc{G}$-measurable then
    \begin{equation}
      \label{e:conditional}
      h(W_1)
      =
      \Eprod[f(W_1,W_2)|\mc{G}],
      \quad \textrm{a.s.}
    \end{equation}
  \end{lemma}
\begin{proof}
  If $f(w_1,w_2)=f_1(w_1)f_2(w_2)$, where $f_1$ and $f_2$ are bounded
  and $\mc{B}(\RR)$-measurable then $h(w_1)=f_1(w_1)\Eprod[f_2(W_2)]$.
  Then $h(W_1)=f_1(W_1)\Eprod[f_2(W_2)]$ (is $\mc{G}$ measurable) and
  by independence, almost surely
  \begin{equation}
    f_1(W_1)\Eprod[f_2(W_2)]
    =
    \Eprod[f_1(W_1)f_2(W_2)| \mc{G}].
  \end{equation}
  In particular this holds for any $f$ of the form
  $f(w_1,w_2)=\1_{A_1\times A_2}\equiv\indic{w_1\in A_1}\indic{w_2\in
    A_2}$, where $A_i \in \mc{S}_i$.  Therefore by linearity of
  expectation it also holds for indicators of finite disjoint unions
  of events of the form $A_1\times A_2$.

  Let $\mc{A}\subset \mc{S}$ denote the collection of events for which
  the claim of the lemma holds with $f=\1_A$.  Then $\mc{A}$ contains
  the field of finite disjoint unions of events of the form
  $A_1\times A_2$, and by dominated convergence $\mc{A}$ is a monotone
  class.  Thus by the Monotone Class Theorem $\mc{A}=\mc{S}$, and
  hence by linearity the claim holds for all simple
  functions $f$.  
  
  For non-negative $f$ such that $f(W_1,W_2)$ is integrable we can
  take non-negative simple functions $f_n$ increasing to $f$
  pointwise.  Let $h_n(w_1)=\Eprod[f_n(w_1,W_2)]$.  Then
  $h_n(w_1)\uparrow \Eprod[f(w_1,W_2)]=:h(w_1)$ pointwise by monotone
  convergence.  Next, by the result for simple functions we have for
  each $n$
  \begin{equation}
    h_n(W_1)=\Eprod[f_n(W_1,W_2)| \mc{G}].
  \end{equation}
  The right hand side increases to $\Eprod[f(W_1,W_2)| \mc{G}]$ by
  monotone convergence and the left hand side increases to $h(W_1)$ by
  the above pointwise convergence.  This proves the result for
  non-negative $f$ such that $f(W_1,W_2)$ is integrable.

  The claim for bounded measurable $f$ follows by considering the
  positive and negative parts of $f$.
\end{proof}

\begin{proof}[Proof of \Cref{lem:clemma}]
  For $\V{r} \in [0,\infty)^\Lambda$ and a path $\tilde{y}$ in
  $\Omega_{1}$ let
  \begin{equation}
    f(\V{r},\tilde{y})
    \bydef
    \int_{[0,\infty)}d\ell'\,
    \frac{Z_{\V{r}+
        \tilde{\V{\tau}}_{[0,\ell']}(\tilde{y})
      }}{Z_{\V{r}}}
    \indic{\tilde{y}_{\ell'}=b },
  \end{equation}
  where $\tilde{\V{\tau}}_{[u,v]}(\tilde{y})$ denotes the vector of
  local times of the path $\tilde{y}$ in the interval $[u,v]$.
  
  Let $\theta_{\mc{S}}$ be the time shift
  $\theta_{s}:\Omega_{1} \rightarrow \Omega_{1}$ defined at the
  beginning of this subsection with $s=\mc{S}$. Since $\mc{S}$ is
  finite $\theta_{\mc{S}}$ is defined for $P_{a}$ almost all paths.
  Let $\tilde{X}=\theta_{\mc{S}}(X)$. Then
  $\V{\tau}_{[\mc{S},\ell]}=\tilde{\V{\tau}}_{[0,\ell-\mc{S}]}$, where
  $\tilde{\V{\tau}}$ is the local time of $\tilde{X}$, and
  $\indic{X_{\ell}=b}=\indic{\tilde{X}_{\ell-\mc{S}}=b}$.  Equalities
  of random variables in this proof are $P_a$-a.s.  Integrals and
  expectations are applied only to nonnegative functions and
  integrability is eventually implied by the assumed boundedness of
  $\G{I}^{\sssL}(X_{\mc{S}},b)$.  Let $R$ be the right hand side of
  \eqref{eq:clemma}. Then
  \begin{align}
    R
    &=
      E_{a} \Big[
      \int_{[\mc{S},\infty)} d\ell\,
      \;
      \frac{Z_{\V{\tau}_{I} + \tilde{\V{\tau}}_{[0,\ell-\mc{S}]}}}{\Z{I}}
      \indic{\tilde{X}_{\ell-\mc{S}}=b} \big\vert \mc{F}_{\mc{S}}
      \Big]
      \nonumber\\
    &=
      E_{a} \Big[
      f(\V{\tau_{I}},\tilde{X})
      \big\vert \mc{F}_{\mc{S}}
      \Big]
      =
      \sum_{x\in \Lambda} \indic{X_{\mc{S}}=x}
      E_{a} \Big[
      f(\V{\tau_{I}},\tilde{X})
      \big\vert \mc{F}_{\mc{S}}
      \Big].
    \label{stepp1}
  \end{align}
  The first equality is obtained by the change of variables
  $\ell=\ell'+\mc{S}$ followed by inserting the definition of
  $f(\V{r},\tilde{y})$ and the second equality is obtained by
  inserting $1=\sum_{x\in \Lambda} \indic{X_{\mc{S}}=x}$ under the
  conditional expectation. The indicator function can be moved outside
  since the indicator function is $\mc{F}_{\mc{S}}$ measurable. If for
  some $x$, $P_a(X_\mc{S}=x)=0$, then the corresponding contribution
  to the above sum is 0 $P_a$-a.s.  Otherwise, let
  $\tilde{P}_x(\cdot)\bydef \Pa( \cdot |\tilde{X}_0=x)$. By Lemma
  \ref{lem:funnyE} \eqref{stepp1} becomes
  \begin{align}
    R
    &=
      \sum_{x\in \Lambda} \indic{X_{\mc{S}}=x}
      \tilde{E}_{x} \Big[
      f(\V{\tau_{I}},\tilde{X})
      \big\vert \mc{F}_{\mc{S}}
      \Big].
    \label{stepp2}
  \end{align}
  Let $h_{x}(\V{r}) = \tilde{E}_{x} \Big[f(\V{r},\tilde{X})\Big]$.  By
  the strong Markov property $\tilde{X}$ is a random walk with
  $\tilde{X}_{0}=x$ that is independent of $\mc{F}_{\mc{S}}$.  By
  \Cref{lem:conditional}
  $\tilde{E}_{x} \big[ f(\V{\tau_{I}},\tilde{X}) \big\vert
  \mc{F}_{\mc{S}} \big] = h_{x}(\V{\tau}_{I})$, but this is a
  $\tilde{P}_x$-a.s.equality.  However it also holds
  $P_a$-a.s. because
  $P_{a} = \sum_{x\in\Lambda}P_{a}(\tilde{X}_{0}=x)
  \tilde{P}_x(\cdot)$ is a countable sum. Therefore \eqref{stepp2}
  becomes
  \begin{equation}
    R
    =
    \sum_{x\in \Lambda} \indic{X_{\mc{S}}=x} h_{x}(\V{\tau}_{I})
    =
    \sum_{x\in \Lambda} \indic{X_{\mc{S}}=x}     
    \GrL _{\V{\tau}_{I}}(x,b)
    =
    \GrL _{\V{\tau}_{I}}(X_{\mc{S}},b)
    \qquad
    \text{$P_a$-a.s.}
  \end{equation}
  as desired. The second equality is obtained from \eqref{e:GF} by
  interchanging the integral over $\ell'$ in the definition of
  $f(\V{\tau_{I}},\tilde{X})$ with the expectation.
\end{proof}

\section*{Acknowledgements}
\label{sec:acknowledgements} 

The authors thank Gordon Slade for various helpful comments regarding
a previous version of this manuscript.  TH thanks Gady Kozma for a
helpful discussion. TH and DB thank the Isaac Newton Institute for
Mathematical Sciences for support and hospitality during the programme
``Scaling limits, rough paths, quantum field theory'' when work on
this paper was undertaken. DB is grateful for partial support by the
Simons Foundation during the same programme and 
thanks the University of Virginia for hospitality during 2016 and 2017
during this collaboration.  This work was supported by EPSRC grant
LNAG/036, RG91310. TH held positions at UC Berkeley and the University
of Bristol while this project was being carried out, and was supported
by EPSRC grant EP/P003656/1 and an NSERC PDF. MH is supported by
Future Fellowship FT160100166, from the Australian Research Council.



\bibliographystyle{plain}
\bibliography{refs}





                                



\end{document}